\documentclass[letterpaper,12pt,titlepage,final,oneside]{book}

\usepackage{amsmath,amssymb,amstext,amsthm} 
\usepackage{amsfonts}
\usepackage{bbm}
\usepackage{caption}
\usepackage{cite}
\usepackage{commath}
\usepackage{enumerate}
\usepackage{epsfig}
\usepackage[latin1]{inputenc}
\usepackage{float}
\usepackage{graphicx}
\usepackage{geometry}
\usepackage{leftidx}
\usepackage{mathrsfs}
\usepackage{mathtools}
\usepackage{titlesec}
\usepackage{undertilde}
\usepackage{verbatim}
\usepackage[all]{xy}
\usepackage{mathtools}
\usepackage{enumerate}
\usepackage{verbatim}
\usepackage{float}
\usepackage{hyperref}
\usepackage{titlesec}
\usepackage[subsection]{algorithm}
\usepackage{algorithmic}
\usepackage{bbm}
\usepackage{cite}
\usepackage[inline]{enumitem}
\usepackage{float}
\usepackage[rightcaption]{sidecap}
\usepackage{undertilde}
\usepackage{verbatim}

\usepackage{hyperref} 
\allowdisplaybreaks

\DeclareMathOperator*{\argmax}{arg\,max}

\DeclareMathOperator{\sgn}{sgn}
\newcommand{\defeq}{\vcentcolon=}
\newcommand{\backdefeq}{=\vcentcolon}
\newcommand{\ve}{\varepsilon}
\newcommand{\grid}{G}

\newcommand{\Lpm}{\Lambda^{\pm}}
\newcommand{\Lp}{\Lambda^{+}}
\newcommand{\Lm}{\Lambda^{-}}
\newcommand{\lpm}{\lambda^{\pm}}
\newcommand{\lp}{\lambda^{+}}
\newcommand{\lm}{\lambda^{-}}
\newcommand{\delpm}{\delta^{\pm}}
\newcommand{\delp}{\delta^{+}}
\newcommand{\delm}{\delta^{-}}

\newcommand{\Npm}{N^{\pm}}
\newcommand{\Nm}{N^{-}}
\newcommand{\Np}{N^{+}}
\newcommand{\Hp}{H^{+}}
\newcommand{\Hm}{H^{-}}
\newcommand{\Hpm}{H^{\pm}}

\newcommand{\hpm}{h^{\pm}}

\newcommand{\apm}{a^{\pm}}

\newcommand{\bpm}{b^{\pm}}

\newcommand{\cpm}{c^{\pm}}

\newcommand{\pii}{\Pi^i}
\newcommand{\pij}{\Pi^j}
\newcommand{\piia}{\Pi^{\alpha,i}}
\newcommand{\pija}{\Pi^{\alpha,j}}

\newcommand{\pia}{\Pi^\alpha}
\newcommand{\pit}{\Pi^{\alpha,t,n,\pi}}
\newcommand{\Yt}{Y^{t,i}}

\newcommand{\hlpmt}{\widehat{\lpm}^{\alpha,t,n,\pi}}

\newcommand{\hlptu}{\widehat{\lp_u}^{\alpha,t,n,\pi}}
\newcommand{\hlmtu}{\widehat{\lm_u}^{\alpha,t,n,\pi}}

\newcommand{\Xdel}{X^{\delm,\delp}}

\newcommand{\Spm}{S^{\pm}}
\newcommand{\Sp}{S^{+}}
\newcommand{\Sm}{S^{-}}
\newcommand{\fipm}{f^{\pm}_i}
\newcommand{\Lipm}{\Lambda_i^{\pm}}
\newcommand{\Lip}{\Lambda_i^{+}}
\newcommand{\Lim}{\Lambda_i^{-}}
\newcommand{\Ljpm}{\Lambda_j^{\pm}}
\newcommand{\Ljp}{\Lambda_j^{+}}
\newcommand{\Ljm}{\Lambda_j^{-}}
\newcommand{\lmax}{\lambda^{*}}

\newcommand{\delmax}{\delta^{*}}
\newcommand{\delmin}{\delta_{*}}
\newcommand{\Nmax}{N^*}

\newcommand{\mPa}{\mathbb P^{\alpha}}
\newcommand{\mPt}{\mathbb P^{\alpha,t,n,\pi}}
\newcommand{\mPtfull}{\mathbb P^{\alpha,t,n,i}}
\newcommand{\mQa}{\mathbb Q^{\alpha}}
\newcommand{\mEdel}{\mathbb E^{\delm,\delp}}
\newcommand{\mEa}{\mathbb E^{\alpha}}
\newcommand{\mEt}{\mathbb E^{\alpha,t,n,\pi}}

\newcommand{\Dpm}{d^{\pm}}
\newcommand{\Dp}{d^{+}}
\newcommand{\Dm}{d^{-}}

\newcommand{\Xt}{X^{\alpha,t,x,s}}
\newcommand{\St}{S^{t,s}}
\newcommand{\Nt}{N^{t,n}}

\newtheorem{assumptions}{Assumptions}[subsection]

\newtheorem{theorem}[assumptions]{Theorem}
\newtheorem{lemma}[assumptions]{Lemma}
\newtheorem{corollary}[assumptions]{Corollary}
\newtheorem{proposition}[assumptions]{Proposition}

\theoremstyle{definition}
\newtheorem{definition}[assumptions]{Definition}
\newtheorem{definitions}[assumptions]{Definitions}

\theoremstyle{remark}
\newtheorem{example}[assumptions]{Example}
\newtheorem{examples}[assumptions]{Examples}
\newtheorem{remark}[assumptions]{Remark}
\newtheorem*{remark*}{Remark}

\newtheorem{notation}[assumptions]{Notation}
\newtheorem*{notation*}{Notation}
\newtheorem{convention}[assumptions]{Convention}
\newtheorem*{convention*}{Convention}

\numberwithin{equation}{subsection}
\allowdisplaybreaks

\begin{document}


\pagestyle{empty}
\pagenumbering{roman}

\begin{titlepage}
        \begin{center}
        \vspace*{1.0cm}

        \Huge
        {Optimal market making under partial information\\ and numerical methods for impulse control games\\ with applications}

        \vspace*{1.0cm}

        \normalsize
        by \\

        \vspace*{1.0cm}

        \Large
        Diego Zabaljauregui\\
				
				\vspace*{1.0cm}
				
				\Large
				Department of Statistics\\London School of Economics and Political Science

        \vspace*{3.0cm}

        \normalsize
        A thesis \\
        submitted for the degree of\\
        Doctor of Philosophy\\
				
        \vspace*{2.0cm}

        London, United Kingdom, 2019 \\

        \vspace*{1.0cm}

        \copyright\ Diego Zabaljauregui 2019 \\
        \end{center}
\end{titlepage}


\clearpage
\vspace*{3.0cm}
\begin{center}
\large
\textbf{Declaration}
\end{center}
\bigskip

\noindent
I certify that the thesis I have presented for examination for the PhD degree of the London
School of Economics and Political Science is solely my own work other than where I have
clearly indicated that it is the work of others (in which case the extent of any work carried
out jointly by me and any other person is clearly identified in it). The copyright of this thesis
rests with the author. Quotation from it is permitted, provided that full acknowledgement is
made. This thesis may not be reproduced without my prior written consent. I warrant that
this authorization does not, to the best of my belief, infringe the rights of any third party. I
declare that my thesis consists of less than 100,000 words.

I confirm that Chapter 1 was jointly co-authored with my supervisor, Professor Luciano Campi, and I contributed 70\% of this work.

I confirm that Chapter 2 was jointly co-authored with Professors Ren\'e A{\"\i}d, Francisco Bernal, Mohamed Mnif and J.P. Zubelli, and I contributed 40\% of this work.
\clearpage


\clearpage
\vspace*{1.0cm}
\begin{center}
\large
\textbf{Acknowledgments and dedication}
\end{center}
\bigskip

\noindent
I would like to thank my supervisor, my department, my family and friends; with a special dedication to my wife and parents for their unconditional support.

\clearpage

\vspace*{1.0cm}
\begin{center}
\large
\textbf{Abstract}
\end{center}

The topics treated in this thesis are inherently two-fold. The first part considers the problem of a market maker who wants to optimally set bid/ask quotes over a finite time horizon, to maximize her expected utility. The intensities of the orders she receives depend not only on the spreads she quotes, but also on unobservable factors modelled by a hidden Markov chain. This stochastic control problem under partial information is solved by means of stochastic filtering, control and piecewise-deterministic Markov processes theory. The value function is characterized as the unique continuous viscosity solution of its dynamic programming equation. Afterwards, the analogous full information problem is solved and results are compared numerically through a concrete example.  The optimal full information spreads are shown to be biased when the exact market regime is unknown, as the market maker needs to adjust for additional regime uncertainty in terms of P\&L sensitivity and observable order flow volatility.

The second part deals with numerically solving nonzero-sum stochastic differential games with impulse controls. These offer a realistic and far-reaching modelling framework for applications within finance, energy markets and other areas, but the difficulty in solving such problems has hindered their proliferation. Semi-analytical approaches make strong assumptions pertaining very particular cases. To the author's best knowledge, there are no numerical methods available in the literature. A policy-iteration-type solver is proposed to solve an underlying system of quasi-variational inequalities, and it is validated numerically with reassuring results. In particular, it is observed that the algorithm does not enjoy global convergence and a heuristic methodology is proposed to construct initial guesses.

Eventually, the focus is put on games with a symmetric structure and a substantially improved version of the former algorithm is put forward. A rigorous convergence analysis is undertaken with natural assumptions on the players strategies, which admit graph-theoretic interpretations in the context of weakly chained diagonally dominant matrices. A provably convergent single-player impulse control solver is also provided. The main algorithm is used to compute with high precision equilibrium payoffs and Nash equilibria of otherwise too challenging problems, and even some for which results go beyond the scope of the currently available theory.

\clearpage


\newpage
\phantomsection
\renewcommand\contentsname{Table of Contents}
\begingroup
\let\cleardoublepage\clearpage
\tableofcontents
\endgroup
\clearpage
\phantomsection


\begingroup
\let\cleardoublepage\clearpage
\listoftables
\endgroup
\clearpage
\phantomsection


\begingroup
\let\cleardoublepage\clearpage
\listoffigures
\endgroup
\clearpage
\phantomsection	


\pagenumbering{arabic}
\pagestyle{headings}


\setcounter{chapter}{0}
\chapter*{Introduction}
\label{c:introduction}
\addcontentsline{toc}{chapter}{Introduction}

This thesis is divided in two main parts. The first one (Chapter \ref{c:1}) deals with the well-known and very relevant problem of \textit{optimal market making} under a novel and realistic framework of \textit{partial information}. The second one (Chapters \ref{c:2} and \ref{c:3}), focuses on the development of numerical methods for a very general class of far-reaching models known as \textit{nonzero-sum stochastic differential games with impulse controls}. From a mathematical viewpoint, \textit{stochastic control} (in a broad sense) is the unifying underlying topic, spanning from classical `continuous' controls, to impulse controls and impulse games; and covering the different levels of the problems: from the applications and motivation to the most technical aspects, from the theoretical analysis to the numerical methods and their implementation. 

However, the problems, models and techniques used vary so widely, that a unique and comprehensive introduction is not only difficult, but likely counterproductive. For this reason, each chapter starts with a detailed introduction of its own. This section is therefore intended only to briefly motivate the problems to be studied, comment on the techniques used, and list the main resulting scientific contributions. 

The optimal market making problem (Chapter \ref{c:1}) consists on determining how a liquidity provider for a given asset should (continuously) choose to set her bid/ask quotes if she wishes to maximize her expected utility. In doing so, she faces a complicated problem with both dynamic and static components, and several risks ranging from adverse price movements, exposure, execution costs and more. An increasingly popular approach, both academically and in practice, is the stochastic optimal control framework proposed by Avellaneda and Stoikov \cite{AS}, where the liquidity taken from the market maker decays as a function of the spreads she quotes. Different variants of the first model have been proposed and thoroughly studied since, both for market making and optimal liquidation \cite{AS,BL,CDJ,CJ,CJR,FL1,FL2,G,GL,GLFT}. 

Motivated by empirical evidence \cite{CJ0}, we consider a \textit{hidden Markov-chain} model that unifies and generalizes features of all the previously mentioned ones, and admits the possibility that liquidity may also vary due to unobservable market conditions. We solve this combined control and partial information problem in full, using stochastic control, \textit{stochastic filtering} and \textit{piecewise-deterministic Markov processes} (PDMPs) theory \cite{DF}. The analogous and idealized problem in which the market maker can observe the market state is also solved, and results are compared numerically through a concrete example. Our main contributions are:
\begin{itemize}
\item A general and realistic Avellaneda--Stoikov-type model for optimal market making, incorporating for the first time the liquidity dependence on hidden market conditions and allowing for very general intensity (liquidity) shapes.
\item A novel approach, within this framework, to characterize the value function of the market maker. The complexity of the dynamic programming equation in our model prevents us from using the most standard technique: proposing an ansatz and
proving it valid by a verification theorem. Instead, we explicitly find the ansatz decomposition in terms of the value function of a new, diffusion-free, problem. This result is interesting in itself, as it opens the door to probabilistic and PDMPs numerical techniques when the dimension of the problem renders partial differential equations (PDEs) methods prohibitive.
\item Ultimately, we characterize the `reduced' value function as the unique continuous \textit{viscosity solution} \cite{FS} of its formally derived equation. 
\item Having solved the idealized full information problem with a general verification theorem, we show that the corresponding optimal strategies become suboptimal when the market regime is unknown, as the market maker needs to manage a higher regime risk. We interpret the adjustment needed in terms of observable order flow volatility and sensitivity of the expected profit to observable regime changes. Ultimately we see and interpret why the bias of the full information optimal strategies becomes higher, the longer the waiting time in between orders.
\end{itemize}

Chapters \ref{c:2} and \ref{c:3} consider very general models known as nonzero-sum stochastic impulse control games \cite{ABCCV,BCG}, in which players interact by shifting a certain continuous-time stochastic process at discrete (usually random) points in time. Such models have motivation and wide applicability within finance, energy markets, insurance and many other fields. Unfortunately, they are utterly challenging, what has rendered their study far too underdeveloped. In light of the evident need for numerical methods, and the nonexistence of any such one, Chapter \ref{c:2}:
\begin{itemize}
\item puts forward the very first iterative algorithm to solve nonzero-sum stochastic impulse games. The method is of \textit{policy-iteration-type} \cite{AF,BMZ} and solves an underlying system of differential \textit{quasi-variational inequalities}, involving highly nonlinear, nonlocal and noncontractive operators, as well as non-smooth solutions.  
\end{itemize}
The proposed algorithm is admittedly heuristic and no convergence analysis is carried out in this chapter. Instead, it is validated numerically over a series of experiments, considering different games over both fixed and refining grids, performing simulations and comparing results with the only (almost) fully analytically tractable game in the literature \cite{ABCCV}. The results suggest that the algorithm can effectively be used to gain insight into general nonzero-sum stochastic impulse games.

In Chapter \ref{c:3}, we specialize our study by putting the focus on \textit{symmetric} games. This subclass is broad enough to include plenty of interesting applications, such as competition between central banks \cite{ABCCV,AF,CZ,J,MO}, cash management problems \cite{BCG} and the generalization of many impulse control problems to the two-player instance. 

For this type of games, an iterative algorithm which substantially improves the one in Chapter \ref{c:2} is presented and rigorously studied, using techniques from \textit{policy iteration for impulse control} \cite{AF}, \textit{fixed-point policy-iteration} \cite{HFL1} and graph-theoretic notions relating to \textit{weakly chained diagonally dominant (WCDD) matrices} and their recently introduced matrix sequence counterpart \cite{A}. The main contributions of this chapter are:
\begin{itemize}
\item An iterative algorithm for the effective solution of symmetric nonzero stochastic impulse games, which improves the general one in terms of simplicity, efficiency, precision and stability.
\item Providing the missing convergence analysis, which is performed under very natural assumptions on the players strategies and interpreted in terms of WCDD matrices and their extension in \cite{A}. The latter is applied to impulse control for the first time, to the author's best knowledge. 
\item Proof of properties of contractiveness, boundedness of iterates and convergence to solutions, as well as sufficient conditions for convergence.
\item As a by-product, a novel impulse control solver is also provided, and its convergence is proved. 
\item An extensive numerical validation is carried out, considering different performance metrics and addressing practical matters of implementation. 
\item Ultimately, equilibrium payoffs and Nash equilibria are computed with high precision, for games seemingly too challenging for the available analytical approaches, and even some for which results go beyond the scope of the currently available theory (including discontinuous impulses and very irregular discontinuous payoffs). 
The latter also motivate further research into this field, with an emphasis on the need for a viscosity solutions framework.
\end{itemize}


\chapter[Optimal market making with partial information]{Optimal market making under partial information with general intensities}
\chaptermark{Optimal market making with partial information}
\label{c:1}

\phantomsection
\section*{Introduction}
\addcontentsline{toc}{section}{Introduction}
Given a financial market, a market maker (MM) can be understood as someone who provides liquidity for a certain asset. That is, she (almost) continuously posts bid/ask quotes for the asset, in the hope to profit from the bid-ask spread. In choosing how to do so, the MM faces a complicated problem on several levels; namely: the instantaneous margin/volume trade-off (the further away she quotes from the `fair price', the less she gets executed, and vice-versa), adverse price movements and inventory risk (exposure), execution costs, and many others. 

An increasingly popular mathematical approach to the MM problem, both in academia and in practice, is by means of stochastic optimal control. In particular, a line of research has focused on the modelling framework proposed by Avellaneda and Stoikov \cite{AS} (rooted in turn in \cite{HS}). Although widely motivated in the literature by order-driven markets such as equity markets, the shape of the limit order book is not explicitly taken into account. Thus, the framework is more easily understood and applied to over-the-counter (OTC) quote-driven markets such as the foreign exchange (FX) market, and we therefore choose to present it in this setting.

In this framework, the MM gives firm bid-ask quotes during a finite time interval by choosing bid/ask spreads with respect to a certain \textit{reference price}.\footnote{Also referred to by some authors as \textit{micro-price} \cite{CJP} or \textit{efficient price} \cite{DRR} in the martingale case.} Depending on the market, this price could be for example an aggregated mid-price or a dealer-to-dealer price, and it is frequently assumed to behave as an arithmetic Brownian motion. To explicitly model the margin/volume trade-off, the probability that the MM gets executed decays as a function of the corresponding spread. More precisely, it is assumed that the MM receives market orders according to counting processes of stochastic intensity. Most typically in the mathematical literature, intensities are assumed to decay exponentially on the spreads (mainly for tractability reasons). The goal of the MM is to find a strategy that allows her to maximize her expected terminal utility, which is taken as a constant absolute risk aversion (CARA) utility. The problem is then translated into a deterministic one: solving the associated Hamilton--Jacobi--Bellman (HJB) equation, which is a partial-integro differential equation (PIDE) for the MM's value function, and ultimately retrieving the optimal strategy in feedback form. 

Within the described framework, a lot variants have been put forward and extensively studied. In \cite{BL, GLFT0}, Bayraktar and Ludkovski, and Gu{\'e}ant, Lehalle and Fernandez-Tapia, apply a one-trading-side version to optimal liquidation in the risk-neutral and risk averse contexts, respectively. The introduction of a constraint on the inventory in \cite{GLFT}, allowed the authors to rigorously solve the original problem of \cite{AS} with exponential intensities, by means of a verification theorem. This constraint has been widely used moving forward, with the exception of \cite{FL1, FL2}, where strategies are derived without a verification theorem.

Gu{\'e}ant and Lehalle continued to actively contribute to the area. In \cite{GL}, they revisit the risk averse optimal liquidation problem for general intensities satisfying a certain ordinary differential inequality, and in \cite{G} the same is done for market making.\footnote{The latter paper also considers the multi-asset case.} This assumption had been firstly introduced in \cite[Sect.5.3]{BL} under risk-neutrality.

Other prolific contributors in the field are Cartea, Jaimungal and their coauthors \cite{CJ, CJP, CJR}, who introduced a quadratic running penalty on the inventory to manage the `accumulated' inventory risk for an otherwise risk-neutral MM. Constraints on the MM's spreads have also been considered in some of them. It is worth noting that \cite{G} shows how the two seemingly different subclasses of models (risk averse and `risk-neutral' with running penalty) can ultimately be characterized by a unique system of ordinary differential equations (ODEs) when considering two appropriate ansatz.

A relevant issue in practice that has not been included in the previous models is the following one. Empirical evidence (see, e.g., \cite{CJ0}) suggests that liquidity taken by clients depends not only on the quoted spreads but also on other unobservable factors. Indeed, the confluence of factors such as market sentiment towards the asset and the competition with other market makers also affects the intensities at which the MM receives orders. This complicated effect can be modelled in a simplified fashion by making the intensities depend as well on a hidden finite-state Markov chain, effectively reflecting the regime or state of the market (different levels from very slow to very active). To the best of our knowledge, this has only been briefly done in the Avellaneda--Stoikov framework by proposing an approximation for the optimal strategy with exponential intensities \cite[Sect.5.1]{CJ0} or studying a simple two-states version with power-law intensities \cite[Sect.5.4]{BL}. Both of these papers deal only with the `risk-neutral' (possibly penalized) case and make the unrealistic assumption that the current market regime is known by the MM.\footnote{\cite[Sect.3.3]{BL0} (arXiv version) also studies optimal trade execution under partial information, albeit with uncontrolled intensity.} 

When the state of the market is unknown, the problem becomes significantly more challenging for several reasons:
\begin{itemize}
\item It becomes a combined problem of stochastic control and filtering. The MM needs to dynamically make her best possible prediction of the market regime (or more precisely, its distribution), based on the information she has (i.e., the orders she has received so far and the evolution of the reference price), and adjust her spreads accordingly. This prediction is known as a filter.
\item The associated HJB PIDE has higher dimension and more non-linearities.
\item The standard approach used in all of the previously cited papers relies on reducing the HJB PIDE to a system of ODEs by means of an ansatz, proving that such a system has a classical (smooth) solution, and recovering the value function via a verification theorem. Under partial information however, the reduced HJB equation is still a complicated PIDE such that, in general, a classical solution may not exist (or it may be too difficult to prove otherwise). Hence, the ansatz argument breaks down. 
\item As a consequence, one needs to resort to the concept of viscosity solutions \cite{FS}. In addition, the numerical resolution unfailingly becomes a lot more involved than for simple systems of ODEs.
\item On a technical level, the construction of the model is not straightforward. The MM needs to adjust her strategy based on her observable information, such as the arriving orders, but this flow of information is in turn affected by the MM's actions.
\end{itemize}

In this chapter (based on \cite{CZ1}), we solve the problem of the MM under partial information. First, we start by unifying in one single formulation all the modelling features described so far. This allows us to simultaneously tackle all the models at once with a single approach, while generalising them at the same time. Indeed, our formulation allows for the interaction of any CARA utility (whether risk-neutral or risk averse) with running inventory penalty, terminal execution cost, inventory constraints and spread constraints. Further, motivated by practitioners' needs, we strongly generalize the intensity shapes to any continuous, decreasing to zero functions, adding modelling flexibility.\footnote{This is done at the expense of renouncing to uniqueness in the optimal strategy. We also assume the decay to be `fast enough' in certain cases.} We even allow for the inventory to be unconstrained when no penalties are present. (This scenario is considered mainly for the sake of completeness and comparison, at almost no extra cost.)

Secondly, we let the intensities depend on a $k$-dimensional hidden Markov chain. Following \cite{CEFS}, we use a weak formulation\footnote{That is, with controlled probability measures.} to construct a well-defined model (with exogenous information) and solve the filtering problem by means of the reference probability approach of stochastic filtering \cite[Chpt.VI]{B}. The rigorous setting of the model and its full characterization are carried out in Section \ref{s:model}, while the filtering problem is solved in Section \ref{s:filtering}.

The optimization problem of the MM is then reformulated in terms of the usual state variables together with the $k$-dimensional observable distribution of the Markov chain (Section \ref{s:viscosity}). At this point, the problem is too involved both analytically and numerically. We proceed by showing that, formally, there is an ansatz for the value function that reduces the dimensionality of the problem (i.e., the standard approach). However, as this methodology is not valid any more, we rigorously prove that the ansatz decomposition holds true (Theorem \ref{main_theorem_1}), and that the reduced value function can be characterized as the unique continuous viscosity solution of its formally derived equation (Theorem \ref{main_theorem_2}). The latter is done by harnessing results of piecewise-deterministic Markov process (PDMPs) \cite{CEFS, DF}, as first defined in \cite{D1}.

Thirdly, we solve the idealized problem of a MM with full information, who can observe the Markov chain (Section \ref{s:full_info}). We show that a similar ansatz and the standard approach work out in this case. We prove that the MM's value function is a classical solution of its HJB equation via a general verification theorem (Theorem \ref{verif_full_info}) and we recover the well known strategies for one regime as particular cases. 

Finally, we compare the optimal strategies under full and partial information through a concrete example, by numerical analysis (Section \ref{s:numerics}). In particular, we show that the optimal full information spreads are biased when the exact regime is unknown, and using them becomes suboptimal. We interpret the adjustment needed in terms of observable order flow volatility and sensitivity of the expected profit to observable regime changes; and we show how this effect becomes higher, the longer the waiting time in between orders (leading to higher uncertainty for the MM).


\section{Setting and main assumptions}
\label{s:model}
\setcounter{assumptions}{0}
\setcounter{subsection}{0}

\subsection{Preliminaries on the probability spaces}
\label{s:preliminaries}

We start by setting up our framework in an abstract fashion, deferring the question of existence of a model to Proposition \ref{Girsanov1}, and a characterization of all such models to Proposition \ref{Girsanov2}. As mentioned in the Introduction, we seek to construct a model under a weak formulation, so that the information flow remains exogenous (i.e., unaffected by the MM's actions).

Let $T>0$ be a given finite horizon and $(\Omega,\mathcal F,\mathbb F=(\mathcal F_t)_{0\leq t\leq T})$ a right-continuous filtered measurable space with $\mathcal F_T=\mathcal F$. Suppose it supports three adapted stochastic processes $W,\Np,\Nm$. (More assumptions to be added in the sequel.) Let $\mathbb F^{W,\Np,\Nm}=(\sigma(W_u,\Np_u,\Nm_u:0\leq u\leq t)\vee\mathcal F_0)_{0\leq t\leq T}$ be the natural filtration of these processes enlarged (`completed') by $\mathcal F_0$ and define the set of \textit{admissible spreads} as
\begin{equation}
\label{A}
\mathcal A\defeq\{\delta:[0,T]\times\Omega\to\overline{(\delmin,\delmax)}: \delta\mbox{ is }\mathbb F^{W,\Np,\Nm}\mbox{-predictable and bounded}\},
\end{equation}
where $-\infty\leq\delmin<\delmax\leq+\infty$ are fixed constants and $\overline{(a,b)}=\mathbb R\cap [a,b]$ denotes the closure of the interval $(a,b)$ in $\mathbb R$, for any $-\infty\leq a<b\leq+\infty$. Note that for $\delmin=-\infty$ and $\delmax=+\infty$, the admissible spreads are not uniformly bounded. The self-imposed constraints $\delmin,\delmax$ for the MM's spreads can be taken to be different for the bid and the ask if wanted, without any additional effort.

Consider on the former space a family $(\mPa)_{\alpha\in{\mathcal A}^2}$, $\alpha=(\delm,\delp)$, of equivalent probability measures such that the sigma algebra generated by their null sets is $\mathcal F_0$. Note that for each $\alpha$, $(\Omega, \mathcal F, \mathbb F, \mPa)$ is under the usual conditions and $\mathcal F_0$ is the completed trivial sigma algebra. We refer to $\mPa$ as the \textit{physical} or \textit{historical} probability given the \textit{admissible strategy} (or \textit{control}) $\alpha$, and we write $\mEa$ for the expectation under $\mPa$. We will drop the `$\alpha$' from the notation when there is no room for ambiguity. 

Henceforth, all the processes (resp. properties) considered are supposed to be defined (resp. hold) on the space $(\Omega, \mathcal F, \mathbb F, \mPa)$ for all $\alpha\in{\mathcal A}^2$, unless otherwise stated. For example, a `Brownian motion' is a process on $[0,T]\times\Omega$ that is an $(\mathbb F,\mPa)$-Brownian motion for all $\alpha$. All subfiltrations are understood to have been augmented to satisfy the usual conditions (which in the case of the natural filtrations of Feller processes, it amounts simply to completing them). C\`adl\`ag versions of the processes are used whenever available. 

\subsection{Description of the model}

A market maker in a quote-driven market gives binding bid/ask quotes $\Sm, \Sp$ resp. for a certain financial asset, during the time interval $[0,T]$. She receives unitary size market orders\footnote{Any constant size would work in the same way, but this convention simplifies the notation. See, e.g., \cite{G} for some formulas with arbitrary constant size under full information.} to buy/sell according to counting processes $\Nm, \Np $ starting at 0, with a.s. no common jumps and stochastic intensities $\lm, \lp$ resp. (see \cite[p.27, Def. D7]{B}). Each intensity at time $t$ depends on the corresponding spread $\delpm_t\defeq\pm\big(\Spm_t - S_t\big)$ between the MM's quotes and a reference price $S$. $S$ may be interpreted, e.g., as an aggregated mid-price or a dealer-to-dealer price, depending on the market. We assume $S$ is a Brownian motion with drift, of the form 
\begin{equation}
\label{S}
S_t= s_0 + \mu t + \sigma W_t,
\end{equation}
where $W$ is a Wiener process, $s_0,\mu\in\mathbb R$ and $\sigma\geq 0$. Note that the MM can fully specify her quote $\Spm_t$ by choosing the spread $\delpm_t$, since $\Spm_t = S_t \pm\delpm_t$. We assume $\delm,\delp\in\mathcal A$. In particular, the MM chooses her spreads based on the observation of $W,\ \Nm\mbox{ and }\Np$. We refer to $\mathbb F^{\Nm,\Np,W}$ as the \textit{observable filtration}.

In addition, the intensities also depend on a hidden Markov chain $Y$ with state space $\{1,\dots,k\}$, initial distribution $\mu_0$ and deterministic generator matrix $Q=(q^{ij}_t)_{1\leq i,j\leq k}$. We assume $Q$ is a continuous, stable and conservative $Q(t)$-matrix, i.e.: 
\begin{assumptions}[\textbf{Generator matrix}]
\label{assumptions_Q}
For all $t\in [0,T]$ and $1\leq i\leq k$:

$0\leq q^{ij}_t<+\infty\mbox{ for all }\ 1\leq j\leq k,\ j\neq i,\qquad \displaystyle\sum_{j=1}^k q^{ij}_t =0\quad\mbox{and}\quad Q\in C[0,T].$
\end{assumptions}
\noindent The previous are standing assumptions in the literature of Markov chains and they guarantee in particular the existence of a unique (up to evanescence) c\`adl\`ag  version of $Y$, and the existence of the predictable intensity kernel for its jump measure, as given in (\ref{intensity_kernels}) below. In addition, we suppose that 
\begin{equation}
\label{no_common_jumps}
Y\mbox{ has a.s. no common jumps with }\Nm\mbox{ and }\Np. 
\end{equation}
In this model, $Y_t$ represents the \textit{market regime} at time $t$ and results from the interaction of a range of different factors, such as market sentiment towards the asset and varying levels of competition with other liquidity providers. These effects are almost never explicitly modelled in the mathematical literature of market making. To this end, it is natural to assume that $Y$ is not directly observable by the MM who can only see $\Nm,\Np$ and $W$. (As often assumed in filtering and control with hidden Markov chain models, the parameters $k,Q,\mbox{ and }\mu_0$ of $Y$ are known to the MM, who needs to model/estimate them in practice.)

We define the MM's inventory process 
\begin{equation}
\label{N}
N_t\defeq n_0 + \Nm_t - \Np_t
\end{equation}
and the cash account process
\begin{equation}
\label{X}
\Xdel_t\defeq x_0 + \int_0^t \left(S_u+\delp_u\right)d\Np_u - \int_0^t \left(S_u-\delm_u\right)d\Nm_u,
\end{equation}
for some fixed initial values $n_0\in\mathbb Z$ and $x_0\in \mathbb R$. Note that 
$$\mathbb F^{\Nm,\Np} = \mathbb F^N\quad\mbox{ and }\quad\mathbb F^{\Nm,\Np,W}=\mathbb F^{N,W},$$ 
since $\Nm,\Np$ have a.s. no common jumps.

We shall make some very natural modelling assumptions on the intensities. These strongly generalize those in \cite{BL, G, GL} to a context with several market regimes, without the need of any conditions on the derivatives or even smoothness.

\begin{assumptions}[\textbf{Orders intensities}]
\label{assumptions_intensities}
There exist functions $\Lipm:\overline{(\delmin,\delmax)}\to (0,+\infty)$ for $i=1,\dots,k$ and $\Nmax\in \mathbb N \cup\{+\infty\}$ with $-\Nmax\leq n_0\leq\Nmax$, such that:
\begin{enumerate}[label=(\roman*)] 
\item \label{assumptions_intensities_lambda}
$\lpm_t = \lambda_t^{\pm,\alpha} = \Lpm_{Y_{t^-}}(\delpm_t)\mathbbm 1_{\{\mp N_{t^-}<\Nmax\}}$ if $\alpha=(\delm,\delp)$. 
\item \label{assumptions_intensities_Lambda}
$\Lipm$ is continuous, decreasing (not necessarily strictly) and $\displaystyle\lim_{\delta\to+\infty}\delta^{\mathbbm 1_{\{\gamma=0\}}}\Lipm(\delta)=0$ if $\delmax=+\infty$, for all $1\leq i\leq k$.
\end{enumerate}
\end{assumptions}

\begin{notation*}
For the sake of readability, we will often write $\mathbbm 1^{\pm}_t$ instead of $\mathbbm 1_{\{\mp N_{t^-}<\Nmax\}}$.
\end{notation*}

\noindent $\Nmax$ is a self-imposed constraint on the MM's inventory (whichever its sign) as originally proposed in \cite{GLFT}. When $\Nmax=\infty$ we are in the unconstrained context, while a finite $\Nmax$ effectively means that the MM will not buy (resp. sell) whenever $N_{t^-}=\Nmax$ (resp. $N_{t^-}=-\Nmax$). In practice, the MM could achieve this either by abstaining from quoting or by quoting with an excessively large spread that prevents any transaction (i.e., a `stub' or `placeholder quote').\footnote{For simplicity and without loss of generality, only the first alternative is formalized in our model (as done in \cite{G, GL, GLFT}) and this is reflected in the admissible spreads being real-valued. We could also allow for different constraints depending on the sign of the inventory, but we refrain from this to simplify the notation.} 

\begin{remark}
\label{expected_margin}
The last assumption states in particular that the intensities should decrease to zero when the spreads grow arbitrarily large. Furthermore, when $\gamma=0$ we require that they decay faster than $1/\delta$. Loosely speaking, this states that for any $\gamma\geq 0$ the `expected' utility of the MM's instantaneous margin, $\digamma^\pm_i(\delta)\defeq\Lambda_i(\delta) U_\gamma(\delta)$, should vanish for `stub quotes'. This is a reasonable assumption in practice, as the opposite could lead to unrealistic optimal strategies such as continuously quoting `infinite spreads'.
\end{remark}

\begin{remark*}
Note that the above intensities are predictable, and since any admissible spreads $\delm,\delp$ are bounded, $\lm,\lp$ are in turn bounded. Accordingly, $\Nm,\Np, N$ and $X$ are non-explosive (see, e.g., \cite[p.27 T8]{B}). Furthermore, for any constant $\lmax>0$ such that $\lm + \lp\leq\lmax$, it is easy to see that for all $p\in\mathbb N,\ t\in[0,T]$ and $r\geq 0$,
\begin{equation}
\label{bounded_moments}
\mEa[|N_t|^p]\leq \mEa[(n_0 +\Nm_t+\Np_t)^p]\leq\mEa[(n_0 + M)^p],\ {\mathcal M}^\alpha_{|N_t|}(r)\leq {\mathcal M}^\alpha_{n_0 +\Nm_t+\Np_t}(r) \leq {\mathcal M}^\alpha_{n_0+M}(r),
\end{equation}
where $M\sim\mbox{Poisson}(\lmax T)$ and ${\mathcal M}^\alpha_R$ is the moment generating function of a random variable $R$.\footnote{By possibly enlarging the space, one can consider a counting process $Z$ with no common jumps with $\Nm,\Np$ and stochastic intensity $\lmax-\lm-\lp\geq 0$. Then the process $Z+\Nm+\Np$ is a Poisson with intensity $\lmax$ that dominates $\Np+\Nm$. The claim follows immediately.}
\end{remark*}

The following are some examples which we shall come to back in Section \ref{s:computing_spreads}, when revisiting the standard assumptions in the literature. We remark that exponential, power-law and logistic intensities are the explicit ones most commonly used in the mathematical literature, for tractability reasons.
\begin{examples}
\label{ex}
\begin{enumerate}
\item \label{ex1}
$\Lipm(\delta)=\apm_i e^{-\bpm_i \delta}$, with $\apm_i,\bpm_i>0,\ -\infty\leq\delmin<\delmax\leq+\infty$.
\item \label{ex2}
$\Lipm(\delta)=\frac{\apm_i}{1+\cpm_i e^{\bpm_i \delta}}$, with $\apm_i,\bpm_i,\cpm_i>0,\ -\infty\leq\delmin<\delmax\leq+\infty$.
\item \label{ex3}
$\Lipm(\delta)=\apm_i \big(\pi/2 + \arctan(-\bpm_i\delta+\cpm_i)\big)$, with $\apm_i,\bpm_i>0,\ \cpm_i\in\mathbb R,\ -\infty\leq\delmin<\delmax\leq+\infty$. 
\item \label{ex4}
$\Lipm(\delta)=\apm_i (\bpm_i\delta+\cpm_i)^{-\Dpm_i}$, with $\apm_i,\bpm_i,\Dpm_i>0,\ \cpm_i\in\mathbb R,\ \max_i\{-\cpm_i/\bpm_i\}<\delmin<\delmax\leq+\infty$.
\end{enumerate}
\end{examples}
To fix ideas, consider Example \ref{ex1} and let us observe that if $Y$ represents decreasing levels of competition for the MM, the values of $\apm_i$ (resp. $\bpm_i$) should increase (resp. decrease) as $Y$ increases. The same is true if $Y$ represents increasing levels of positive sentiment (or `bullishness') towards the asset. In practice, $Y$ will result from the combination of these effects and many more.

Suppose the MM can hedge any remaining inventory at time $T$ at the reference price $S_T$ minus a certain execution cost and that she has a terminal CARA utility function
\begin{equation*}
U_\gamma(c)=\frac{1-e^{-\gamma c}}{\gamma},\mbox{ for }c\in\mathbb R,
\end{equation*}
when $\gamma> 0$ and $U_0=Id$. The parameter $\gamma$ is known as absolute risk aversion.

Neglecting discounting between $0$ and $T$, we consider the optimization problem faced by the MM who tries to maximize the expected utility of her terminal \textit{penalized profit and loss} (P\&L),

\begin{equation}
\label{problem}
\sup_{\delm,\delp\in\mathcal A}\mEdel\left[U_\gamma\Big(\Xdel_T+S_T N_T - \ell(N_T)-\frac{1}{2}\sigma^2\zeta\int_0^TN_t^2dt \Big)\right],
\end{equation}
where: 
\begin{enumerate}[label=(\roman*)]
\item $\zeta\geq 0$. When $\zeta>0$ we are including in the model a \textit{running penalty}, as firstly done in \cite{CJ}, for the MM to further control her \textit{accumulated inventory risk}. This is the same as subtracting the variance of the mark-to-market value of the inventory, weighted by $\zeta$.
\item $\ell:\mathbb R\to\mathbb [0,+\infty)$ represents the final execution cost and is increasing on $[0,+\infty)$, decreasing on $(-\infty,0]$ and $\ell(0)=0$. (And usually convex in practice.)
\item If $\Nmax=+\infty$, we set $\gamma=\zeta=0\equiv\ell$.
\end{enumerate}
The last restriction on the parameters states that the only case we will consider in which the inventory could be arbitrary large (whichever its sign) is that of a completely risk-neutral MM with negligible costs. The risk averse case is far more challenging and has not been treated in full mathematical detail even in the complete information case (see discussion in \cite{GLFT}). The model given by (iii) is of secondary interest in practice, yet, it will allow us to further understand the general problem and methodology and to have a more holistic view at almost no extra cost.

\begin{remark*}
Previous works in optimal market making do not consider both penalty and CARA utility (with $\gamma>0$) in unison, and instead treat the two families of models separately \cite{G}. Besides the obvious interest in unifying the approach and generalizing existing models, allowing for $\gamma>0$ and $\zeta>0$ simultaneously adds flexibility to the risk managing capabilities of the MM. Indeed, in \cite{G} the author derives some HJB-type systems of ODEs for each problem with a unique risk aversion parameter $\bar\gamma$, and then relates them by the introduction of an auxiliary parameter $0\leq \bar\xi\leq \bar\gamma$. The later is afterwards interpreted as measuring risk aversion to non-execution risk only and it is carried forward to the formulae of the `optimal strategies'. However, such strategies are not shown to result from any particular optimization problem for $0<\bar\xi<\bar\gamma$. By looking at the equations of the full information version of our model (Section \ref{s:full_info}) with a single market regime, one can immediately see that the equations in \cite{G} are obtained via the reparametrization $\gamma=\bar\xi$ and $\zeta = \bar\gamma - \bar\xi$. Thus, the formulation in (\ref{problem}) allows us, in particular, to establish the optimality of the strategies in \cite{G} for any value of the non-execution risk aversion and, in general, to differentiate between aversion to different types of risks. We can say that, in our model, $\gamma$ represents aversion to all types of risks, while the penalty $\zeta$ is used to further increase the aversion to price risk only.
\end{remark*}

\subsection{Jump measures}

It will be useful in the sequel to understand which are the jump measures and intensity kernels involved in our model. Let $\mu^N(dt,dz)$ (or simply $\mu^N$) be the jump (or counting) measure of $N$ (see def. in \cite[p.234]{B} or \cite[p.69]{JS}). We use a similar notation for the jump measures of $Y$ and the pair $(Y,N)$. These are all finite random measures since the processes are non-explosive for any $\delp,\delm\in\mathcal A$ (see (\ref{bounded_moments})), and they admit predictable intensity kernels (see \cite[p.235 D2]{B}). If $m_{z_0}(dz)$ is the Dirac measure at a point $z_0\in \mathbb R$, then the $(\mPa,\mathbb F)$-predictable intensity kernels of $\mu^N,\mu^Y$ and $\mu^{(Y,N)}$ are respectively given by:

\begin{equation}
\label{intensity_kernels}
\begin{split}
\eta^{\alpha,N}_t(dz) & =\eta^N(Y_{t^-}, \delm_t,\delp_t,N_{t^-},dz)\\
																		& = \lm_t m_{1}(dz) + \lp_t m_{-1}(dz)\\
																		& = \Lm_{Y_{t^-}}(\delm_t) \mathbbm 1_{\{N_{t^-}<\Nmax\}} m_{1}(dz)+ \Lp_{Y_{t^-}}(\delp_t)\mathbbm 1_{\{-N_{t^-}<\Nmax\}} m_{-1}(dz),\\
				\eta^Y_t(dh)  & = \eta^Y(Y_{t^-},dh)=\sum_{j\neq Y_{t^-}}q^{Y_{t^-},j}_t m_{j-Y_{t^-}}(dh)												
\end{split}
\end{equation} 
and 
\begin{equation}
\label{intensity_kernel_pair}
\eta^{\alpha,(Y,N)}_t(dh,dz) =\eta^{\alpha,N}_t(dz)\otimes m_0(dh) + \eta^Y_t(dh)\otimes m_0(dz).
\end{equation}

\noindent We denote by ${\bar\mu}^N_\alpha(dt,dz)$, $\bar{\mu}^Y_\alpha(dt,dh)$ and ${\bar\mu}^{(Y,N)}_\alpha(dt,dh,dz)$ the corresponding $(\mathbb F,\mathbb P^\alpha)$-compen\-sated measures.
\begin{remark*}
Note that equation (\ref{intensity_kernel_pair}) is a consequence of assuming that $Y,N$ have a.s. no common jumps (which in turn is equivalent to (\ref{no_common_jumps})). Indeed, it is easy to see that $Y,N$ have a.s. no common jumps if and only if
\begin{equation}
\mu^{(Y,N)}(dt,dh,dz)= \mu^{N}(dt,dz)\otimes m_0(dh) + \mu^Y(dt,dh)\otimes m_0(dz),
\end{equation}
which implies (\ref{intensity_kernel_pair}).
\end{remark*}

\subsection{Construction and characterization of the model}

We now prove the existence of a model within our framework (Proposition \ref{Girsanov1}). We construct it by means of a reference probability (an overview of which can be found, e.g., in \cite[Chpt.VI]{B}) and, in particular, the Girsanov theorem for point processes.

For the following proposition, let us consider a filtered probability space $(\Omega,\mathcal F,\mathbb F,\mathbb Q)$ under the usual conditions, with $\mathcal F=\mathcal F_T$ and $\mathcal F_0$ the completed trivial sigma algebra. We refer to $\mathbb Q$ as the \textit{reference probability}. Suppose this space supports a two-dimensional counting process $(\Nm,\Np)$ (see def. in \cite[Chapt.II]{B}), a Wiener process $W$ and a Markov chain $Y$ with finite state space $\{1,\dots,k\}$, initial distribution $\mu_0$ and time-dependent generator matrix $Q=(q^{ij}_t)_{1\leq i,j\leq k}$ satisfying Assumptions \ref{assumptions_Q}, such that $\Nm,\Np,Y$ verify (\ref{no_common_jumps}). We set $N=\Nm-\Np$ as before and suppose $\Npm$ has intensity $\mathbbm 1^{\pm}_t=\mathbbm 1_{\{\mp N_{t^-}<\Nmax\}}$. ($\Nmax\in\mathbb N\cup \{+\infty\}$ is given.) Such a simpler model can be constructed for example as a product of canonical spaces, with the existence of the counting processes with the right intensities proved in \cite[Thm.24 and Cor.31]{JP}.\footnote{The finite-dimensional result has the same proof as in one dimension, starting from independent Poisson measures.} Let $\Lipm:\overline{(\delmin,\delmax)}\to\mathbb R$ for $i=1,\dots,k$ be functions under the Assumptions \ref{assumptions_intensities} \textit{(\ref{assumptions_intensities_Lambda})},for given $-\infty\leq\delmin<\delmax\leq+\infty$ and $\Nmax\in \mathbb N \cup\{+\infty\}$ with $-\Nmax\leq n_0\leq\Nmax$, and define $\mathcal A$ as in (\ref{A}).

\begin{proposition}
\label{Girsanov1}
Let $\alpha=(\delm,\delp)\in{\mathcal A}^2$ and define the process $Z^\alpha$ as the \textit{stochastic exponential}
$$Z^\alpha\defeq \mathcal E\left( \int_0^\cdot\big(\Lm_{Y_{u^-}}(\delm_u)-1\big)\big(d\Nm_u- \mathbbm 1^-_u du\big)+\big(\Lp_{Y_{u^-}}(\delp_u)-1\big)\big(d\Np_u- \mathbbm 1^+_u du\big)\right).$$ 

Then,
\begin{enumerate}[label=(\roman*)]
\item $Z^\alpha$ is a strictly positive uniformly integrable (UI) martingale. In particular, $\mathbb E^{\mathbb Q}[Z^\alpha_T]=1$.
\item $\mPa$ defined by $\frac{d\mPa}{d\mathbb Q} = Z^\alpha_T$ is an equivalent probability measure such that, for $(\mathbb F,\mPa)$, $\Npm$ is a counting process with intensity $\lpm_t = \Lpm_{Y_{t^-}}(\delpm_t)\mathbbm 1^{\pm}_t$, $W$ is a Wiener process and $Y$ is a Markov chain with state space $\{1,\dots,k\}$, initial distribution $\mu_0$ and generator matrix $Q$. 
(That is, Assumptions \ref{assumptions_Q}, \ref{assumptions_intensities} and (\ref{no_common_jumps}) are all verified.)
\end{enumerate}
\end{proposition}

\begin{proof}
$\Nm,\Np$ are finite variation (FV) processes with no common jumps, hence $[\Nm,\Np]=0$ $\mathbb Q$-a.s. Therefore, the multiplicativity of the stochastic exponential \cite[p.138]{JS} and the exponential formula for FV processes \cite[p.337 T4]{B} yield the explicit expression $Z^\alpha= Z^{\alpha,-} Z^{\alpha,+}$, with
\begin{equation}
\label{exponential_explicit}
Z^{\alpha,\pm}_t=\exp\left(\int_0^t \big(1-\Lpm_{Y_{u^-}}(\delp_u)\big)\mathbbm 1^{\pm}_u du\right)\prod_{\substack{u\leq t:\\ \Delta\Npm_u\neq 0}}\Lpm_{Y_{u^-}}(\delp_u).
\end{equation}
Then \textit{(i)} follows from the strict positivity and boundedness of $\Lpm_i(\delpm)$ for $1\leq i\leq k$, as in \cite[p.168 T4]{B} with terminal time $T$. The only difference is that in our case the $\mathbb Q$-intensity of $\Npm$ is $\mathbbm 1_{\{\mp N_{t^-}<\Nmax\}}\leq 1$ instead of $1$. The proof remains the same though, simply recalling that the $\mathbb Q$-moment generating function of $\Npm_T$ is dominated by the one of a standard Poisson random variable (see (\ref{bounded_moments})). Uniform integrability is immediate.

\textit{(i)} guarantees that $\mPa$, as defined in \textit{(ii)}, is an equivalent probability measure. The shape of the $\mPa$-intensities of $\Nm,\Np$ is due to \cite[p.166 T3]{B}. 

The fact that $W$ is still a Wiener process under $\mPa$ is a consequence of the Girsanov--Meyer Theorem \cite[p.132 Thm.35]{P} and Levy's characterization Theorem.

As for $Y$, note first that its initial distribution does not change as $Z^\alpha_0=1$, and the same goes for its infinitesimal generator operator. To see this, consider the $\mathbb Q$-generator operator $\mathcal A^Y_t:\mathbb R^E\to\mathbb R^E$, $\mathcal A^Y_tf(i)= \sum_{j=1}^k q_t^{ij}f(j)$. We know the process $M_t=f(Y_t)-f(Y_0) - \int_0^t\mathcal A^Y_uf(Y_u)du$ is a $\mathbb Q$-local martingale. Once again, by the Girsanov--Meyer Theorem., $M$ is also a $\mPa$-local martingale ($[M,Z]=0$ since $Y,\Nm,\Np$ are FV processes satisfying (\ref{no_common_jumps})). Moreover, being bounded, $M$ is a true  $\mPa$-martingale and $Y$ solves, for $(\mathbb F,\mPa)$, a well-posed martingale problem for $(\mathcal A^Y_t, \mu_0)$. This implies $Y$ is a $\mPa$-Markov chain with a uniquely determined law \cite[p.184 Thm.4.2]{EK}.\footnote{Although our martingale problem is non-homogeneous in time, $(Q_t)$ is deterministic, so this does not represent a problem.} That $Q$ is also the $\mPa$-generator matrix follows from uniqueness.    

(\ref{no_common_jumps}) clearly remains unchanged under an equivalent change of probability measure.

\end{proof}

Reciprocally, suppose we start with a family of filtered probability spaces $\{(\Omega, \mathcal F, \mathbb F = (\mathcal F_t)_{0\leq t\leq T}, \mPa)\}_{\alpha\in{\mathcal A}^2}$ as in Section \ref{s:preliminaries}, supporting a counting process $(\Nm,\Np)$ and a Markov chain $Y$ satisfying condition (\ref{no_common_jumps}), together with a Wiener process $W$. Suppose also that $Y$ has finite state space $\{1,\dots,k\}$, initial distribution $\mu_0$ and time-dependent generator matrix $Q$, and that our Assumptions \ref{assumptions_Q} and \ref{assumptions_intensities} are in place. We put $N= \Nm-\Np$ as always. We would like to characterize any such model in terms of a reference probability as in Proposition \ref{Girsanov1} by an inverse change of measure. However, we only claim the uniqueness of the reference probability on $\mathbb F^N$. We will come back to this result in the sequel.

\begin{proposition}
\label{Girsanov2}
Let $\alpha=(\delm,\delp)\in{\mathcal A}^2$ and define 
$$\bar{Z}^\alpha\defeq \mathcal E\left(\int_0^\cdot\Big(1/\Lm_{Y_{u^-}}(\delm_u)-1\Big)\big(d\Nm_u-\lm_u du\big)+\Big(1/\Lp_{Y_{u^-}}(\delp_u)-1\Big)\big(d\Np_u-\lp_u du\big)\right).$$  

Then,
\begin{enumerate}[label=(\roman*)]
\item $\bar{Z}^\alpha$ is a strictly positive UI martingale. In particular, $\mEa[\bar{Z}_T]=1$. 
\item $\mQa$ defined by $\frac{d\mQa}{d\mPa} = \bar{Z}^\alpha_T$ is a probability measure equivalent to $\mPa$ such that, for $(\mathbb F,\mQa)$, $\Npm$ is a counting process with intensity $\mathbbm 1^{\pm}_t$, $W$ is a Wiener process and $Y$ is a Markov chain with state space $\{1,\dots,k\}$, initial distribution $\mu_0$ and generator matrix $Q$. Furthermore, $Y, N, W$ are independent.
\item If we define $Z^\alpha$ as in Lemma \ref{Girsanov1} under $\mathbb Q^\alpha$, then $Z^\alpha=1/\bar{Z}^\alpha$ (i.e., the two changes of measure are inverse of each other).
\item For all $\tilde\alpha\in{\mathcal A}^2$, $\mathbb Q^{\tilde\alpha}\equiv\mQa$ on $\mathbb F^N_T$. 
\end{enumerate}
\end{proposition}

\begin{proof}
\textit{(i)} and \textit{(ii)} are proved just as in Proposition \ref{Girsanov1} (the processes $1/\Lpm_i(\delpm)$ are strictly positive and bounded for $1\leq i\leq k$, and the intensities $\lp,\lpm$ are bounded). The $\mQa$-independence of $Y,N,W$ is a consequence of \cite[p.543 Lem.9.5.4.1]{JYC}. That is, the $(\mQa,\mathbb F)$-martingales $Y - \int_0^\cdot\sum_{j=1}^k q_u^{Y_{u^-},j}j du,\ N-\int_0^\cdot(\mathbbm 1^-_u - \mathbbm 1^+_u)du\mbox{ and }W$ have the predictable representation property with respect to $\mathbb F^Y,\ \mathbb F^N\mbox{ and }\mathbb F^W$ resp. (see \cite[p.239 T8]{B} for compensated counting measures) and they are $(\mathbb F,\mQa)$-orthogonal, which implies their independence.

The explicit expressions for the stochastic exponentials (see (\ref{exponential_explicit})) show straightforwardly that $Z^\alpha\bar{Z}^\alpha=1$, proving \textit{(iii)}.

\textit{(iv)} is due to \cite[p.64 T8]{B}.
\end{proof}

\section{Filtering problem}
\label{s:filtering}
\setcounter{assumptions}{0}
\setcounter{subsection}{1}

Since the MM cannot directly observe all the information in $\mathbb F$ but only $\mathbb F^{N,W}$ (in particular, she cannot observe $Y$), in order to solve the optimization problem (\ref{problem}) under partial information we want to reduce it first to an equivalent one under full information. Throughout this section we work under $\mPa$ with $\alpha=(\delm,\delp)\in{\mathcal A}^2$ fixed. We sometimes omit $\alpha$ from the notation for simplicity.

Recall that for any c\`adl\`ag bounded process $M$ (not necessarily adapted) on a filtered probability space $(\Sigma,\mathcal H,\mathbb H,\mathbb P)$ satisfying the usual conditions, the \textit{optional projection} of $M$ on $\mathbb H$ is the unique c\`adl\`ag process $^oM^{(\mathbb P,\mathbb H)}$ such that $ ^oM^{(\mathbb P,\mathbb H)}_t = \mathbb E^{\mathbb P}[M_t|\mathcal H_t]$ a.s. for each $t$. Its existence is guaranteed by the Optional Projection Theorem (see, e.g., \cite[p.264]{JYC} or \cite[p.357-358]{N}). (We omit $\mathbb P$ and/or $\mathbb H$ when clear from the context.) Let us consider the optional projections
$$\piia \defeq^o(\mathbbm 1_{\{Y_t=i\}})^{(\mPa,\mathbb F^{N,W})}, 1\leq i\leq k,\quad\mbox{ and }\quad\pia\defeq (\Pi^{\alpha,1},\dots,\Pi^{\alpha,k}).$$ 
In other words, $\pia$ is the unique c\`adl\`ag version of the conditional distribution of $Y$ given the observable information, i.e., $\piia_t = \mPa(Y_t=i|\mathcal F^{N,W}_t)$.  We now characterize the observable (that is, the $\mathbb F^{N,W}$-) predictable intensities of $\Nm,\Np\mbox{ and }\mu^N$ in terms of $\pia$. Loosely speaking, the observable intensity of, say $\Np$, is obtained projecting: $\mEa[\lp_t|\mathcal F^{W,N}_t]$ \cite[p.32 Comment and Pseudo-Proof of T14]{B}. The only technical difficulty is that of finding a predictable version of such process. $ ^o(\lp)$ has the desired projective property but it is not predictable in general. $ ^o(\lp)_{t^-}$, on the other hand, is predictable but does not normally enjoy the projective property. In fact, the process we are looking for is `in between' these two.

\begin{proposition}[\textbf{Observable intensities}]
\label{observable_intensities}
The $(\mPa,\mathbb F^{N,W})$-predictable intensities of $\Npm$ and $\mu^N$, resp., are
\begin{equation*}
\widehat{\lpm_t}^\alpha\defeq\mathbbm 1^{\pm}_t\sum_{i=1}^k\piia_{u^-}\Lipm(\delpm_u)\quad\mbox{ and }\quad
\widehat{\eta}^{\alpha,N}_t(dz) \defeq \widehat{\lm_t}^\alpha m_{1}(dz) + \widehat{\lp_t}^\alpha m_{-1}(dz).
\end{equation*}
\end{proposition}

\begin{notation*}
We set $\widehat{\Lpm}(\pi,\delta)=\sum_{i=1}^k\pi^i\Lipm(\delta)$; thus, $\widehat{\lpm_t}^\alpha=\mathbbm 1^{\pm}_t\widehat{\Lpm}(\pia_{u^-},\delpm_u)$.
\end{notation*}

\begin{proof}
We prove it just for $\Np$, the others being analogous, and we omit $\alpha$ for simplicity. It is clear that $\widehat{\lp}$ is predictable. We need to check that for any $\mathbb F^{W,N}$-predictable process $\psi\geq 0$, 
$$
\mathbb E\left[\int_0^T\psi_t\widehat{\lp_t}dt\right] = \mathbb E\left[\int_0^T\psi_td\Np_t\right].
$$
For any $1\leq i\leq k$, each path of the c\`adl\`ag processes $\pii$ and $\mathbbm 1_{\{Y_\cdot=i\}}$ has only countably many jumps. We can therefore interchange these processes and their left limits when integrating with respect to $dt$. By properties of the conditional expectation and Fubini's Theorem,

\begin{equation*}
\begin{split}
&\mathbb E\left[\int_0^T\psi_t\widehat{\lp_t}dt\right]=\mathbb E\left[\int_0^T \psi_t\mathbbm 1^+_t\sum_{i=1}^k\pii_{t^-}\Lip(\delp_t) dt\right] = \mathbb E\left[\int_0^T \psi_t\mathbbm 1^+_t\sum_{i=1}^k\pii_t\Lip(\delp_t) dt\right]\\
																																& = \mathbb E\left[\int_0^T \mathbb \psi_t\mathbbm 1^+_t\sum_{i=1}^k\mathbb E\big[\mathbbm 1_{\{Y_t=i\}}\big|\mathcal F^{N,W}_t\big]\Lip(\delp_t)dt\right] =\int_0^T \mathbb E\left[\mathbb E\Big[\psi_t\mathbbm 1^+_t\sum_{i=1}^k\Lip(\delp_t) \big|\mathcal F^{N,W}_t\Big]\right]dt\\
																																& = \int_0^T \mathbb E\Big[\psi_t\mathbbm 1^+_t\sum_{i=1}^k\mathbbm 1_{\{Y_t=i\}}\Lip(\delp_t)  \Big]dt= \mathbb E\left[\int_0^T \psi_t\mathbbm 1^+_t\sum_{i=1}^k\mathbbm 1_{\{Y_{t^-}=i\}}\Lip(\delp_t) dt\right] =\mathbb E\left[\int_0^T \psi_t\lp_t dt\right]\\
																																&= \mathbb E\left[\int_0^T \psi_td\Np_t\right].
\end{split}
\end{equation*} 
\end{proof}

We give now the \textit{filtering (or Kushner-Stratonovich) equations} for the observable distribution of $Y$. These are coupled stochastic differential equations (SDEs) governing the dynamics of $\pia$. 

\begin{notation*}
We denote by $\Delta\subset\mathbb R^k$ the $(k-1)$-simplex (i.e., $\Delta=\{\pi\in\mathbb R^k: 0\leq\pi^i\leq 1\ \mbox{for all }i\mbox{, and }\sum_{i}\pi^i=1\}$ and by $\Delta^\circ$ its interior relative to the hyperplane $\{\pi\in\mathbb R^k: \sum_{i}\pi^i=1\}$ (i.e., $\Delta^o=\{\pi\in\Delta: 0<\pi^i< 1\mbox{ for all }i\}$).
\end{notation*}

\begin{proposition}[\textbf{Observable distribution of $Y$}]
\label{proposition_Pi}
The process $\pia=(\Pi^{\alpha,1},\dots,\Pi^{\alpha,k})$ is the unique strong solution of the constrained system of SDEs
\begin{equation}
\label{Pi}
\begin{split}
d\pii_t = \sum_{j=1}^k q_t^{ji}\pij_t dt & + \pii_{t^-}\left(\frac{\Lim(\delm_t)}{\sum_{j=1}^k\pij_{t^-}\Ljm(\delm_t)}-1 \right)\left(d\Nm_t-\mathbbm 1^-_t\sum_{																						j=1}^k\pij_t\Ljm(\delm_t)dt\right) \\
																				 & + \pii_{t^-}\left(\frac{\Lip(\delp_t)}{\sum_{j=1}^k\pij_{t^-}\Ljp(\delp_t)}-1 \right)\left(d\Np_t-\mathbbm 1^+_t\sum_{																						j=1}^k\pij_t\Ljp(\delp_t)dt\right),
\end{split}
\end{equation}
such that $\Pi_0 = \mu_0$ and $\Pi_t\in\Delta$ for all $t\in[0,T]$ a.s. Equivalently,
\begin{equation}
\label{Pi_equivalent}
\begin{split}
d\pii_t &= \sum_{j=1}^k\left( q_t^{ji}\pij_t+\pii_t\pij_t\Big(\mathbbm 1^-_t(\Ljm-\Lim)(\delm_t) + \mathbbm 1^+_t(\Ljp- \Lip)(\delp_t) \Big) \right)dt\\ 
				& + \pii_{t^-}\left(\frac{\Lim(\delm_t)}{\sum_{j=1}^k\pij_{t^-}\Ljm(\delm_t)}-1 \right) d\Nm_t + \pii_{t^-}\left(\frac{\Lip(\delp_t)}{\sum_{j=1}^k\pij_{t^-}\Ljp(\delp_t)}-1 \right) d\Np_t,
\end{split}
\end{equation}
with $\Pi_0 = \mu_0$ and $\Pi_t\in\Delta$ for all $t\in[0,T]$ a.s.
\end{proposition}

\begin{proof}
The equivalence between the constrained systems of SDEs (\ref{Pi}) and (\ref{Pi_equivalent}) results from rearranging the terms and using $\sum_{j=1}^k\pij_t=1$.

Let us check that $\pia$ solves (\ref{Pi_equivalent}). Clearly, the constraint and the initial condition are satisfied. The verification of the SDEs is due to \cite[Prop.3.3]{CEFS} (with more details in \cite[App.A, Lemma A.2 and Prop.3.3]{CEFSv3}), albeit some considerations need to be made. 

On the one hand, the authors work with a pure jump model, with strategies adapted to the natural filtration of the driving jump process only (i.e., there is no diffusion) and constant generator matrix for the Markov chain. However, \textit{mutatis mutandis} the former differences yield no major change in the proofs. 

On the other hand, the main assumption of the authors, \cite[Asm.2.1]{CEFS}, postulates the existence of some deterministic measure $\tilde{\eta}^{N}(dz)$ on $\mathbb R$ with compact support\footnote{In \cite{CEFS} the support is assumed to be a subset of $(-1,\infty)$, but this is only for a `return (or yield) process' as in their case.} such that for all $ i\in E,\ \delp,\delm\in \overline{(\delmin,\delmax)},\ n\in [-\Nmax,\Nmax]\cap\mathbb Z$, the measure $\eta^N(i, \delp,\delm,n,dz)$ is equivalent to $\tilde{\eta}^{N}(dz)$ and the Radon-Nikodym derivative $d\eta^{N}(i, \delp,\delm,n,\cdot)/d\tilde{\eta}^N$ is uniformly bounded and bounded away from zero $\tilde{\eta}^{N}(dz)$-a.s. 

Since the spread processes are fixed and bounded, we can assume without loss of generality that $\delmin$ is finite. Setting $\tilde{\eta}^N(dz) \defeq m_{1}(dz)  + m_{-1}(dz)$ we see straightforwardly that the two measures are equivalent with derivative 
\begin{equation}
\frac{d\eta^{N}(i, \delm,\delp,n,\cdot)}{d\tilde{\eta}^N}(z) = \Lim(\delm) \mathbbm 1_{\{n<\Nmax\}} \mathbbm 1_{\{z= 1\}}+ \Lip(\delp)\mathbbm 1_{\{-n<\Nmax\}} \mathbbm 1_{\{z=-1\}},
\end{equation} 
uniformly bounded by $\Lim(\delmin) + \Lip(\delmin)$. However, our model allows for $d\eta^{N}/d\tilde{\eta}^N(z)=0$, which is a consequence of having vanishing intensities $\lpm$. This poses no issue nonetheless, as $d\eta^{N}/d\tilde{\eta}^N(z)>0$ is only used in \cite[Prop.3.3]{CEFS} to guarantee $\bar{Z}^\alpha>0$ (see Proposition \ref{Girsanov2}). This condition, also satisfied in our model,\footnote{This was ultimately a consequence of the decomposition in Assumptions \ref{assumptions_intensities} \textit{(\ref{assumptions_intensities_lambda})}, that allowed for the vanishing factors of the intensities to be secluded as the reference probability intensities.} allows to go from the physical probability $\mPa$ to a reference probability $\mQa$ and backwards.

We turn now to the proof of uniqueness. We remark first that the jump height coefficients in (\ref{Pi_equivalent}) will typically not be Lipschitz (classical results for SDEs such as \cite[p.253 Thm.7]{P} cannot be applied) and the paths of the spreads need not be continuous between the jump times of $N$ (ruling out the most classical results of ODEs \cite{CL}). Nevertheless, we can still follow a pathwise ODEs approach. Let us fix a path and alternatively verify uniqueness inductively on the intervals $[\tau_m,\tau_{m+1})$, where $\tau_0\defeq 0$ and $0<\tau_1<\tau_2<\dots<\tau_M=T$ are the jump times of $N$ (including the terminal time $T$ even if there is no jump at that point). Then $A^{ji}$ is bounded for all $1\leq i,j\leq k$. Now observe that any c\`adl\`ag process $\tilde\Pi$, solving the constrained system of SDEs (\ref{Pi_equivalent}), must solve pathwise for $m=0,1,\dots,M-1$ the following system of ODEs in integral form, for $t\in [\tau_m,\tau_{m+1})$:
\begin{equation}
\label{Pi_ODE}
\tilde\Pi^i_t = R^i_m\big(\tilde\Pi_{\tau_m^-}\big) + \int_{\tau_m}^t\sum_{j=1}^k \left(q_u^{ji}\tilde\Pi^j_u+\tilde\Pi^i_u\tilde\Pi^j_u A^{ji}_u \right)du,
\end{equation}
with $R^i_m\big(\tilde\Pi_{\tau_m^-}\big)\defeq\tilde\Pi^i_{\tau_m^-}\Lipm(\delpm_{\tau_m})/\sum_{j=1}^k\tilde\Pi^j_{\tau_m^-}\Ljpm(\delpm_{\tau_m})$ if $\Delta N_{\tau_m}=\mp 1,\ m>0$, and $R^i_0\big(\tilde\Pi_{\tau_0^-}\big) = \mu^i_0$. Elementary algebra of bounded Lipschitz functions shows that $f^i:[\tau_m,\tau_{m+1})\times\Delta\to\mathbb R$, defined by $f^i(u,\pi)= \sum_{j=1}^k (q_u^{ji}\pi^j+\pi^i\pi^j A^{ji}_u)$ is Lipschitz in $\pi$ (uniformly in $u$). Let $K$ be the maximum Lipschitz constant of $f^l$ for $1\leq l\leq k$ and suppose $\pia_{\tau_m^-}=\tilde\Pi_{\tau_m^-}$ (clearly satisfied for $m=0$). Then (\ref{Pi_ODE}) yields $\|\pia_t-\tilde\Pi_t\|\leq K\int_{\tau_m}^t\|\pia_u-\tilde\Pi_u\|du$, implying $\pia_t=\tilde\Pi_t$ on $[\tau_m,\tau_{m+1})$ by Gr\"onwall's inequality. As a consequence, the equality on $[0,T)$ follows by induction. It must clearly hold at time $T$ as well, either by continuity or (if there is a jump) because $\pia_T=R_M\big(\pia_{\tau_M^-}\big)=R_M\big(\tilde\Pi_{\tau_M^-}\big)= \tilde\Pi_T$. 
\end{proof}

\begin{remark}
\label{identification}
Consider the identification $\Delta\simeq [0,1]^{k-1}$ (resp. $\Delta^\circ\simeq (0,1)^{k-1}$) obtained by the substitution $\pi^k=1-\sum_{j<k}\pi^j$ (where the choice of the $k$-th coordinate over the rest is completely arbitrary). Then the constrained system of SDEs (\ref{Pi}) (or equivalently, (\ref{Pi_equivalent})) for $(\Pi^{\alpha,1}\dots,\Pi^{\alpha,k})$ becomes an `unconstrained' system for $(\Pi^{\alpha,1}\dots,\Pi^{\alpha,k-1})\in[0,1]^{k-1}$. Henceforth, we shall use this identification whenever convenient.  
\end{remark}

We finish this section with a short lemma. It states that the conditional distribution $\pia$ can never reach the relative border of the simplex $\Delta$, provided it starts from the relative interior. This amounts to saying that all regimes have some positive probability at time zero.
\begin{lemma}
\label{open_simplex}
If $\mu_0\in\Delta^\circ$, then $\pia_t\in\Delta^\circ$ for all $0\leq t\leq T$ a.s.
\end{lemma}

\begin{proof}
We want to show that $\piia_t>0$ for all $1\leq i\leq k,\ 0\leq t\leq T$ a.s. We proceed by induction on the jump times of $N$, for each path, as at the end of the proof of Proposition \ref{proposition_Pi}. Using the same notations, let $1\leq i \leq k$ and suppose $\pija_{\tau_m^-}>0$ for all $1\leq j\leq k$ (satisfied for $m=0$ by assumption). Then (\ref{Pi_ODE}), Assumptions \ref{assumptions_Q} and the fact that $A^{ii}\equiv 0$ by definition, show that $\piia$ is absolutely continuous on the interval $[\tau_m,\tau_{m+1})$ and satisfies 
\begin{equation}
\label{aux1}
\begin{split}
(\piia_t)' & = \sum_{j=1}^k \left(q_t^{ji}\pija_t+\piia_t\pija_t A^{ji}_t \right) = \piia_t\Big(q^{ii}_t + \sum_{j\neq i}\pija_t A^{ji}_t \Big) + \sum_{j\neq i}q^{ji}_t\pija_t\\
					 & \geq\piia_t\Big(q^{ii}_t + \sum_{j\neq i}\pija_t A^{ji}_t\Big),
\end{split}
\end{equation}
for $dt$-a.e. $t\in [\tau_m,\tau_{m+1})$, subject to $\piia_{\tau_m}=R^i_m(\pia_{\tau_m^-})>0$. Let us set $s\defeq \sup ([\tau_m,\tau_{m+1}) \cap \{t\in [0,T] : \piia_t >0\})$. We need to prove that $s=\tau_{m+1}$. By the continuity of $\piia$, it must be $s>\tau_m$. Consequently, (\ref{aux1}) and the absolute continuity of $\log\piia$ on $[\tau_m,t]\subset[\tau_m,s)$ yield 
$$
\piia_t\geq R^i_m(\pia_{\tau_m^-})\exp\Big(\int_{\tau_m}^t\big(q^{ii}_u + \sum_{j\neq i}\pija_u A^{ji}_u \big)du\Big)\quad\mbox{ for }dt\mbox{-a.e. }t\in [\tau_m,s).
$$ 
If it were $s<\tau_{m+1}$, the continuity of $\piia$ again and the former inequality would imply $0=\piia_s>0$, proving by contradiction that $s=\tau_{m+1}$ and $\piia$ is positive on the whole interval $[\tau_m,\tau_{m+1})$. Positivity on $[0,T)$ now follows by induction, and it must clearly hold at time $T$ as well, as either there is a jump or we can reason as we just did.
\end{proof}
In light of the previous lemma, we will assume from here onwards that 
\begin{equation}
\label{assumption_mu_0}
\mu_0\in\Delta^\circ,
\end{equation}
and therefore work with $\Delta^\circ$ instead of $\Delta$.


\section{Value function and HJB equation}
\label{s:viscosity}
\setcounter{assumptions}{0}
\setcounter{subsection}{1}
In this section, we tackle the control problem of the MM with the filter as an additional state variable. We define the MM's value function and we aim to characterize it by means of an HJB equation. 

Let $t\in[0,T]$. We consider our model `starting at $t$' instead of $0$. Whenever a process is defined from time `$t^-$ onwards' (i.e., from time $t$ onwards and decreeing its left-limit value at $t$) this implicitly means it is constant on $[0,t)$. In particular, we use this convention for all integrals (stochastic or not) of the form $\int_t^\cdot$ and (with a slight abuse of notation) for the processes $W - W_t$ and $\Npm-\Npm_{t^-}$. We work with $t^-$ instead of $t$ due to the jumps of the processes $\Nm,\Np$. However, since $\Nm,\Np$ are quasi-left continuous,\footnote{For example, because they are increasing c\`adl\`ag processes admitting continuous compensators for (any one of) the physical probabilities \cite[p.70 Prop.1.19 or p.77 Prop.2.9]{JS}.} for most intended purposes one can drop the left limit with no harm. We define $\mathbb F^{t,W,N}=(\sigma(W_r-W_t,N_r-N_{t^-}: t\leq r\leq u\vee t)\vee \mathcal F_0)_u$ and $\mathbb F^{t,N}$ analogously. 

Let $s,x\in\mathbb R,\ n\in\mathbb Z\cap [-\Nmax,\Nmax]$ and $\pi\in\Delta^\circ\subset\mathbb R^k$. The set $\mathcal A_t$ of admissible spreads starting at $t$ is the set of $\delta\in\mathcal A$ which are independent of $\mathcal F^{N,W}_{t^-}$ (equivalently, the $\delta\in\mathcal A$ which are $\mathbb F^{t,W,N}$-predictable). Consider for each $\alpha=(\delm,\delp)\in\mathcal A^2_t$ the processes $\St, \Xt, \Nt, \pit$ defined pathwise, outside some set $A\in\mathcal F_0$, by (\ref{S}), (\ref{X}), (\ref{N}), (\ref{Pi}) resp., replacing the initial conditions $s_0,x_0,n_0,\mu_0$ at time $0$ by $s,x,n,\pi$ resp. at time $t^-$. We remark that $\mathbb F^{t,W,N} = \mathbb F^{t,W,\Nt}$ (since $\Nt = n+\Nm -\Np - (\Nm_{t^-} - \Np_{t^-})$) and all the processes defined in this section are adapted to this filtration. We further assume there exists a family of `physical' probabilities $(\mPt)_{\alpha\in\mathcal A_t^2}$ such that their null sets generate $\mathcal F_0$, and for $(\mPt,\mathbb F^{t,W,N})$ it holds that $W - W_t$ is a Wiener process and $\Npm-\Npm_{t^-}$ has predictable intensity $\widehat{\lpm}^{\alpha,t,n,\pi}$, as defined in Proposition \ref{observable_intensities} in terms of $\Nt$ and $\pit$. 

We define the penalized P\&L from $t$ to $T$ (see (\ref{problem}) for parameter restrictions) as
\begin{equation}
\label{PnL}
P^{\alpha,s,x,n}_{t,T}\defeq \Xt_T+\St_T \Nt_T- \ell(\Nt_T) - \frac{1}{2}\sigma^2\zeta\int_t^T(\Nt_u)^2du,
\end{equation}
and the \textit{value function} of problem (\ref{problem}) as
\begin{equation}
\label{value_function}
V(t,s,x,n,\pi)\defeq\sup_{\alpha\in \mathcal A^2_t}\mEt\left[U_\gamma\Big(P^{\alpha,s,x,n}_{t,T}\Big)\right].
\end{equation}
Our goal is to compute optimal or `close to optimal' strategies.\footnote{By `close to optimal' we mean that for each $\ve>0$ there exists a strategy such that the supremum in (\ref{value_function}) is attained up to $\ve$.} The Dynamic Programming Principle and Ito's Lemma allow us to formally derive (see, e.g., \cite{Bo}) the \textit{Hamilton--Jacobi--Bellman} (or \textit{dynamic programming}) partial-integro differential equation associated to $V$: 

\begin{equation}
\label{HJB_V}
\begin{split}
0 & = v_t + \mu v_s +  \frac{1}{2}\sigma^2 v_{ss} + \sum_{i,j=1}^k q_t^{ji}\pi^j v_{\pi^i} +  \frac{1}{2}\sigma^2\zeta n^2 (\gamma v -1)\\
  & + \mathbbm 1_{\{n<\Nmax\}}\sup_{\delm\in\overline{(\delmin,\delmax)}}\Big\{\sum_{i,j=1}^k (\Ljm-\Lim)(\delm)\pi^j\pi^i v_{\pi^i} + \Dm_{\delm}(v) \sum_{i=1}^k\pi^i\Lim(\delm)\Big\} \\
	& + \mathbbm 1_{\{-n<\Nmax\}}\sup_{\delp\in\overline{(\delmin,\delmax)}}\Big\{\sum_{i,j=1}^k(\Ljp-\Lip)(\delp)\pi^j\pi^i v_{\pi^i}+ \Dp_{\delp}(v) \sum_{i=1}^k\pi^i\Lip(\delp)\Big\},
\end{split}
\end{equation}
with terminal condition $v(T,s,x,n,\pi) = U_\gamma(x+sn -\ell(n))$, where:

\begin{equation*}
\begin{split}
&\Dpm_{\delta}(v)(t,s,x,n,\pi)\\
        & = v\Big(t,s,x \pm(s \pm \delta),n\mp 1, \frac{1}{\sum_{j=1}^k\pi^j\Ljpm(\delta)}\big(\pi^1\Lpm_1(\delta),\dots,\pi^k\Lpm_k(\delta)\big)\Big)- v(t,s,x,n,\pi),
\end{split}
\end{equation*}
and we convene the following:
\begin{notation*}
\textit{(i)} The derivatives with respect to $\pi=(\pi^1,\dots,\pi^k)$ should be understood via the identification of Remark \ref{identification}. 
\textit{(ii)} Although it is not meaningful to evaluate $v$ on the inventories $\pm\Nmax\pm 1$, this only happens in equation (\ref{HJB_V}) when the corresponding term vanishes. This slight abuse of notation can be found throughout previous works and we will be using it as well. 
\end{notation*}

Equation (\ref{HJB_V}) can also be seen as a coupled system of PIDEs indexed in $n\in\mathbb Z\cap [-\Nmax,\Nmax]$. (We will talk about system of equations or simply `equation' indistinctly). Nonlinearity aside, (\ref{HJB_V}) is rather complex, in particular due to being of second order, high-dimensional and with derivatives in almost all of these dimensions. Tackling it directly (either analytically or numerically) is utterly challenging. Consequently, it has become common practice for optimal market making and optimal liquidation models \textit{\`a la} Avellaneda--Stoikov \cite{AS} to propose an ansatz for the solution \cite{AS,BL,CDJ,CJ,CJR,FL1,FL2,G,GL,GLFT}. This approach however, relies heavily on the existence of a classical solution for the resulting simplified equation, so that the ansatz is ultimately proved valid by a suitable verification theorem. (See Section \ref{s:full_info} for more details.) When the simplified equation does not admit (or cannot be guaranteed to admit) a classical solution, and a viscosity approach needs to be used instead, the previous argument breaks down. 

If we attempted to solve our problem by the standard approach, a plausible ansatz for the value function could be
\begin{equation}
\label{decomposition}
V(t,s,x,n,\pi) = U_\gamma(x+sn + \Theta(t,n,\pi)).
\end{equation}
\noindent Formal substitution yields the following equation for $\Theta$:

\begin{equation}
\label{HJB_Theta}
\begin{split}
0 & = \theta_t + \mu n - \frac{1}{2}\sigma^2 n^2 (\gamma+\zeta)+ \sum_{i,j=1}^k q_t^{ji}\pi^j \theta_{\pi^i}\\
  & + \mathbbm 1_{\{n<\Nmax\}}\sup_{\delm\in\overline{(\delmin,\delmax)}}\Big\{\sum_{i,j=1}^k (\Ljm-\Lim)(\delm)\pi^j\pi^i \theta_{\pi^i} + U_\gamma\big(\delm+\Dm_{\delm}(\theta)\big) \sum_{i=1}^k\pi^i\Lim(\delm)\Big\}\\
	& + \mathbbm 1_{\{-n<\Nmax\}}\sup_{\delp\in\overline{(\delmin,\delmax)}}\Big\{\sum_{i,j=1}^k(\Ljp-\Lip)(\delp)\pi^j\pi^i \theta_{\pi^i}+ U_\gamma\big(\delp+\Dp_{\delp}(\theta)\big) \sum_{i=1}^k\pi^i\Lip(\delp)\Big\},
\end{split}
\end{equation}
with terminal condition $\theta(T,n,\pi) = -\ell(n)$, where:
\begin{equation*}
\begin{split}
\Dpm_{\delta}(\theta)(t,n,\pi) = \theta\Big(t,n\mp 1, \frac{1}{\sum_{j=1}^k\pi^j\Ljpm(\delta)}\big(\pi^1\Lpm_1(\delta),\dots,\pi^k\Lpm_k(\delta)\big)\Big) - \theta(t,n,\pi).
\end{split}
\end{equation*}
The new system of PIDEs is of first order and no longer depends on the variables $s$ and $x$ (there is no diffusion anymore). This is a considerable simplification; one that will permit effective numerical solution in Section \ref{s:numerics}. But it is not good enough for us to assert existence of a classical solution. Notwithstanding, we are able to rigorously prove the decomposition (\ref{decomposition}) and explicitly find $\Theta$ as a new `value function' (Theorem \ref{main_theorem_1}). When the control space is compact, this ultimately allows us to characterize $\Theta$ as the unique solution of the terminal condition PIDE (\ref{HJB_Theta}) in the viscosity sense (Theorem \ref{main_theorem_2}), further simplified in the unconstrained inventory case. These two theorems constitute the main theoretical results of this chapter. They allow us to safely postulate reasonable candidates for optimal (or $\epsilon$-optimal) strategies for the MM, i.e., those given by spreads that (at least approximately) realize the suprema in (\ref{HJB_Theta}).

\begin{theorem}
\label{main_theorem_1}
There exists a unique function $\Theta:[0,T]\times(\mathbb Z\cap[-\Nmax,\Nmax])\times\Delta^\circ\to\mathbb R$ such that the decomposition (\ref{decomposition}) holds true. Furthermore, there exists a family of equivalent probability measures $\tilde{\mathbb P}^{\alpha,t,n,\pi}\sim\mPt$, $\alpha=(\delm,\delp)\in\mathcal A_t^2$, such that
\begin{enumerate}[label=(\roman*)]
\item $\pit$ is the unique strong solution of (\ref{Pi}) with initial condition $(t^-,\pi)$ under $\tilde{\mathbb P}^{\alpha,t,n,\pi}$.
\item $\mu^{\Nt}$ has $(\tilde{\mathbb P}^{\alpha,t,n,\pi},\mathbb F^{t,W,N})$-predictable intensity kernel 
$$\tilde{\eta}_u^N(dz) \defeq e^{-\gamma\delm_u}\hlmtu m_{1}(dz) + e^{-\gamma\delp_u}\hlptu m_{-1}(dz).$$
\item $\Theta=U_\gamma^{-1}\circ\Psi = -\frac{1}{\gamma}\log(1-\gamma\Psi)$ with 
\begin{equation}
\label{Psi}
\Psi(t,n,\pi)\defeq\sup_{\alpha\in\tilde{\mathcal A}^2_t}\tilde{\mathbb E}^{\alpha,t,n,\pi}\left[U_\gamma\Big(\tilde P_{t,T}^{\alpha,n,\pi}\Big)\right],
\end{equation}
\end{enumerate}
where $\tilde P_{t,T}^{\alpha,n,\pi}\defeq\int_t^T\big\{ U_\gamma(\delm_u) \hlmtu + U_\gamma(\delp_u) \hlptu + \mu \Nt_u -\frac{1}{2}\sigma^2(\gamma+\zeta)(\Nt_u)^2\big\}du - \ell(\Nt_T)$ and $\tilde{\mathcal A}_t\defeq\{\delta\in\mathcal A:\delta\mbox{ is }\mathbb F^{t,\Nt}\mbox{-predictable}\}$.
\end{theorem}

\begin{proof}
For shortness, we make an abuse of notation and omit $(t,n,\pi)$ from the probability measures and expectations. Let us start by proving (\ref{decomposition}) and finding $(\tilde{\mathbb P}^\alpha)_{\alpha\in\mathcal A_t^2}$ with the desired properties.  
Using integration by parts we can re-write the penalized P\&L (\ref{PnL}) as
\begin{equation}
\label{Pbar}
\begin{split}
P^{\alpha,s,x,n}_{t,T}& = x +sn + \int_t^T\left\{\mu\Nt_u -\frac{1}{2}\sigma^2(\gamma+\zeta)(\Nt_u)^2\right\}du - \ell(\Nt_T)\\
					 &  +\sigma\int_t^T \Nt_u dW_u +  \frac{1}{2}\sigma^2\gamma\int_t^T (\Nt_u)^2du +\int_t^T \delm_u d\Nm_u + \delp_u d\Np_u \\
					& \backdefeq x+ sn + \overline P_{t,T}^{\alpha,n},
\end{split}
\end{equation}
Consider first the case $\gamma=0$. The integrals with respect to $W,\Nm,\Np$ all have bounded integrands (and predictable for $\Nm,\Np$), except in the case of unconstrained inventory: $\Nmax=+\infty$ and $\gamma=\zeta=0\equiv\ell$ (see (\ref{problem})). Regardless, we still have
$
\mEa\left[\int_t^T (\Nt_u)^2 du\right]=\int_t^T \mEa\left[(\Nt_u)^2\right] du<+\infty\mbox{ by (\ref{bounded_moments})}.
$
Choosing $\tilde{\mathbb P}^\alpha\defeq\mPa$ the conclusion follows by taking expectation and by Propositions \ref{observable_intensities} and \ref{proposition_Pi}. 

Consider now $\gamma>0$. Hence, we are in the case $|N|\leq\Nmax<+\infty$. We define 
$$
A^\alpha\defeq\mathcal E\left( -\gamma\sigma\int_t^\cdot \Nt_udW_u\right)=\exp\left(-\gamma\sigma\int_t^\cdot \Nt_ud W_u-  \frac{1}{2}\sigma^2\gamma^2\int_t^\cdot (\Nt_u)^2du \right).
$$
By Novikov's condition, $A^\alpha$ is a strictly positive UI martingale with $\mEa[A_T]=1$, and therefore defines an equivalent probability measure $\mathbb A^\alpha\sim\mPa$ via $\frac{d\mathbb A^\alpha}{d\mPa}=A^\alpha_T$. Note that the Girsanov--Meyer Theorem. \cite[p.132 Thm.35]{P} ensures the $\mathbb F^{t,W,N}$-intensities of $\Npm - \Npm_{t^-}$ remain the same when changing to $\mathbb A^\alpha$. Let us set 
\begin{equation*}
\begin{split}
B^\alpha & \defeq\mathcal E\left(-\gamma\int_t^\cdot U_\gamma(\delm_u)d\overline{\Nm_u}^\alpha + U_\gamma(\delp_u)d\overline{\Np_u}^\alpha \right)\\
				 & =\mathcal E\left(\int_t^\cdot \big(e^{-\gamma\delm_u} -1\big)d\overline{\Nm_u}^\alpha + \big(e^{-\gamma\delp_u} -1\big)d\overline{\Np_u}^\alpha\right),
\end{split}
\end{equation*}
where $\overline{\Npm_u}^\alpha$ denote the corresponding $(\mPa,\mathbb F^{t,W,N})$-compensated (or equivalently, $(\mathbb A^\alpha,\mathbb F^{t,W,N})$-compensated) processes. By the same arguments of Propositions \ref{Girsanov1} and \ref{Girsanov2}, $B^\alpha$ is a strictly positive UI martingale with $\mathbb E^{\mathbb A^\alpha}[B^\alpha_T]=1$ and defines an equivalent probability measure $\tilde{\mathbb P}^\alpha\sim\mathbb A^\alpha\sim\mPa$ via $\frac{d\tilde{\mathbb P}^\alpha}{d\mathbb A^\alpha}=B^\alpha_T$, such that \textit{(ii)} holds true. Note that \textit{(i)} is also trivially verified due to the equivalence of the probability measures. 

Suppose for the time being that $\Psi$ is defined as in (\ref{Psi}) but taking supremum over the whole set of admissible controls $\mathcal A_t^2\supseteq \tilde{\mathcal A}_t^2$ instead. We will see afterwards that this makes no difference. To see (\ref{decomposition}), observe that the identity $U_\gamma(a+b)=U_\gamma(b)e^{-\gamma a}+U_\gamma(a)$ and (\ref{Pbar}) yield 
$
U_\gamma(P^{\alpha,s,x,n}_{t,T}) = U_\gamma(\overline P^{\alpha,n}_{t,T})e^{-\gamma(x+sn)} + U_\gamma(x+sn),
$
giving
$$
V(t,s,x,n,\pi)= \sup_{\alpha\in\mathcal A_t^2}\mEa\left[ U_\gamma(\overline P^{\alpha,n,\pi}_{t,T})\right]e^{-\gamma(x+sn)} + U_\gamma(x+sn).
$$
On the other hand, by the same identity, 
$$
U_\gamma(x+sn+\Theta)= (U_\gamma\circ\Theta) e^{-\gamma(x+sn)} + U_\gamma(x+sn)= \Psi e^{-\gamma(x+sn)} + U_\gamma(x+sn).
$$
As a consequence, (\ref{decomposition}) is equivalent to the equality $\Psi(t,n,\pi)=\sup_{\alpha\in\mathcal A_t^2}\mEa\left[ U_\gamma(\overline P^{\alpha,n}_{t,T})\right]$. We check instead the stronger statement
\begin{equation}
\label{aux}
\mEa\left[\exp\Big(-\gamma\overline P^{\alpha,n}_{t,T} \Big) \right]= \tilde{\mathbb E}^{\alpha}\left[\exp\Big(-\gamma \tilde P_{t,T}^{\alpha,n} \Big) \right],\mbox{ for all }\alpha\in\mathcal A_t^2.
\end{equation} 

Using the explicit exponential formula (see equation (\ref{exponential_explicit})) and by straightforward computations:
\begin{equation*}
\begin{split}
\exp\Big(-\gamma\overline P^{\alpha,n}_{t,T} \Big) & = \exp\left(-\gamma \Big(\int_t^T\left\{\mu\Nt_u -\frac{1}{2}\sigma^2(\gamma+\zeta)(\Nt_u)^2\right\}du - \ell(\Nt_T) \Big)\right)\\
& \times A^\alpha_T\exp\left(-\gamma\int_t^T \delm_u d\Nm_u + \delp_u d\Np_u\right)\\
& = \exp\Big(-\gamma \tilde P_{t,T}^{\alpha,n} \Big)A^\alpha_T
\exp\left(\gamma\int_t^T\big\{ U_\gamma(\delm_u) \hlmtu + U_\gamma(\delp_u) \hlptu\big\}du\right)\\
& \times \prod_{\substack{t\leq u \leq T:\\ \Delta\Nm_u\neq 0}}\exp\left(-\gamma\delm_u \right)\prod_{\substack{t\leq u \leq T:\\ \Delta\Np_u\neq 0}}\exp\left(-\gamma\delp_u \right) \\
& = \exp\Big(-\gamma \tilde P_{t,T}^{\alpha,n}\Big)A^\alpha_T B^\alpha_T,
\end{split}
\end{equation*}
which yields (\ref{aux}) after taking $\mPa$-expectation. 

It remains to see that 
$$\Psi(t,n,\pi)\defeq\sup_{\alpha\in\mathcal A^2_t}\tilde{\mathbb E}^{\alpha}\left[U_\gamma\Big(\tilde P_{t,T}^{\alpha,n,\pi}\Big)\right] = \sup_{\alpha\in\tilde{\mathcal A}^2_t}\tilde{\mathbb E}^{\alpha}\left[U_\gamma\Big(\tilde P_{t,T}^{\alpha,n,\pi}\Big)\right]\backdefeq \tilde{\Psi}(t,n,\pi).$$
Clearly $\Psi\geq \tilde{\Psi}$. Let us check $\Psi\leq\tilde{\Psi}$. As done in Proposition \ref{Girsanov2} we can define a family of `reference' equivalent probability measures $\tilde{\mathbb Q}^\alpha\sim\tilde{\mathbb P}^\alpha$, $\alpha\in\mathcal A^2_t$, such that for $(\tilde{\mathbb Q}^\alpha,\mathbb F^{t,W,N})$ it holds: $W - W_t$ is a Wiener process independent of the counting process $(\Nm-\Nm_{t^-},\Np-\Np_{t^-})$ and $\Npm-\Npm_{t^-}$ has predictable intensity $\mathbbm 1^{\pm,t,n}_u\defeq\mathbbm 1_{\{\mp \Nt_{u^-}<\Nmax\}}$ (in particular, its law does not depend on $\alpha$). Furthermore, the inverse change of measure is given by $d\tilde{\mathbb P}^\alpha/d\tilde{\mathbb Q}^\alpha=\tilde{Z}^\alpha_T$ with
\begin{equation*}
\begin{split}
& \tilde{Z}^\alpha\defeq\\
& \mathcal E\left(\int_t^\cdot \big(\widehat{\Lm}(\pit_{u^-},\delm_u)-1\big)\big(d\Nm_u- \mathbbm 1^{-,t,n}_u du\big)+\big(\widehat{\Lp}(\pit_{u^-},\delp_u)-1\big)\big(d\Np_u- \mathbbm 1^{+,t,n}_u du\big)\right).
\end{split}
\end{equation*}
Let us fix $\alpha\in\mathcal A_t^2$. Denote by $\mathcal D=\mathcal D([t,T],\mathbb R)$ the Skorokhod space of c\`adl\`ag functions with its usual sigma algebra and by $\mathbb P^W, \mathbb P^N$ the laws (or pushforward measures) induced on $\mathcal D$ by $W-W_t,\Nt$ resp. when starting from $\mathbb Q^\alpha$. These laws do not depend on $\alpha$ and characterize the joint law of $(W-W_{t},\Nt)$ on $\mathcal D^2$ as $\mathbb P^W\otimes\mathbb P^N$, due to the independence of the two processes. Since $\alpha$ is $\mathbb F^{t,W,N}$-predictable, by a monotone class argument one can show there exists a jointly measurable process $f:[t,T]\times\mathcal D^2\to\mathbb R$ such that $\alpha_u=f_u(W-W_t, \Nt)$ and for $\mathbb P^W$-almost every $w\in\mathcal D$, the process $\tilde{\alpha}_u\defeq f_u(w, \Nt)$ is in $\tilde{\mathcal A}_t^2$. Note also that we can write $\tilde{Z}^\alpha_T U_\gamma\Big(\tilde P_{t,T}^{\alpha,n,\pi}\Big)=g(\alpha, \Nt) = g(f(W-W_t,\Nt),\Nt)$ for some function $g$. By Fubini's theorem,
\begin{equation*}
\begin{split}
\tilde{\mathbb E}^{\alpha}\left[U_\gamma\Big(\tilde P_{t,T}^{\alpha,n,\pi}\Big)\right] & = \tilde{\mathbb E}^{\mathbb Q^\alpha}\left[\tilde{Z}^\alpha_T U_\gamma\Big(\tilde P_{t,T}^{\alpha,n,\pi}\Big)\right] = \int \mathbb E^{\mathbb P^N}\left[g(f(w,\cdot),\cdot)\right]d\mathbb P^W(w)\\
& \leq \int\tilde{\Psi} d\mathbb P^W(w)=\tilde{\Psi}.
\end{split} 
\end{equation*}
Since $\alpha\in\mathcal A_t^2$ was arbitrary, we conclude that $\Psi=\tilde{\Psi}$.
\end{proof}
 
Just as it occurs under full information (see Section \ref{s:full_info}), for a fully risk-neutral MM with negligible costs (i.e. $\Nmax=+\infty,\ \gamma=\zeta=0\equiv\ell$), $\Theta$ can be further decomposed. Note, from their definition, that in this case $\hlpmt$ and $\pit$ do not depend on $n$. As it was proved in Theorem \ref{main_theorem_1}, when $\gamma=0$ the family $(\tilde{\mathbb P}^{\alpha,t,n})$ can be taken as the original physical probabilities, and these do not depend on $n$ either. The following corollary is now immediate. 
\begin{corollary}
\label{coro_main_theorem_1}
If $\Nmax=+\infty$ and $\gamma=\zeta=0\equiv\ell$ then 
$$V(t,s,x,n,\pi) = x+sn+\mu n(T-t)+\Phi(t,\pi),$$ 
with $\Phi(t,\pi)=\sup_{\alpha\in\tilde{\mathcal A}_t^2}\mathbb E^{t,\pi}\left[\int_t^T \big\{\delm_u\widehat{\lm_u}^{\alpha,t,\pi} + \delp_u\widehat{\lp_u}^{\alpha,t,\pi} + \mu (\widehat{\lm_u}^{\alpha,t,\pi} - \widehat{\lp_u}^{\alpha,t,\pi})\big\} du\right].$
\end{corollary}

In the context of the previous corollary, formal substitution in (\ref{HJB_V}) or (\ref{HJB_Theta}) yields the following PIDE for $\Phi$:
\begin{equation}
\label{HJB_Phi}
\begin{split}
0 & = \phi_t + \sum_{i,j=1}^k q_t^{ji}\pi^j \phi_{\pi^i} \\
  & + \sup_{\delm\in\overline{(\delmin,\delmax)}}\Big\{\sum_{i,j=1}^k (\Ljm-\Lim)(\delm)\pi^j\pi^i \phi_{\pi^i} + \big(\delm+\mu(T-t)+\Dm_{\delm}(\phi)\big) \sum_{i=1}^k\pi^i\Lim(\delm)\Big\}\\
	& + \sup_{\delp\in\overline{(\delmin,\delmax)}}\Big\{\sum_{i,j=1}^k(\Ljp-\Lip)(\delp)\pi^j\pi^i \phi_{\pi^i}+ \big(\delp-\mu(T-t)+\Dp_{\delp}(\phi)\big) \sum_{i=1}^k\pi^i\Lip(\delp)\Big\},
\end{split}
\end{equation}
with terminal condition $\phi(T,\pi) = 0$, where:
\begin{equation*}
\begin{split}
\Dpm_{\delta}(\phi)(t,\pi) = \phi\Big(t, \frac{1}{\sum_{j=1}^k\pi^j\Ljpm(\delta)}\big(\pi^1\Lpm_1(\delta),\dots,\pi^k\Lpm_k(\delta)\big)\Big) - \phi(t,\pi).
\end{split}
\end{equation*}

We now want to prove that $\Theta$ (resp. $\Phi$) is the unique continuous viscosity solution of the terminal condition PIDE (\ref{HJB_Theta}) (resp. (\ref{HJB_Phi})). (See, e.g., \cite[Def.2.1]{Son} for the relevant definition, or more in general \cite[Def.7.3]{DF}, recalling that in our case we have no boundary conditions other than that at terminal time.) A complication inevitably arises, as classical viscosity techniques \cite{Bo,FS,OS} cannot be applied directly to weak formulation models such as ours. However, the decomposition of Theorem \ref{main_theorem_1} (resp. Corollary \ref{coro_main_theorem_1}) not only reduces the dimensionality of the problem, but also states that the MM may neglect the diffusion component of the state process altogether, focusing solely on the time-space state variable $(u,\Nt,\pit)$ (resp. $(u,\pit)$). This is a PDMP as introduced in \cite{D1} (detailed treatments also found in \cite{BR3, D2}). Using results from PDMPs theory, our continuous-time problem is identified with a control problem for a discrete-time Markov decision model, as in \cite{BR1,BR2,BR3,CEFS}, and linked again with viscosity solutions of HJB PIDEs as in \cite{CEFS,DF}. 

An inevitable drawback is that the PDMPs approach relies on the use of the so-called \textit{randomized} (or \textit{relaxed}) controls and requires the control space to be compact. Hence, for the following theorem we will assume $-\infty<\delmin<\delmax<+\infty$. Assuming a uniform lower constraint $\delmin>-\infty$ is hardly a problem. On the contrary, $\delmin=0$ (or even some small positive number) is the most meaningful in practice, as negative spreads imply the MM is willing to offer her clients better prices than the reference price $S$. ($\delmin=-\infty$ is motivated in the literature by mathematical convenience rather than modelling accuracy.) A uniform upper bound $\delmax<+\infty$, on the other hand, is harder to assess a priori. Fortunately, in most situations encountered in practice, the unconstrained optimization will yield bounded optimal spreads nonetheless, and the MM can dispense with $\delmin,\delmax$ if she wishes to do so (see Section \ref{s:numerics} for an example). 

\begin{theorem}
\label{main_theorem_2}
Assume $-\infty<\delmin<\delmax<\infty$. For $\Nmax<\infty$ (resp. $\Nmax=+\infty$ and $\gamma=\zeta=0\equiv\ell$), let $\Theta$ (resp. $\Phi$) be as in Theorem \ref{main_theorem_1} (resp. Corollary \ref{coro_main_theorem_1}). Then $\Theta$ (resp. $\Phi$) is the unique continuous viscosity solution of the terminal condition PIDE (\ref{HJB_Theta}) (resp. (\ref{HJB_Phi})). 
\end{theorem}

\begin{proof}
\textit{Case $\Nmax<+\infty$ and $\gamma>0$}: Let us write $\Psi=\frac{1 + \Upsilon}{\gamma}$, with 
\begin{equation*}
\Upsilon(t,n,\pi) = \sup_{\alpha\in\tilde{\mathcal A}_t^2}\tilde{\mathbb E}^{\alpha,t,n,\pi}\left[ 
\exp\Big(-\gamma\big(\tilde P_{t,T}^{\alpha,n,\pi}+\ell(\Nt_T)\big)\Big) \Big(-\exp\big(\gamma\ell(\Nt_T)\big) \Big)\right].
\end{equation*}
Then $\Upsilon$ can be regarded as the value function of an optimization problem in the standard Bolza-Lagrange formulation, i.e., 
$$
\Upsilon(a)= \sup_{\alpha\in\tilde{\mathcal A}_t^2}\tilde{\mathbb E}^{\alpha,a}\left[\int_t^T D^{\alpha,a}_{t,u}f(u,A^{\alpha,a}_u,\alpha_u) du + D^{\alpha,a}_{t,T}g(A^{\alpha,a}_T)\right],
$$
where the state variable $A^{\alpha,a}_u=(u,\Nt_u,\pit_u)$ is a PDMP with bounded state space $[0,T]\times(\mathbb Z\cap[-\Nmax,\Nmax])\times\Delta^\circ$ and initial condition $a=(t,n,\pi)$, $D^{\alpha,a}_{t,u}=\exp\Big(-\int_t^v\rho(u,A^{\alpha,a}_u,\alpha_u)du\Big)$ is the discount factor and $f,g,\rho$ are bounded functions. (These functions are bounded thanks to the control space being bounded.) The continuity of $\Upsilon$ can be proved now in the same way as in \cite[Thm.4.10]{CEFS} albeit in a more straightforward manner. This is due to the boundedness of $f,g,\rho$, the fact that $\pit$ never visits the relative border of the simplex $\Delta$, and that there is no exit time of the state space other than the terminal time. Assumptions \cite[Asm.4.7]{CEFS} are clearly verified in our model, and \cite[Asm.2.1]{CEFS} has already been accounted for at the beginning of the proof of Proposition \ref{proposition_Pi}.\footnote{An additional detail now is that \cite[Asm.2.1]{CEFS} is also used in \cite[Lem.4.1]{CEFS}, which states that the drift coefficient in equation (\ref{Pi_equivalent}) is Lipschitz in the state variable, uniform in time and control. This is routinely verified in our case, under our new assumption: $-\infty<\delmin<\delmax<\infty$.} We remark that in our case the bounding function (see \cite[Lem.4.6]{CEFS}) can be taken simply as $b(t,n,\pi)=\exp(\eta(T-t))$, for an $\eta>0$ large enough to prove contractiveness. 

Having proved the continuity, the same proof of \cite[Thm.5.3]{CEFS} (or \cite[Thm.7.5]{DF}) shows that $\Upsilon$ is the unique continuous viscosity solution of its standard HJB equation. We remark once again that our case is simpler, in that $\Upsilon$ is bounded and we do not have any boundary conditions other than the terminal time condition. In particular, there is no need for additional assumptions on the growth of $\Upsilon$. 

Finally, the result for $\Theta$ is obtained via the two increasing diffeomorphic transformations $\Psi=\frac{1 + \Upsilon}{\gamma}$ and $\Theta=U_\gamma^{-1}\circ\Psi$.

\textit{Case $\Nmax<+\infty$ and $\gamma=0$}: The only difference with the previous case is that the Bolza-Lagrange representation of the problem is obtained directly, since 
$$
\Theta(t,n,\pi)=\sup_{\alpha\in\tilde{\mathcal A}^2_t}\mathbb E^{\alpha,t,n,\pi}\left[\int_t^T\big\{\delm_u \hlmtu + \delp_u \hlptu + \mu \Nt_u -\frac{1}{2}\sigma^2\zeta(\Nt_u)^2\big\}du - \ell(\Nt_T) \right],
$$
with no need for any transformation.

\textit{Case $\Nmax=+\infty$ and $\gamma=\zeta=0\equiv\ell$}: 
The same as the latter case but working instead with the state variable $(u,\pit)$ and the value function $\Phi$. 
\end{proof}

\section{Full information}
\label{s:full_info}
\setcounter{assumptions}{0}
\setcounter{subsection}{0}

In this section we consider the idealized case of a MM with full information. We assume the MM has inside information in such a way that she can observe the full filtration $\mathbb F$, and in particular she can observe $Y$. For example, if $Y$ represents different levels of competition amongst 
liquidity providers, this would practically mean the MM has information regarding her competitors' quotes. We will see that in this case the value function turns out to be a regular, classical solution, of its HJB equation. Afterwards, we will compare the results with those obtained in the more realistic setting of partial information.

\subsection{Dimensionality reduction and the general system of ODEs}
We consider problem (\ref{problem}) but under full information, with the set of admissible spreads: 
$$\mathcal U\defeq\{\delta:[0,T]\to\overline{(\delmin,\delmax)}: \delta\mbox{ is }\mathbb F\mbox{-predictable and bounded from below}\}.$$
Note that in this section, and when $\delmax=+\infty$, we do not assume upper-boundedness of the spreads. Let $t\in [0,T],\ s,x\in\mathbb R$ and $(n,i)\in\mathcal I\defeq\left(\mathbb Z\cap [-\Nmax,\Nmax]\right)\times \{1,\dots,k\}$. Consider the processes $\St,\Xt,\Nt,P^{\alpha,s,x,n}$ as defined in Section \ref{s:viscosity} and $\Yt$ a Markov chain with deterministic generator matrix $Q$, state space $\{1,\dots,k\}$ and such that $\Yt_{t^-}=i$. We assume the physical probabilities $\mPtfull$ are defined for every $\alpha\in\mathcal U^2$ and that Assumptions \ref{assumptions_Q}, \ref{assumptions_intensities} and (\ref{no_common_jumps}) (starting at time $t^-$) are still in place. The value function in this case is 
\begin{equation}
\label{value_function_full_info}
V(t,s,x,n,i)\defeq \sup_{\alpha\in\mathcal U_t^2}\mathbb E^{t,n,i}\Big[U_\gamma\big(P^{\alpha,s,x,n}_{t,T}\big)\Big],
\end{equation}

By means of the Dynamic Programming Principle and Ito's Lemma one can formally derive the following HJB PIDE for $V$:

\begin{equation}
\label{HJB_V_full_info}
\begin{split}
0 & = v_t(t,s,x,n,i) + \mu v_s(t,s,x,n,i) +  \frac{1}{2}\sigma^2 v_{ss}(t,s,x,n,i)\\
  & + \frac{1}{2}\sigma^2\zeta n^2 (\gamma v(t,s,x,n,i) -1)+ \sum_{j=1}^k q_t^{ij}v(t,s,x,n,j)\\
  & + \mathbbm 1_{\{n<\Nmax\}}\sup_{\delm\in\overline{(\delmin,\delmax)}}\Lim(\delm)\Big(v\big(t,s,x -(s - \delm),n+1, i\big)- v(t,s,x,n,i)\Big)\\
	& + \mathbbm 1_{\{-n<\Nmax\}}\sup_{\delp\in\overline{(\delmin,\delmax)}}\Lip(\delp)\Big(v\big(t,s,x +(s + \delp),n-1, i\big)- v(t,s,x,n,i)\Big),
\end{split}
\end{equation}
with terminal condition $v(T,s,x,n,i) = U_\gamma(x+sn -\ell(n))$. 

In this new context, instead of formally proving a decomposition of $V$ as in Theorem \ref{main_theorem_1}, it is more straightforward to propose an ansatz and ultimately prove it valid with a verification theorem. (This is the standard approach used in the Avellaneda--Stoikov framework.) Let us consider an ansatz for the value function analogous to those used for the one regime case:
$$
V(t,s,x,n,i) = U_\gamma(x+sn + \Theta(t,n,i)),
$$
for some function $\Theta:[0,T]\times\mathcal I\to\mathbb R$, $C^1$ in time. Substituting in (\ref{HJB_V_full_info}) and using Assumptions \ref{assumptions_Q}, we see that $\Theta$ must satisfy a system of ODEs indexed in $(n,i)\in\mathcal I$:

\begin{equation}
\label{HJB_Theta_full_info}
\begin{split}
0 & = \theta_t(t,n,i) + \mu n - \frac{1}{2}\sigma^2 n^2(\zeta+\gamma)+ \sum_{j\neq i} q_t^{ij} U_\gamma\big(\theta(t,n,j)-\theta(t,n,i)\big)\\
  & + \mathbbm 1_{\{n<\Nmax\}}\Hm_i\big(\theta(t,n+1,i)-\theta(t,n,i)\big)+ \mathbbm 1_{\{-n<\Nmax\}}\Hp_i\big(\theta(t,n-1,i)-\theta(t,n,i)\big),
\end{split}
\end{equation}
with terminal condition $\theta(T,n,i) = -\ell(n)$, where:
\begin{enumerate}
\item $\Hpm_i(d)\defeq\sup_{\delta\in\overline{(\delmin,\delmax)}}\hpm_i(\delta, d)$, for $d\in\mathbb R$.
\item $\hpm_i(\delta, d)\defeq \Lipm(\delta)U_\gamma(\delta+ d)$, for $d\in\mathbb R,\ \delta\in\overline{(\delmin,\delmax)}$.
\end{enumerate}

The original problem is simplified in this way, both by the dimension of the state variable and by the complexity of the equations, provided we can show that problem (\ref{HJB_Theta_full_info}) admits a solution. With this aim in mind, let us prove first the following property of the Hamiltonian functions $\Hpm_1,\dots,\Hpm_k$. 

\begin{lemma}
\label{lipschitz_hamiltonians}
For all $1\leq i\leq k$, it holds:
\begin{enumerate}[label=(\roman*)]
\item For each compact $K\subset\mathbb R$, there exists $[a,b]\subseteq\overline{(\delmin,\delmax)}$ such that $\Hpm_i(d)=\max_{\delta\in [a,b]}\hpm_i(\delta, d)$, for all $d\in K$. 
\item $\Hpm_i$ is locally Lipschitz.
\end{enumerate}
\end{lemma}

\begin{proof}
Fix $1\leq i\leq k$ and let $K\subset\mathbb R$ be a compact set. We verify first that \textit{(i)} is a consequence of Assumptions \ref{assumptions_intensities}. Let $C>0$ such that $|d|\leq C$ for all $d\in K$. If $\delmin=-\infty$, then we can take any $a<\min\{-C,\delmax\}$ since $\hpm_i(\delta,d)<\hpm_i(a,d)$ for all $\delta<a$ and $d\in K$. On the other hand, if $\delmax=+\infty$, take some $c>\max\{\delmin,2C\}$. It holds that $\hpm_i(c,d)\geq\Lpm_i(c)U_\gamma(C)\backdefeq\varepsilon>0$ and we can choose $b>c$ such that $\hpm_i(\delta,d)\leq\hpm_i(\delta,C)<\varepsilon$ for all $\delta\geq b$ and $d\in K$. Replacing supremum by maximum is now immediate due to the continuity of $\hpm_i(\cdot,d)$ on $[a,b]$ for all $d$.

\textit{(ii)} is routinely verified using that the family $\{\hpm_i(\delta,\cdot)\}_{\delta\in[a,b]}$ is equi-Lipschitz on $K$. 
\end{proof}

We want to prove now that the Cauchy problem (\ref{HJB_Theta_full_info}) admits a unique global classical solution $\theta$ which is $C^1$ in time. To this purpose, we will treat the cases of the finite system ($\Nmax<\infty$) and infinite system ($\Nmax=\infty$) of equations separately.

\subsection{Constrained inventory ODEs}
For $\Nmax<\infty$ we are dealing with a finite system of ODEs. We know that under certain regularity conditions the Cauchy problem (\ref{HJB_Theta_full_info}) is guaranteed to have a classical solution on some neighbourhood $(\tau,T]\subset[0,T]$ of $T$. Nonetheless, it is not always the case that such a local solution can be extended to a global one on $[0,T]$. Following \cite{G, GL}, we start by proving a comparison principle for (\ref{HJB_Theta_full_info}) that will allow us, in particular, to show the existence of a global solution. The argument used is standard for comparison principles of HJB-type equations. 

\begin{proposition}[\textbf{Comparison Principle}]
\label{comparison_ODEs}
Let $I\subseteq [0,T]$ be an interval containing $T$ and let $\overline\theta,\underline\theta:I\times \mathcal I\to\mathbb R$ be classical ($C^1$ with respect to time) super- and sub-solutions resp. of (\ref{HJB_Theta_full_info}). That is,
\begin{equation}
\label{HJB_Theta_full_info_supersol}
\begin{split}
0 & \leq -\overline\theta_t(t,n,i) - \mu n + \frac{1}{2}\sigma^2 n^2(\zeta+\gamma) - \sum_{j\neq i} q_t^{ij} U_\gamma\big(\overline\theta(t,n,j)-\overline\theta(t,n,i)\big)\\
  & - \mathbbm 1_{\{n<\Nmax\}}\Hm_i\big(\overline\theta(t,n+1,i)-\overline\theta(t,n,i)\big) - \mathbbm 1_{\{-n<\Nmax\}}\Hp_i\big(\overline\theta(t,n-1,i)-\overline\theta(t,n,i)\big),
\end{split}
\end{equation} 
\begin{equation}
\label{HJB_Theta_full_info_subsol}
\begin{split}
0 & \geq -\underline\theta_t(t,n,i) - \mu n + \frac{1}{2}\sigma^2 n^2(\zeta+\gamma) - \sum_{j\neq i} q_t^{ij} U_\gamma\big(\underline\theta(t,n,j)-\underline\theta(t,n,i)\big)\\
  & - \mathbbm 1_{\{n<\Nmax\}}\Hm_i\big(\underline\theta(t,n+1,i)-\underline\theta(t,n,i)\big) - \mathbbm 1_{\{-n<\Nmax\}}\Hp_i\big(\underline\theta(t,n-1,i)-\underline\theta(t,n,i)\big)
\end{split}
\end{equation} 
and
\begin{equation}
\label{terminal_cond_comparison}
\underline\theta(T,\cdot,\cdot)\leq-\ell\leq\overline\theta(T,\cdot,\cdot).
\end{equation}
Then
$$ \underline\theta\leq\overline\theta.$$
\end{proposition}
\begin{proof}
Suppose first $I=[\tau,T]$ for some $0\leq\tau<T$ and let $\ve>0$. Since $\Nmax<\infty$, there exists $(t_\ve,n_\ve,i_\ve)\in[\tau,T]\times\mathcal I$ such that
\begin{equation}
\label{min}
\overline\theta(t_\ve,n_\ve,i_\ve)-\underline\theta(t_\ve,n_\ve,i_\ve)+\ve(T-t_\ve) = \min_
{(t,n,i)\in[\tau,T]\times\mathcal I}\overline\theta(t,n,i)-\underline\theta(t,n,i)+\ve(T-t).
\end{equation}
If $t_\ve<T$, then we must have
$$\overline\theta_t(t_\ve,n_\ve,i_\ve)-\underline\theta_t(t_\ve,n_\ve,i_\ve)\geq \ve.$$
Let us see that the left-hand side is non-positive. By (\ref{HJB_Theta_full_info_supersol}) and (\ref{HJB_Theta_full_info_subsol}),
\begin{equation*}
\begin{split}
& \overline\theta_t(t_\ve,n_\ve,i_\ve)-\underline\theta_t(t_\ve,n_\ve,i_\ve)\\
& \leq\sum_{j\neq i_\ve} q_t^{i_\ve j}\Big(U_\gamma\big(\underline\theta(t_\ve,n_\ve,j)-\underline\theta(t_\ve,n_\ve,i_\ve)\big)-U_\gamma\big(\overline\theta(t_\ve,n_\ve,j)-\overline\theta(t_\ve,n_\ve,i_\ve)\big)\Big)\\
& + \mathbbm 1_{\{n<\Nmax\}}\Big(\Hm_i\big(\underline\theta(t_\ve,n_\ve+1,i_\ve)-\underline\theta(t_\ve,n_\ve,i_\ve)\big)-\Hm_i\big(\overline\theta(t_\ve,n_\ve+1,i_\ve)-\overline\theta(t_\ve,n_\ve,i_\ve)\big)\Big)\\
& + \mathbbm 1_{\{-n<\Nmax\}}\Big(\Hp_i\big(\underline\theta(t_\ve,n_\ve-1,i_\ve)-\underline\theta(t_\ve,n_\ve,i_\ve)\big)-\Hp_i\big(\overline\theta(t_\ve,n_\ve-1,i_\ve)-\overline\theta(t_\ve,n_\ve,i_\ve)\big).
\end{split}
\end{equation*} 
$\Hpm_i$ increasing (resp. $U_\gamma$ increasing) and (\ref{min}) imply that the last two terms (resp. the first one) are non-positive. We must have then that $t_\ve=T$, and due to (\ref{min}) and (\ref{terminal_cond_comparison}) for all $(t,n,i)\in [0,T]\times\mathcal I$: 
\begin{align*}
\overline\theta(t,n,i)-\underline\theta(t,n,i)+\ve(T-t)&\geq\overline\theta(T,n_\ve,i_\ve)-\underline\theta(T,n_\ve,i_\ve)+\ve(T-T)\geq 0\\
																			\overline\theta(t,n,i)&\geq	\underline\theta(t,n,i)-\ve(T-t).
\end{align*}
Since $\ve>0$ was arbitrary, we obtain the desired result. 

The case $I=(\tau,T]$ is now a consequence of comparing $\underline\theta$ and $\overline\theta$ on intervals of the form $[t_n,T]\subseteq (\tau,T]$ with $t_n\searrow\tau$.
\end{proof}
We can now prove the existence and uniqueness of a classical global solution of the Cauchy problem (\ref{HJB_Theta_full_info}). 

\begin{theorem}
\label{Theta_constrained_inventory}
There exists a unique $\Theta:[0,T]\times\mathcal I\to\mathbb R$, $C^1$ in time, which (classically) solves the Cauchy problem (\ref{HJB_Theta_full_info}).
\end{theorem}

\begin{proof}
For each $1\leq i\leq k$, Lemma \ref{lipschitz_hamiltonians} tells us that $\Hpm_i:\mathbb R\to\mathbb R$ is locally Lipschitz. We also assumed $Q:[0,T]\to\mathbb R^{k\times k}$ continuous (Assumptions \ref{assumptions_Q}). Hence, by the Cauchy-Lipschitz Theorem, the terminal condition system of ODEs (\ref{HJB_Theta_full_info}) admits a unique $C^1$ local solution $(\Theta(\cdot,n,i))_{(n,i)\in\mathcal I}$, defined on some maximal interval $I\subset [0,T]$ containing $T$.

Suppose that $I\subsetneq [0,T]$. Then $I=(\tau,T]$ for some $0\leq\tau<T$ and $\|\Theta(t)\|\to\infty$ as $t\searrow\tau$. We claim that $\Theta$ is actually bounded, resulting in a contradiction. Indeed, take 
$$\overline K=\max_{(i,n)\in\mathcal I}\left\{\Big|\mu n -\frac{1}{2}\sigma^2 n^2(\zeta+\gamma)+\mathbbm 1_{\{n<\Nmax\}}\Hm_i(0)+ \mathbbm 1_{\{-n<\Nmax\}}\Hp_i(0)\Big|\right\}$$
and
$$
\underline K=\ell(\Nmax),
$$
and define $\overline\theta:I\times\mathcal I\to\mathbb R$ by $\overline\theta(t,n,i)=\overline K(T-t)$ and $\underline\theta\defeq -\overline\theta -\underline K$. Then $\overline\theta$ (resp. $\underline\theta$) is a super- (resp. sub-) solution of the Cauchy problem (\ref{HJB_Theta_full_info}). By the Comparison Principle (\ref{comparison_ODEs}), 
$$
-\overline K T - \underline K\leq\underline\theta\leq\Theta\leq\overline\theta\leq\overline K T,
$$
proving that $\Theta$ is bounded.
\end{proof}

\subsection{Unconstrained inventory ODEs} 
We consider now $\Nmax=+\infty$, for which (\ref{HJB_Theta_full_info}) becomes an infinite system of ODEs. Recall that in this case we assumed $\gamma=\zeta=0\equiv\ell$, i.e., the MM is fully risk-neutral and has negligible costs. This allows us to further reduce the dimensionality of the state variable by the additional ansatz
\begin{equation}
\label{reduction_Gamma}
\Theta(t,n,i)=\mu n(T-t)+\Phi_i(t),
\end{equation}
for some $\Phi=(\Phi_i)_{i=1}^k\in C^1([0,T],\mathbb R^k)$. Substituting in (\ref{HJB_Theta_full_info}), we get that $\Phi$ must solve the finite linear system of ODEs
\begin{equation}
\label{HJB_Phi_full_info}
\begin{cases}
\phi'(t) =-Q(t)\phi(t) + b(t)\\
\phi(T) =0,
\end{cases}
\end{equation}
with $b_i(t)= - \Hm_i(\mu(T-t))-\Hp_i(-\mu(T-t)),\ i=1,\dots,k$. 
By continuity of $Q$ and $b$ (see Lemma \ref{lipschitz_hamiltonians}), the previous system is known to have a unique global solution $\Phi\in C^1([0,T],\mathbb R^k)$ which can be computed by the variation of parameters method. Straightforward verification now gives the following:
\begin{proposition}
\label{Theta_unconstrained_inventory}
Let $\Phi\in C^1([0,T],\mathbb R^k)$ be the unique solution of (\ref{HJB_Phi_full_info}). Then $\Theta:[0,T]\times \mathcal I\to\mathbb R$ such that $\Theta(t,n,i)=\mu n(T-t)+\Phi_i(t)$, is the unique, $C^1$ in time, (classical) solution of the Cauchy problem (\ref{HJB_Theta_full_info}) with $\Nmax=+\infty$ and $\gamma=\zeta=0\equiv\ell$.
\end{proposition}

\subsection{General Verification Theorem}
We give now the complete solution for the general model under full information. For the next theorem we note that given the function $\Theta$ defined in Theorem \ref{Theta_constrained_inventory} for $\Nmax<+\infty$ (resp. Proposition \ref{Theta_unconstrained_inventory} for $\Nmax=+\infty$ and $\gamma=\zeta=0\equiv\ell$) the difference or `jump' terms of equation (\ref{HJB_Theta_full_info}) are bounded by continuity. That is, $\Theta(t,n\mp 1,i)-\Theta(t,n,i)$ is bounded on $[0,T]\times\mathcal I$ for $\Nmax<+\infty$ (resp. $\Theta(t,n\mp 1,i)-\Theta(t,n,i)=\mp\mu(T-t)$ is bounded on $[0,T]$).

\begin{theorem}[\textbf{Verification Theorem}]
\label{verif_full_info}
Let $\Theta$ be as in Theorem \ref{Theta_constrained_inventory} for $\Nmax<+\infty$ and as in Proposition \ref{Theta_unconstrained_inventory} for $\Nmax=+\infty$ and $\gamma=\zeta=0\equiv\ell$. Then the value function in (\ref{value_function_full_info}) is 
$$
V(t,s,x,n,i)=U_\gamma(x+sn+\Theta(t,n,i)).
$$
Furthermore, given $C>0$ such that $|\Theta(t,n\mp 1,i)-\Theta(t,n,i)|\leq C$ for all $(t,n,i)\in[0,T]\times\mathcal I$ with $-\Nmax\leq n\mp 1\leq\Nmax$, there exist Borel measurable functions $\overline{\delpm_1},\dots,\overline{\delpm_k}:[-C,C]\to\overline{(\delmin,\delmax)}$ such that $\overline{\delpm_i}(d)\in\argmax_{\delta\in\overline{(\delmin,\delmax)}}\hpm_i(\delta,d)$ for all $d\in[-C,C]$, $1\leq i\leq k$; and for any such functions, the strategy $(\overline{\delp_u},\overline{\delm_u})$, with
$$
\overline{\delpm_u}\defeq\overline\delpm_{\Yt_{u^-}}\Big(\Theta(u,\Nt_{u^-}\mp 1,\Yt_{u^-})-\Theta(u,\Nt_{u^-},\Yt_{u^-})\Big),
$$
is optimal. 
\end{theorem}
\begin{notation*}
Note that we make a slight abuse of notation, writing $\overline{\delpm_i}$ for the Borel functions in the theorem and $(\overline{\delpm_u})$ for the spread processes. 
\end{notation*}

\begin{proof}
Let $1\leq i\leq k$. We check first that we can choose a maximizer of $\hpm_i(\cdot,d)$ in a measurable way with respect to $d$. Lemma \ref{lipschitz_hamiltonians} tells us that there exists $[a_i,b_i]\subseteq\overline{(\delmin,\delmax)}$ such that $\hpm_i(\delta,d)$ attains its maximum in $[a_i,b_i]$ for all $d\in[-C,C]$. 
Due to Assumptions \ref{assumptions_intensities} \textit{(\ref{assumptions_intensities_Lambda})}, $\hpm_i$ is continuous and, in particular, a Carath\'eodory function \cite[Def.4.50]{AB}. Thus, the Measurable Maximum Theorem \cite[Thm.18.19]{AB} guarantees the existence of a Borel selector of the $\argmax$, $\overline{\delpm_i}:[-C,C]\to[a_i,b_i]\subseteq\overline{(\delmin,\delmax)}$. (Note that the weak measurability assumption is trivially verified in this case.)

From here onwards, let $\overline{\delpm_1},\dots,\overline{\delpm_k}:[-C,C]\to\overline{(\delmin,\delmax)}$ be some Borel selectors as above. By Lemma \ref{lipschitz_hamiltonians} again, these functions must be bounded from below, and the spread processes $(\overline{\delpm_u)}$ defined as in the theorem are clearly admissible.

We define now $\tilde V(t,s,x,n,i)\defeq U_\gamma(x+ns+\Theta(t,n,i))$ and we want to show that $\tilde V=V$.
Let us fix the initial time and values, $t\in[0,T],\ s,x\in\mathbb R\mbox{ and }(n,i)\in\mathcal I$, and consider an arbitrary strategy $\alpha=(\delm_u,\delp_u)\in\mathcal U_t^2$. For shortness, we omit these initial conditions and strategy from the notation of the processes and the expectation. We denote by $\mathcal S\defeq \{j-l:\ 1\leq l,j\leq k,\ l\neq j\}$, the set of jump heights of $Y$, and
$$
R_{t,v}\defeq -\frac{1}{2}\sigma^2\zeta \int_t^v N_u^2 du,\qquad Z_{t,v}\defeq \exp\big(-\gamma R_{t,v}\big).
$$
By the identity $U_\gamma(a+b)=U_\gamma(a)e^{-\gamma b}+U_\gamma(b)$, we can rewrite the utility of the MM's penalized P\&L as 
\begin{equation}
\begin{split}
\label{eq1}
U_\gamma\big(P_{t,T}\big) & = U_\gamma\left(X_T+S_TN_T-\ell(N_T)\right) Z_{t,T} + U_\gamma\big( R_{t,T}\big)\\
& = \tilde V\big(T,S_T,X_T, N_T,Y_T\big) Z_{t,T}+U_\gamma\big( R_{t,T}\big).
\end{split}
\end{equation}
Using integration by parts and Ito's Lemma (recalling (\ref{no_common_jumps})), we re-express the last two terms as 
\begin{align*}
U_\gamma\big( R_{t,T}\big) = -\frac{1}{2}\sigma^2\zeta\int_t^T Z_{t,u}N_u^2du
\end{align*}
and
\begin{align*}
& \qquad\tilde V\big(T,S_T,X_T, N_T,Y_T\big) Z_{t,T}\\
& =\tilde V(t,s,x,n,i)+\int_t^T \tilde V (u,S_u,X_u, N_u,Y_u )d Z_{t,u} + \int_t^T Z_{t,u}d\tilde V (u,S_u,X_u, N_u,Y_u )\\
& = \tilde V(t,s,x,n,i)+\frac{1}{2}\sigma^2\zeta\gamma\int_t^T  V (u,S_u,X_u, N_u,Y_u ) Z_{t,u} N_u^2du+\int_t^T  Z_{t,u}\tilde V_t (u, S_u, X_{u^-}, N_{u^-},Y_{u^-} ) du\\
& + \mu \int_t^T  Z_{t,u}\tilde V_s (u, S_u, X_{u^-}, N_{u^-},Y_{u^-} )du+ \sigma \int_t^T  Z_{t,u}\tilde V_s (u, S_u, X_{u^-}, N_{u^-},Y_{u^-} ) dW_u\\
& + \frac{1}{2}\sigma^2\int_t^T  Z_{t,u}\tilde V_{ss} (u, S_u, X_{u^-}, N_{u^-},Y_{u^-} ) du\\
& +\int_t^T  Z_{t,u}\big( \tilde V(u, S_u, X_{u^-} - (S_u - \delm_u), N_{u^-}+1,Y_{u^-}) - \tilde V(u, S_u, X_{u^-}, N_{u^-},Y_{u^-})\big) d\Nm_u\\
& + \int_t^T  Z_{t,u}\big( \tilde V(u, S_u, X_{u^-} + S_u + \delp_u, N_{u^-}-1,Y_{u^-}) - \tilde V(u, S_u, X_{u^-}, N_{u^-},Y_{u^-})\big) d\Np_u \\
& + \int_t^T Z_{t,u}\int_{\mathcal S}\big( \tilde V (u, S_u, X_{u^-}, N_{u^-},Y_{u^-}+h ) - \tilde V (u, S_u, X_{u^-}, N_{u^-},Y_{u^-} ) \big)\mu^Y(dt,dh),
\end{align*}
where $\mu^Y$ is the jump measure of $Y$. Substituting in (\ref{eq1}),
\begin{equation}
\label{eq2}
\begin{split}
& U_\gamma\big(P_{t,T}\big) = \tilde V(t,s,x,n,i) + \sigma \int_t^T  Z_{t,u}\tilde V_s (u, S_u, X_{u^-}, N_{u^-},Y_{u^-} ) dW_u\\
&+\int_t^T  Z_{t,u} \Big\{\tilde V_t (u, S_u, X_{u^-}, N_{u^-},Y_{u^-} )du+\mu \tilde V_s (u, S_u, X_{u^-}, N_{u^-},Y_{u^-} )du\\
&+\frac{1}{2}\sigma^2\tilde V_{ss} (u, S_u, X_{u^-}, N_{u^-},Y_{u^-} )du + \frac{1}{2}\sigma^2\zeta N_u^2
\big(\gamma V(u,S_u,X_u, N_u,Y_u )-1\big)du\\
& +\int_{\mathcal S}\big( \tilde V(u, S_u, X_{u^-}, N_{u^-},Y_{u^-}+h) - \tilde V(u, S_u, X_{u^-}, N_{u^-},Y_{u^-}) \big)\mu^Y(dt,dh)\\
& +\big( \tilde V(u, S_u, X_{u^-} - (S_u - \delm_u), N_{u^-}+1,Y_{u^-}) - \tilde V(u, S_u, X_{u^-}, N_{u^-},Y_{u^-})\big) d\Nm_u \\
& + \big( \tilde V(u, S_u, X_{u^-} + S_u + \delp_u, N_{u^-}-1,Y_{u^-}) - \tilde V(u, S_u, X_{u^-}, N_{u^-},Y_{u^-})\big) d\Np_u\Big\}.
\end{split}
\end{equation}
Next, we want to verify that the process $\Big(\int_t^v  Z_{t,u}\tilde V_s (u, S_u, X_{u^-}, N_{u^-},Y_{u^-} ) dW_u\Big)_{t\leq v\leq T}$ is a zero mean martingale. Since $Z_{t,u}$ is bounded (either $\Nmax<\infty$ or $\gamma=0$) it suffices to check
\begin{equation}
\label{eq3}
\int_t^T \mathbb E\left[ \tilde V_s (u, S_u, X_u, N_u,Y_u )^2\right] du<+\infty.
\end{equation}
If $\gamma=0$, 
$$
\tilde V_s (u, S_u, X_u, N_u,Y_u )^2={N_u}^2
$$
and (\ref{eq3}) is a consequence of (\ref{bounded_moments}). 
\bigskip

\noindent If $\gamma\neq 0$ (hence $\Nmax<\infty$), let $C$ be a lower bound for $(\delm_u)$ and $(\delp_u)$. Integration by parts and H\"older inequality yield
\begin{align*}
&\mathbb E\left[\tilde V_s (u, S_u, X_u, N_u,Y_u )^2\right] =\gamma^2\mathbb E\left[\exp\big(-2\gamma\big(X_u+S_uN_u + \theta(u,N_u,Y_u)\big)\big)N_u^2 \right]\\
& =\gamma^2\exp(-2\gamma(x+sn))\\
&\times\mathbb E\left[\exp\Big(-2\gamma\int_t^u\delm_wd\Nm_w+ \delp_wd\Np_w +\mu N_wdw+\sigma N_wdW_w+\theta(w,N_w,Y_w)\Big) N_u^2 \right]\\
& \leq \gamma^2{\Nmax}^2\exp\Big(2\gamma\big(x+sn+\mu\Nmax (T-t)+\|\theta\|_{L^\infty([t,T]\times\mathcal I)}+2\sigma^2\gamma{\Nmax}^2(T-t)\big)\Big)\\
&\times{\mathbb E\left[\exp\Big(-4\gamma  C  \big(\Nm_T+\Np_T-\Nm_{t^-}-\Np_{t^-}\big)\Big)\right]}^{\frac{1}{2}}{\mathbb E\left[\exp\Big(-4\gamma\sigma\int_t^T N_u dW_u - 8\gamma^2\sigma^2\int_t^T N_u^2 du\Big)\right]}^{\frac{1}{2}}
\end{align*}
and (\ref{eq3}) is again consequence of (\ref{bounded_moments}) and Novikov's condition (trivially satisfied).
\bigskip

\noindent In the same way, recalling that $(Z_{t,u})$ is bounded, $(\delpm_u)$ is bounded from below and that $Q$ is bounded by continuity, one can check that 
$$
\int_t^T \mathbb E\Big[ Z_{t,u}\big| \tilde V(u, S_u, X_{u^-} \pm (S_u \pm \delpm_u), N_{u^-}\mp 1,Y_{u^-}) - \tilde V(u, S_u, X_{u^-}, N_{u^-},Y_{u^-})\big|\lpm_u\Big] du<+\infty
$$
and 
$$
\int_t^T \mathbb E\Big[Z_{t,u}\sum_{j\neq Y_{u^-}}q_t^{Y_{u^-},j}\big| \tilde V(u, S_u, X_{u^-}, N_{u^-},j) - \tilde V(u, S_u, X_{u^-}, N_{u^-},Y_{u^-}) \big|\Big] du<+\infty.
$$

\noindent Taking expectation in (\ref{eq2}), the Brownian term vanishes and integration with respect to $d\Nm,d\Np$ and $d\mu^Y$ is replaced by integration with respect to their dual predictable projections (see, e.g., \cite[p.27 T8 and p.235 C4]{B}). That is,
 
\begin{equation*}
\label{eq4}
\begin{split}
& \mathbb E\left[U_\gamma\big(P_{t,T}\big)\right] = \tilde V(t,s,x,n,i)+ \mathbb \int_t^T \mathbb E\Big[ Z_{t,u} \Big\{\tilde V_t(u, S_u, X_{u^-}, N_{u^-},Y_{u^-} )\\
& +\mu \tilde V_s (u, S_u, X_{u^-}, N_{u^-},Y_{u^-} ) + \frac{1}{2}\sigma^2\tilde V_{ss} (u, S_u, X_{u^-}, N_{u^-},Y_{u^-} )\\
& + \frac{1}{2}\sigma^2\zeta N_u^2\big(\gamma \tilde V(u,S_u,X_u, N_u,Y_u)-1\big)+\sum_{j=1}^k q^{Y_{u^-}j}_t \tilde V(u, S_u, X_{u^-}, N_{u^-},j )\\
& +\mathbbm 1^-_u\Lm_{Y_{u^-}}(\delm_u)\big( \tilde V(u, S_u, X_{u^-} - (S_u - \delm_u), N_{u^-}+1,Y_{u^-}) - \tilde V(u, S_u, X_{u^-}, N_{u^-},Y_{u^-})\big)\\
& +\mathbbm 1^+_u\Lp_{Y_{u^-}}(\delp_u)\big( \tilde V(u, S_u, X_{u^-} + S_u + \delp_u, N_{u^-}-1,Y_{u^-}) - \tilde V(u, S_u, X_{u^-}, N_{u^-},Y_{u^-})\big) \Big\}\Big] du\\
& \leq  \tilde V(t,s,x,n,i),
\end{split}
\end{equation*}
as $\tilde V$ solves (\ref{HJB_V_full_info}), with the equality attained for $(\delm_u)=(\overline{\delm_u}) \mbox{ and } (\delp_u)=(\overline{\delp_u})$  by definition.
We conclude that $V=\tilde V$ and that the pair $(\overline{\delm_u}),(\overline{\delp_u})$ is optimal.
\end{proof}

\begin{remark*}
For $\Nmax<\infty$, since the MM will not buy (resp. sell) whenever $(N_{u^-})$ hits $\Nmax$ (resp. $-\Nmax$), the value of $(\overline{\delm_u})$ (resp. $(\overline{\delp_u})$) at these stopping times is essentially irrelevant. From a strict mathematical perspective, the only constraint is that whichever the value we choose, the process needs to remain admissible.  
\end{remark*}

\subsection{Computing the optimal spreads: some particular cases}
\label{s:computing_spreads}
We have shown in Theorem \ref{verif_full_info} that the optimal spreads for the full information problem (\ref{value_function_full_info}) can be computed in feedback form in terms of $(t,n,i)=(t,N_{t^-},Y_{t^-})$ at each time $t\in[0,T]$. Practically, this means finding $\Theta$ (typically, numerically) that solves the terminal condition system of ODEs (\ref{HJB_Theta_full_info}), and finding the spreads by maximization:
\begin{equation} 
\label{eq1_argmax}
\overline{\delpm}(t,n,i)\in\argmax_{\delta\in\overline{(\delmin,\delmax)}}\Lipm(\delta)U_\gamma\left(\delta +\Theta(t,n\mp 1,i)-\Theta(t,n,i)\right).
\end{equation}
(See the proof of Lemma \ref{lipschitz_hamiltonians} to see how to reduce the maximization to a compact domain in the cases of $\delmin=-\infty$ or $\delmax=+\infty$.)
The functions in (\ref{eq1_argmax}) may admit multiple maximizers in general. In \cite{BL, G, GL} stronger assumptions are imposed on the orders intensities which guarantee, in particular, the uniqueness of the maximizers. 

We now extend these assumptions to our context and give the corresponding characterization of the spreads. Henceforth, Assumption \ref{assumptions_intensities} \textit{(\ref{assumptions_intensities_Lambda})} is replaced by the following:

\begin{assumptions}
\label{assumptions_intensities_strong}
$\Lipm\in C^2\big(\overline{(\delmin,\delmax)}\big),\ {\Lipm}'<0,\ \Lipm{\Lipm}''<c({\Lipm}')^2$ for some $0<c<2$ and $\displaystyle\lim_{\delta\to+\infty}\Lipm(\delta)=0\mbox{ if }\delmax=+\infty$, for all $1\leq i\leq k$.\footnote{Lateral derivatives are considered on the domain border. The strict inequality $c<2$ is needed when $\gamma=0$ and spreads are unconstrained.}
\end{assumptions}

\begin{remark*}
The intensity functions of Examples \ref{ex} \ref{ex1}, \ref{ex2} and \ref{ex3} all verify these stronger assumptions, while Example \ref{ex4} only does so for $\Dpm_i>1$. As a matter of fact, by solving the differential inequality in Assumptions \ref{assumptions_intensities_strong}, one sees that any $\Lipm$ within this new framework has a strict upper bound of the form of Example \ref{ex4}. In particular, for $\delmax=+\infty$, $\lim_{\delta\to+\infty}\delta\Lipm(\delta)=0$ is necessarily satisfied. 
\end{remark*}

Under Assumptions \ref{assumptions_intensities_strong}, the maximization of (\ref{eq1_argmax}) (for fixed $(t,n,i)$) can be replaced essentially by solving a contractive fixed point equation and flooring and capping the results at $\delmin$ and $\delmax$ respectively. To make this precise, even when the intensity functions are not defined beyond $\overline{(\delmin,\delmax)}$, let us set	 
$$
\Dpm_{i*}\defeq 
\begin{cases}
-U_\gamma^{-1}\left(\frac{\Lipm(\delmin)}{{\Lipm}'(\delmin)}\right) - \delmin &\mbox{ if }\delmin>-\infty\\
+\infty &\mbox{ if }\delmin=-\infty,
\end{cases}
\quad 
{\Dpm_i}^*\defeq 
\begin{cases}
-U_\gamma^{-1}\left(\frac{\Lipm(\delmax)}{{\Lipm}'(\delmax)}\right) - \delmax &\mbox{ if }\delmax<+\infty\\
-\infty &\mbox{ if }\delmax=+\infty. 
\end{cases}
$$
The following result shows how, up to constraints, the optimal spreads are given by a first term that maximizes the `expected' utility of the MM's instantaneous margin (see Remark \ref{expected_margin}), plus an additional risk adjustment taking into account the inventory held and the prospect of the market shifting. Note how the first term depends on ${\Lipm}'/\Lipm$, the percentage sensitivity of the liquidity to spread changes.
\begin{proposition}
\label{prop_flooring_capping}
Under Assumptions \ref{assumptions_intensities_strong}, the functions $\overline{\delpm_1},\dots,\overline{\delpm_k}$ of Theorem \ref{verif_full_info} are uniquely characterized by
\begin{equation}
\label{flooring_capping}
\overline{\delpm_i}(d)= \delmax\mbox{ if }\ d<{\Dpm_i}^*,\qquad\overline{\delpm_i}(d)= \delmin\mbox{ if }\ d>\Dpm_{i*}
\end{equation}
and if $d\in\overline{({\Dpm_i}^*,\Dpm_{i*})}$, then $\overline{\delpm_i}(d)$ is the unique solution of the fixed point equation
\begin{equation}
\label{eq2_fixed_point}
\overline{\delpm_i}(d)= -U_\gamma^{-1}\left(\frac{\Lipm(\overline{\delpm_i}(d))}{{\Lipm}'(\overline{\delpm_i}(d))}\right) - d.
\end{equation}
Additionally, the Hamiltonians of equation (\ref{HJB_Theta_full_info}) can be expressed as
\begin{equation*}
\label{H_constrained_explicit}
\Hpm_i(d)=
\begin{cases}
\hpm_i(\delmax,d)\qquad\qquad\qquad\qquad\mbox{ if }d\leq{\Dpm_i}^*\\
-\frac{{\Lipm}^2(\widehat\delpm(d))}{{\Lipm}'(\widehat\delpm(d))-\gamma\Lipm(\widehat\delpm(d))}\qquad\,\mbox{ if }{\Dpm_i}^*< d<\Dpm_{i*}\\
\hpm_i(\delmin,d)\qquad\qquad\qquad\qquad\mbox{ if }d\geq\Dpm_{i*}.
\end{cases}
\end{equation*}
\end{proposition}

\begin{proof}
For each $1\leq i\leq k$, define 
$$
\fipm(\delta,d)=\delta+U_\gamma^{-1}\left(\frac{\Lipm(\delta)}{{\Lipm}'(\delta)}\right)+d,\quad\mbox{with }\delta\in\overline{(\delmin,\delmax)},\ d\in\mathbb R
.$$
By Assumptions \ref{assumptions_intensities_strong} and straightforward computations (see computations, e.g., in \cite[Lemma 3.1]{G}), one verifies that $\sgn\big(\frac{\partial\hpm_i}{\partial\delta}\big)=-\sgn(\fipm)$ and $\fipm$ is strictly increasing on each variable.\footnote{$\sgn$ denotes the sign function, with $\sgn(0)\defeq 0$.} It follows that for any $d<{\Dpm_i}^*$ and $\delta<\delmax$ (resp. $d>\Dpm_{i*}$ and $\delta>\delmin$), $\fipm(\delta,d)<\fipm(\delmax,{\Dpm_i}^*)=0$ (resp. $\fipm(\delta,d)>\fipm(\delmin,\Dpm_{i*})=0$), which proves (\ref{flooring_capping}). 

In the same way, if $d\in({\Dpm_i}^*,\Dpm_{i*})$ and $-\infty<\delmin<\delmax<+\infty$, then $\fipm(\delmin,d)<0$ and $\fipm(\delmax,d)>0$. Hence, the continuity of $\fipm(\cdot,d)$ implies (\ref{eq2_fixed_point}) (the uniqueness of the solution being due to the strict monotonicity of $\fipm(\cdot,d)$). The cases of unconstrained spreads are proved as in \cite{G}. 

Lastly, the cases $d={\Dpm_i}^*>-\infty$ and $d=\Dpm_{i*}<+\infty$ follow in the same manner; and the new expressions for the Hamiltonians are immediate, putting $\Hpm_i(d) = \hpm_i\big(\overline{\delpm_i}(d),d\big)$.
\end{proof}

As done in \cite[Lemma 3.1]{G}, (\ref{eq2_fixed_point}) can be replaced by the explicit formula 
$$
\overline{\delpm_i}(d) = (\Lipm)^{-1}\left(\gamma \Hpm_i(d) -{\Hpm_i}'(d) \right), 
$$
but the computation of $(\Lipm)^{-1}$, $\Hpm_i$ and ${\Hpm_i}'$ still has to be carried out numerically, in general. We state now a few simplifications that arise in some particular cases, i.e., when the intensities are exponential or when $\Nmax=\infty$, $\gamma=\zeta=0\equiv\ell$. These are all derived by straightforward substitution. We refer to \cite[Sect.4]{G} for some asymptotic approximations when $t<<T$ in the one regime case, with unconstrained spreads and $\Nmax<\infty$. 

\begin{corollary}
\label{coro_spreads_vanilla_model}
If $\Nmax=+\infty$ and $\gamma=\zeta=0\equiv\ell$, then 
\begin{equation}
\overline{\delpm}(t,i)= \delmax\mbox{ if }\ \mp\mu(T-t)<{\Dpm_i}^*,\qquad\overline{\delpm_i}(d)= \delmin\mbox{ if }\ \mp\mu(T-t)>\Dpm_{i*},
\end{equation}
and if $\mp\mu(T-t)\in\overline{({\Dpm_i}^*,\Dpm_{i*})}$, then $\overline{\delpm}(t,i)$ is the unique solution of the fixed point equation
\begin{equation}
\label{eq3_fixed_point}
\overline{\delpm}(t,i) = -\frac{\Lipm(\overline{\delpm}(t,i))}{{\Lipm}'(\overline{\delpm}(t,i))} \pm \mu (T-t).
\end{equation}
If additionally $\Lipm(\delta)=\apm_i e^{-\bpm_i\delta}$, with $\apm_i,\bpm_i>0$ for each $1\leq i\leq k$, then 
$$ \overline{\delpm}(t,i) = \delmin\vee\left(\frac{1}{\bpm_i}\pm\mu (T-t)\right)\wedge\delmax.$$
\end{corollary}
In other words, if $\Nmax=\infty$ and $\gamma=\zeta=0\equiv\ell$, then solving (\ref{HJB_Theta_full_info}) is no longer necessary and one only needs to solve the fixed point equations (\ref{eq3_fixed_point}). Further, the MM's spreads do not depend on $n$. This is to be expected, since in this case she is neutral to all types of inventory risks and the terminal execution cost is neglected. However, she does need to re-adjust for changes in the regime as these impact on her probability of getting orders. If additionally, the orders intensities are exponential, then the spreads are computed straightforwardly avoiding the need of any numerical scheme. Note also how this simple model makes evident a second component of the optimal spreads through equation (\ref{eq3_fixed_point}): a drift adjustment by which the MM takes into account the overall tendency of the asset's price.  

Now we look at the results obtained for exponential intensities in general, and in particular for $\Nmax<+\infty$. Equation (\ref{eq2_fixed_point}) becomes an explicit formula in this case, as the percentage liquidity sensitivities,  ${\Lipm}'/\Lipm$, are constant.
\begin{corollary}
If $\Lipm(\delta)=\apm_i e^{-\bpm_i\delta}$, with $\apm_i,\bpm_i>0$ for each $1\leq i\leq k$, then
$$\overline{\delpm}(t,n,i) =\delmin\vee \left( -U_\gamma^{-1}\left(-1/\bpm_i\right)-\big(\Theta(t,n\mp 1,i)-\Theta(t,n,i)\big)\right)\wedge\delmax.$$
\end{corollary}
\noindent This means that for exponential intensities in general, it is no longer necessary to solve any fixed point equations but only the system of ODEs (\ref{HJB_Theta_full_info}). The latter still needs to be solved numerically in general. (Alternatively, see \cite{G} for some asymptotic approximations.) 

\begin{remark*}
Apart from the case $\Nmax=+\infty,\ \gamma=\zeta=0\equiv\ell$, a noteworthy scenario in which (\ref{HJB_Theta_full_info}) can be further simplified to a linear system of ODEs (with constant coefficients) is when $k=1,\ \Nmax<+\infty,\ \delmin=-\infty,\ \delmax=+\infty$ and $\Lpm(\delta)=a e^{-b\delta}$, with $a,b>0$. The reduction is achieved via the transformation 
$
\Theta(t,n)=\frac{1}{b}\log\Psi(t,n).
$
See \cite{GLFT} for more details.
\end{remark*}

\section{Numerical analysis}
\label{s:numerics}
\setcounter{equation}{0}
\setcounter{subsection}{0}

In this section we present our numerical results, focusing on the difference in optimal behaviours between the partial and full information frameworks, and the intuition behind the filter. To exemplify our findings in a concrete manner, while keeping the presentation as simple as possible, we will assume throughout this section that: 
\begin{itemize}
\item The MM has risk aversion parameter $\gamma=0$.
\item There are only two possible states: state 1 represents a `bad' regime with low liquidity taken by clients, and state 2 a `good' regime with high liquidity.
\item The transition rate matrix $Q$ is constant.
\item The intensities are symmetric, proportional and exponential, i.e., $\Lpm_1(\delta)=\Lambda_1(\delta)=ae^{-b\delta}$ and $\Lpm_2(\delta)=\Lambda_2(\delta)=m\Lambda_1(\delta),\mbox{ with }a,b>0,\ m>1$.
\end{itemize}
The latter assumption allows us to perform the optimizations in equation (\ref{HJB_Theta}) analytically. Practically, proportional intensities mean that while there is more active trading on the good regime, the way in which the clients react to movements in the spreads remains unaffected. As in \cite{CJ}, we will allow for the three type of penalties to manage inventory risk: constraints on the maximum long and short positions, accumulated inventory penalty and a quadratic (possibly negligible) terminal penalty (or cost) for the MM. That is, 
\begin{itemize}
\item $\Nmax<+\infty,\ \zeta\geq 0\mbox{ and }\ell(n)=cn^2$, for some $c\geq 0$.
\end{itemize}

Let us write $\pi\defeq \pi_1$ for the conditional probability of being in the bad regime given the observable information. Note that $\pi$ is a scalar in this section, since $\pi_2=1-\pi_1$ neglects the need of the additional variable. The PIDE (\ref{HJB_Theta}) at the point $(t,n,\pi)$ reads:
\begin{equation}
\label{HJB_numerics}
0 = \theta_t + \mu n - \frac{1}{2}\sigma^2 n^2 \zeta+ 2\hat q(\pi)\theta_\pi +\frac{a}{b}\hat m(\pi)\left(\mathbbm 1_{\{n<\Nmax\}} e^{-b\widehat\delm} + \mathbbm 1_{\{n>-\Nmax\}} e^{-b\widehat\delp} \right),
\end{equation}
with terminal condition $\theta(T,n,\pi) = -cn^2$ and partial information optimal spreads given by
\begin{equation}
\label{partial_info_optimal_spreads}
\widehat\delpm(t,n,\pi)=\frac{1}{b} - 2\frac{w(\pi)}{\hat m(\pi)}\theta_\pi(t,n,\pi) - \big(\theta(t,n\mp 1,\pi/\hat m(\pi))-\theta(t,n,\pi)\big),\mbox{ where:}
\end{equation}
\begin{enumerate}[label=(\roman*)]
\item $\hat q(\pi) = q^{11}\pi + q^{21}(1-\pi)$ is the observable transition rate to the bad regime.
\item $\hat m(\pi) = \pi + (1-\pi)m$ is the observable intensity increase from the bad regime (as a ratio).
\item $w(\pi)=(m-1)\pi(1-\pi)$ is the observable variance, of the square root, of the percentage intensity increase from the bad regime; i.e., a measure of observable order flow volatility.
\end{enumerate}
The previous equations are valid for $\delmin\ll 0\ll\delmax$, or more precisely, for $\delmin,\delmax$ such that $\delmin\leq\widehat\delpm\leq\delmax$ holds true over the whole domain. Otherwise, one needs to floor and cap the optimal spreads and change the Hamiltonians accordingly, as done in Proposition \ref{prop_flooring_capping}. 

A finite differences scheme of simple implementation to solve (\ref{HJB_numerics}) consists of reversing time and using an explicit upwind Euler scheme over a uniform grid for $(t,\pi)$ and $n=-\Nmax,-\Nmax+1,\dots,\Nmax$. The terms of the form $\theta(t,n\mp 1,\pi/\hat m(\pi))$, where $\pi/\hat m(\pi)$ will typically fall outside of the grid, can be approximated by linear interpolation as in \cite{DFo}. The limiting equations can be used for $\pi\to 0^+$ and $\pi\to 1^-$ provided that $-q^{11}>0$, which we assume henceforth. Such a scheme can be shown to be consistent, stable and monotone under an appropriate CFL condition. In light of equation (\ref{HJB_Theta}) satisfying a comparison principle (see \cite[Thm.5.3]{CEFS} and why it holds for our model in the proof of Theorem \ref{main_theorem_2}), we know the scheme converges to the unique continuous viscosity solution \cite{BS}, and we can recover the expected penalized P\&L of the MM (Theorems \ref{main_theorem_1} and \ref{main_theorem_2}). The optimal strategy to be followed by the MM is then given in feedback form by $\big(\widehat\delm(t,N_{t^-},\Pi_{t^-}),\widehat\delp(t,N_{t^-},\Pi_{t^-})\big)$.\footnote{As previously mentioned, these are technically only candidates for optimal (or in general, $\ve$-optimal) strategies. We do not rigorously prove their optimality character here, but merely note that well known results of discrete time dynamic programming, together with the convergence of the discrete solutions of (\ref{HJB_numerics}) to the analytical one, suggest that they can be safely used as such.}

We will focus our attention on the optimal ask spread, with analogous observations holding for the bid spread. The parameter values in Table \ref{table_parameters} were used for all the experiments presented in this section. For those present in the classical one regime models, we have chosen values used in previous works (in the $\gamma=0$ case \cite{CJ}) to make the comparison clearer. In particular, the value of $c$ will be taken as either $0$ or $0.01$, to be further specified in each experiment. The time horizon will always be the one displayed on the corresponding axis. Note that we work in a symmetric market, which justifies analyzing one side only.

\begin{table}[h!]
\centering
 \begin{tabular}{||c |c |c |c |c |c |c |c |c |c |c ||} 
 \hline
 Parameter & $\mu$ & $\sigma$ & $\zeta$ & $a$ & $b$ & $\Nmax$ & m & $-q^{11}=q^{21}$ & $-\delmin=\delmax$ \\ [0.5ex] 
 \hline
 Value & 0 & 0.1 & 0.1 & 2 & 25 & 3 & 5 & 5 & 10 \\
 \hline
\end{tabular}
\caption[Parameters for market making numerical tests]{Standing parameter values for numerical tests presented.}
\label{table_parameters}
\end{table}

\subsection{Comparing full and partial information optimal strategies}

Under the standing assumptions of this section, the full information equation (\ref{HJB_Theta_full_info}), for a function $\tilde\theta$, becomes:
\begin{equation}
\label{HJB_full_info_numerics}
0 = \tilde\theta_t + \mu n - \frac{1}{2}\sigma^2 n^2 \zeta+ \tilde q_i \left(\tilde\theta(t,n,1)-\tilde\theta(t,n,2)\right) +
\frac{a}{b} \tilde m_i \left(\mathbbm 1_{\{n<\Nmax\}} e^{-b\widetilde\delm} + \mathbbm 1_{\{n>-\Nmax\}} e^{-b\widetilde\delp} \right),
\end{equation}
with terminal condition $\tilde\theta(T,n,i) = -cn^2$ and full information optimal spreads given by
\begin{equation}
\label{full_info_optimal_spreads}
\widetilde\delpm(t,n,i)=\frac{1}{b} - \big(\tilde\theta(t,n\mp 1,i)-\tilde\theta(t,n,i)\big),\mbox{ where:}
\end{equation}
\begin{enumerate}[label=(\roman*)]
\item $\tilde q_i =q^{11}\mathbbm 1_{\{1\}}(i)+q^{21}\mathbbm 1_{\{2\}}(i)$ is the effective transition rate to the bad regime.
\item $\tilde m_i = \mathbbm 1_{\{1\}}(i) + \mathbbm 1_{\{2\}}(i) m$ is the effective intensity increase from the bad regime (as a ratio).
\end{enumerate}
\noindent The previous equations are valid for $\delmin\ll 0\ll\delmax$, as in the partial information case. 

Although similar, equation (\ref{HJB_full_info_numerics}) for the bad regime $i=1$ (resp. good regime $i=2$) is not the limiting equation of (\ref{HJB_numerics}) for $\pi\to 1^-$ (resp. $\pi\to 0^+$). Indeed, a MM with full information can expect to make a larger profit. Thus, in general, $\tilde\theta>\theta$ even in these extreme cases.\footnote{Our numerical findings were indeed consistent with this intuitive statement.} However, the corresponding optimal strategies do (at least approximately) agree, as Figures \ref{delp_full_vs_partial_running_penalty} and \ref{delp_full_vs_partial_all_penalties} show.

\begin{figure}[H]
\hspace*{-.2cm}
	\includegraphics[scale=.33]{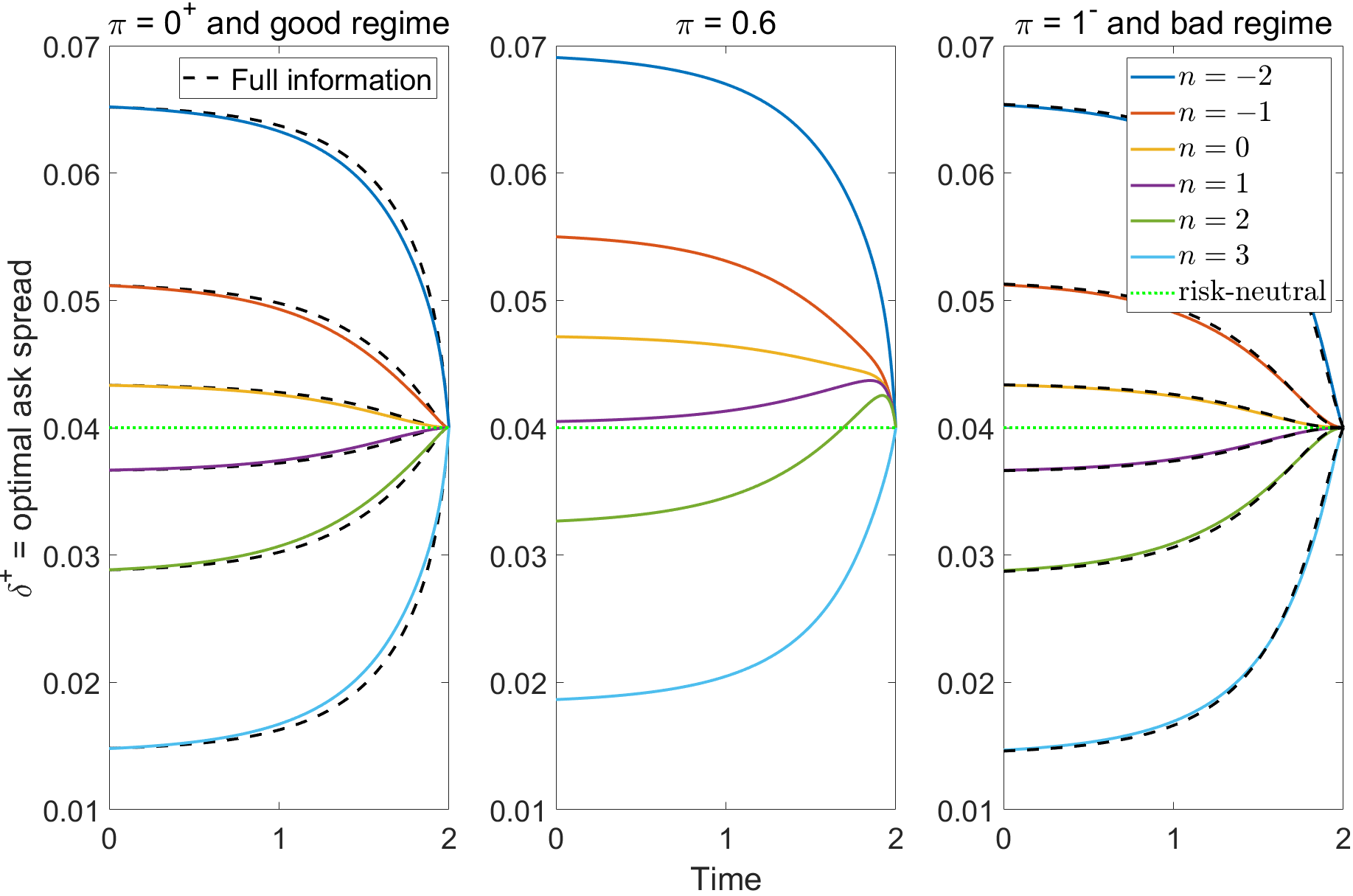}
	\caption[Optimal spreads with full and partial information without terminal penalty]{Optimal ask spread as a function of time, under full and partial information (dashed and solid lines resp.), and for different inventory and filter values. Inventory levels increase from top to bottom in all cases. Terminal execution cost is neglected ($c=0$).}
	\label{delp_full_vs_partial_running_penalty}
\end{figure}

Figure \ref{delp_full_vs_partial_running_penalty} shows the optimal ask spread under partial information as a function of time, for different inventory levels (solid lines). It also displays how it changes for different filter values ($\pi=0,\ 0.6\mbox{ and }1$ from left to right) and how it compares to the optimal ask spread under full information (dashed lines) in the good regime (left) and bad regime (right). The inventory levels increase from top to bottom in all cases, and the terminal execution cost is neglected ($c=0$). We have chosen to display $\pi=0.6$ due to the effect of the filter being more pronounced around this value; it results in an optimal spread which is, for $t\ll T$, between $6\%$ and $30\%$ higher (depending on the inventory level) than the corresponding ones under full information (a sizeable difference in practice). We will come back to this in Section \ref{section_simulations}. All other values show similar intermediate behaviours. Recall that when the inventory reaches the minimum $n=-\Nmax=-3$, the MM will not sell any more (either abstaining from quoting or giving a `stub' quote) until her inventory increases again, which is why there is no spread plotted for this position. We see that some of the features already present in one regime models are preserved for two regimes, both under full and partial information; namely: 

\begin{itemize}

\item For $t\ll T$ the spread does not depend on time (approximately) and its asymptotic value becomes higher as the inventory decreases to the lower constraint $n=-\Nmax$. Indeed, for flat or short positions, any additional ask order will increase inventory risk, moving it closer to the minimum allowed and raising the accumulated inventory penalty. The MM manages this risk by increasing the spread, thus demanding a higher instantaneous profit and reducing the probability of getting executed. Similarly for long positions, the higher the exposure the lower the spread, as the MM seeks to unwind her inventory.

\item In the case of negligible terminal cost, the spread converges to a terminal value independent of the inventory. This is the optimal spread of a fully risk-neutral MM with negligible costs (Corollary \ref{coro_spreads_vanilla_model}), who only maximizes `expected instantaneous margin'. (This value is given by $\widehat\delp(T,\cdot,\cdot) = 1/b$, the reciprocal percentage sensitivity of the liquidity to spread changes; cf. Proposition \ref{prop_flooring_capping}). The reason being that as the time horizon approaches, it becomes less and less likely for the MM to get executed again. The accumulated penalty cannot change too much at this point, and the risk of reaching the inventory constraints diminishes. The MM therefore takes a bit more risk and makes a last attempt at increasing her expected P\&L, either by increasing execution probability (for null or short positions) or increasing her instantaneous profit if executed (for long positions).

\end{itemize}

\noindent Nevertheless, the following differences can be observed:

\begin{itemize}

\item Under full information (and for symmetric markets), the spreads are always symmetric in the inventory with respect the risk-neutral level, i.e., $1/b - \widetilde\delp(\cdot,n,\cdot) = \widetilde\delp(\cdot,1-n,\cdot) - 1/b$, for $0<n\leq\Nmax$. However, under partial information this symmetry is broken, as the MM increases her spreads. This holds in fact for any $0<\pi<1$ and it is suggested by the comparison of equations (\ref{partial_info_optimal_spreads}) and (\ref{full_info_optimal_spreads}), since $\theta_\pi<0$ (an increasing probability of being in the bad regime lowers the expected P\&L; see argument of \cite[Lemma 3.3]{BL0}). As a result, the full information spreads are downward biased when the exact regime is unknown. Loosely speaking, this bias is approximately proportional to the product of $\theta_\pi$ (sensitivity of the expected P\&L to observable regime changes) and $w(\pi)$ (observable order flow volatility), and inversely proportional to the observable intensity increase $\hat m(\pi)$. Intuitively, a partially informed MM faces not only the risk of the market regime shifting but also uncertainty on the current state, and must increase her spreads accordingly. This is done by considering the cost of any observable changes on the P\&L and the fluctuations in the order flow, and discounting by liquidity increases.

\item The good and bad regimes differ in the rate of order flow at the reference price (i.e., the amplitude parameters $a$ and $am$ of the orders intensities). Consequently, one could erroneously think that the partial information strategy should essentially be a one regime strategy for some intermediate parameter, such as $a\hat m(\pi)$. However, in the full information framework, an increase in the liquidity taken by clients results in the asymptotic spreads (i.e., for $t\ll T$) moving closer to the risk-neutral level (see \cite{CJ, GLFT}). On the other hand, this is not true any more in the partial information case, due to the regime risk adjustment previously mentioned that shifts the spreads.

\item Lastly, we remark that there is a distinguished change in behaviour close to expiry, with the spreads overshooting above the risk-neutral level for some inventories, before approaching it again. Intuitively, the overshooting is due to the added regime risk adjustment, and the convergence is due to the vanishing of additional inventory risks. (Recall that the MM's main concern becomes her terminal hedging cost, should there be one, and this does not depend on the market regime.) A somewhat similar effect can be found under full information in asymmetric markets \cite{CJ}.

\end{itemize}
 
Figure \ref{delp_full_vs_partial_all_penalties} serves the same comparison as Figure \ref{delp_full_vs_partial_running_penalty}, but including  in this case a terminal execution cost. The observations remain mostly the same, except for a change in the terminal behaviour of the spreads, just as in the classical one regime case. Here, the spreads diverge from the risk-neutral no-costs value instead of approaching it, as the MM makes a last instant attempt to cover her hedging cost.

\begin{figure}[H]
\hspace*{-.2cm}
	\includegraphics[scale=.33]{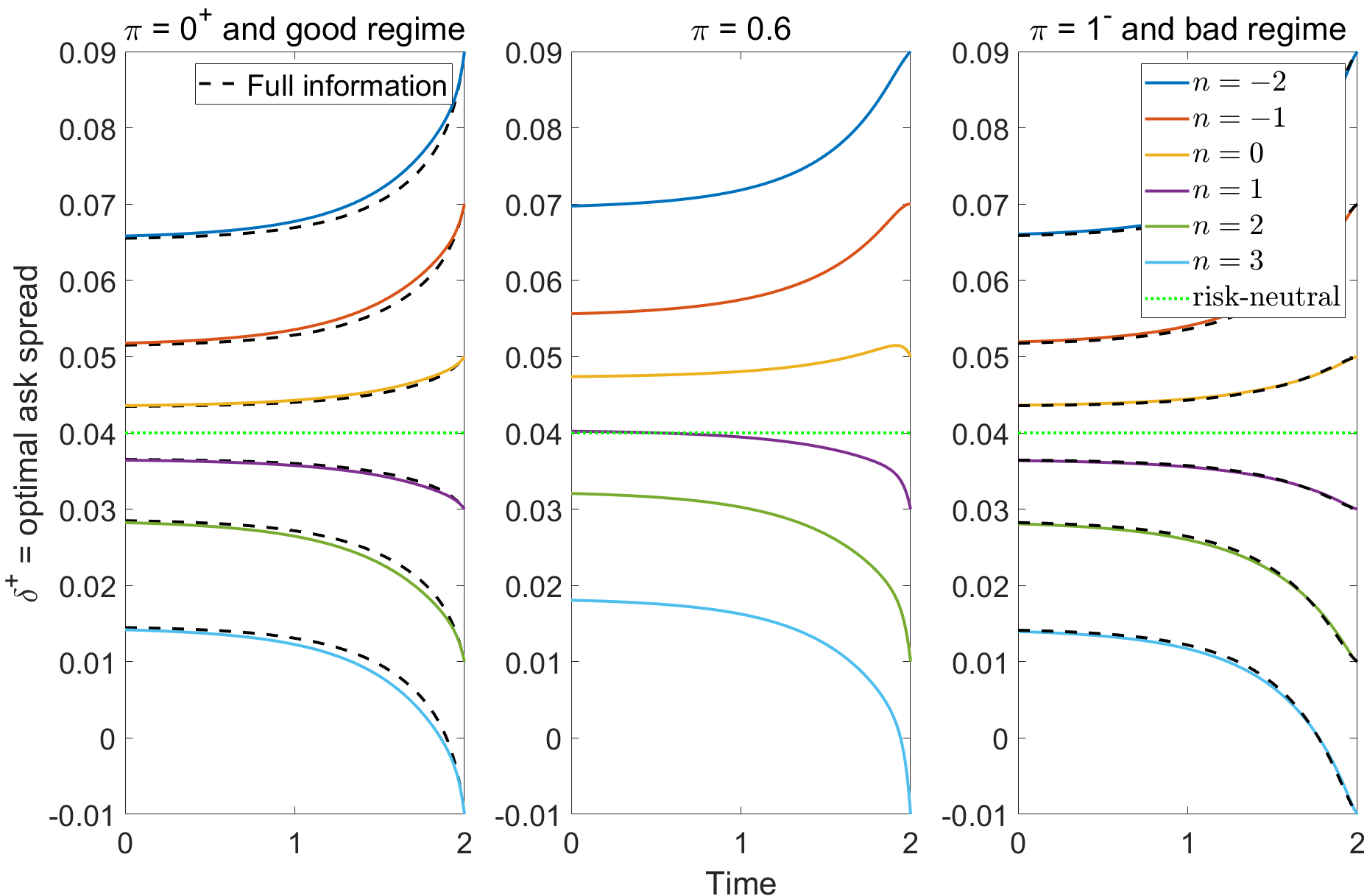}
	\caption[Optimal spreads with full and partial information with terminal penalty]{Optimal ask spread as a function of time, under full and partial information (dashed and solid lines resp.), and for different inventory and filter values. Inventory levels increase from top to bottom in all cases. Terminal execution cost: $\ell(n)=-cn^2$, with $c=0.01$.}
	\label{delp_full_vs_partial_all_penalties}
\end{figure}

Figure \ref{delp_partial_surf} shows the partial information optimal ask spread as a function of time and filter, for three different inventory positions: extreme short (left), flat (center) and extreme long (right), and no terminal penalty. We can see that the spread is concave on the filter, and the maximum concavity is reached for a null inventory and $\pi\approx 0.6$, gradually decreasing as the position becomes short or long. We had observed already in equation (\ref{partial_info_optimal_spreads}) that the partially informed MM increases her spreads to manage higher regime risk, and that this correction is approximately proportional to the expected P\&L sensitivity, $\theta_\pi$, times the observed order flow volatility $w(\pi)=(m-1)\pi(1-\pi)$. Heuristically, it seems reasonable that the change in spread should resemble the concavity of $w(\pi)$ for a risk averse MM, and that the cost of regime uncertainty is the highest for a flat position (since deviations from it increase in turn price exposure). The question of why the maximum variation is reached around $\pi\approx 0.6$ is once again differed to Section \ref{section_simulations}.

\begin{figure}[H]
\hspace*{-.05cm}
	\includegraphics[scale=.33]{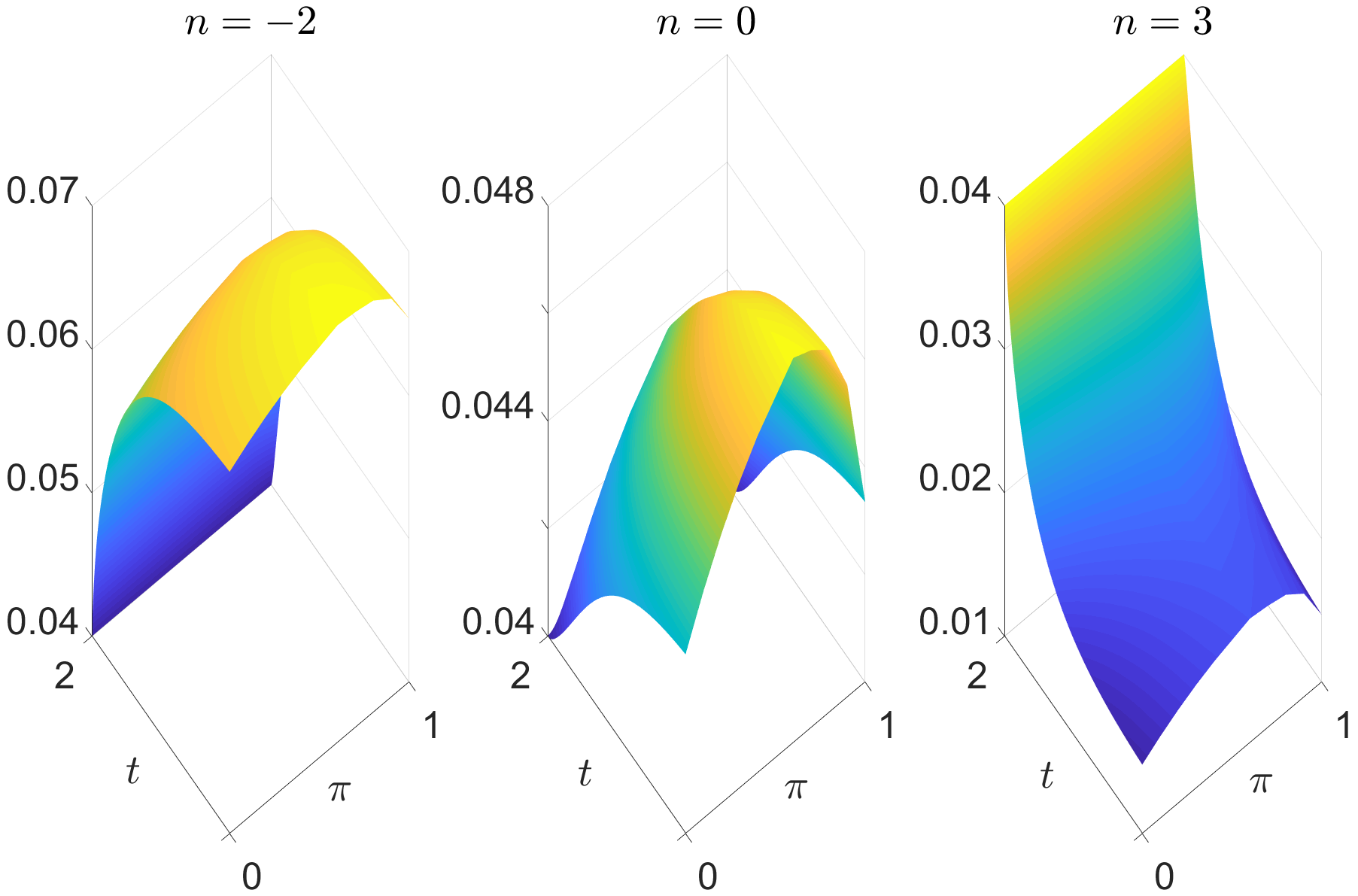}
	\caption[Optimal spreads as functions of time and filter]{Optimal partial information ask spread as a function of time and filter (observable probability of bad regime), for different inventory levels. Terminal execution cost is neglected ($c=0$).}	
	\label{delp_partial_surf}
\end{figure}

\subsection{Sample paths and a closer look at the filter}
\label{section_simulations}

There are different ways to simulate point processes with stochastic intensity. A classical approach of simple implementation is the so-called \textit{thinning method} \cite{O}. This method is particularly well suited to our framework, as it can be combined with the filtering theory developed in Section \ref{s:filtering} by making use of the observable intensity of $N$ (see Theorem \ref{main_theorem_1} and \cite[Alg.3.2]{GKM}). The interested reader is referred to \cite{LTT} for optimizations in terms of the thinning bound. 

Figure \ref{simulations} shows four sample paths resulting from an optimal behaviour of the MM with incomplete information. They were obtained by jointly simulating the inventory $N$ (middle) starting from $n_0=0$, the filter $\Pi=\Pi^{\alpha,1}$ (bottom) starting from $\pi_0=0.5$, and the optimal strategy $\alpha=\big(\widehat\delm(t,N_{t^-},\Pi_{t^-}),\widehat\delp(t,N_{t^-},\Pi_{t^-})\big)$ (top), for $c=0$.

\begin{figure}[H]
\hspace*{-.12cm}
	\includegraphics[scale=.58]{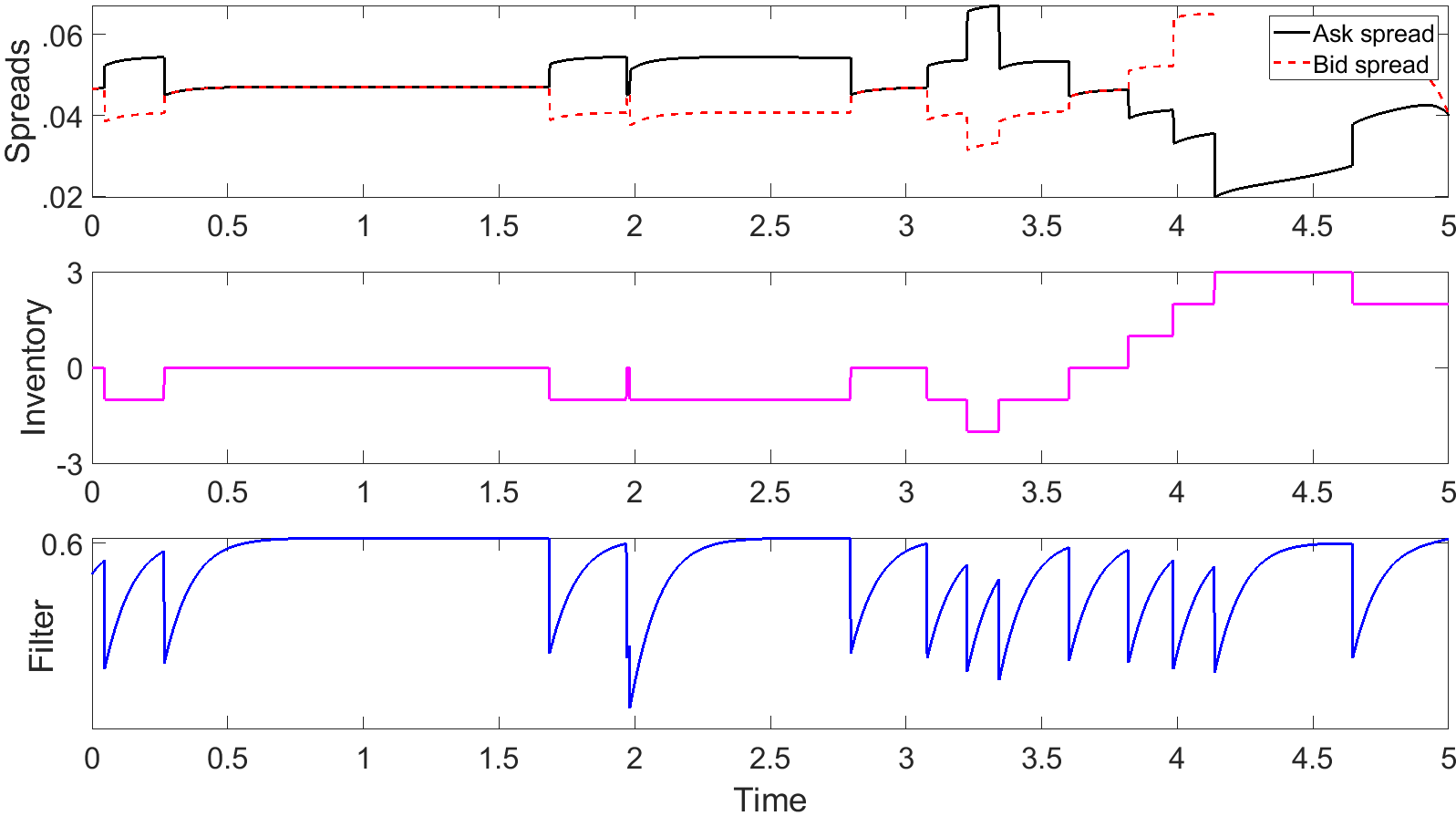}
	\caption[Sample paths of optimal spreads, inventory and filter]{Sample paths of optimal partial information bid/ask spreads (first), inventory (second) and filter (third), with $c=0$, $n_0=0$ and $\pi_0=0.5$.}	
	\label{simulations}
\end{figure}

We want to analyze the behaviour of $\Pi$. Let us recall that $\Pi_t$ represents the probability of being in the bad regime (a slow market with low liquidity) given the information observed by the MM at time $t$. The MM makes her assessment based on the orders (both buy and sell) that she received so far, which is the same as looking at the evolution of her inventory. If we take a close look at what happens between any two consecutive orders, we see that there is a pattern that repeats itself. As time passes without the MM receiving any order, she deems more likely that the market is in the bad regime, and so $\Pi_t$ increases. If enough time goes by in this way, then $\Pi_t$ gets asymptotically closer to some preponderant `equilibrium' value which is around 0.6. (Recall that this is the same value we have encountered before as having the greatest effect on the spreads.) But as soon as an order arrives, the MM revisits her probabilities, as being in the good regime of high liquidity seems a lot more likely. This is why we see that $\Pi_t$ jumps downwards with each trade (i.e., at the jump times of the inventory). 

So what makes $\pi\approx 0.6$ so special? If we look at the SDE (\ref{Pi_equivalent}) governing the filter dynamics, it now reads:

\begin{equation}
d\Pi_t = \left(\hat q(\Pi_t) + w(\Pi_t)a\big(\mathbbm 1^-_t e^{-b\widehat{\delta_t^-}}+\mathbbm 1^+_t e^{-b\widehat{\delta_t^+}}\big)\right)dt + \Pi_{t^-}\left(\frac{1}{\hat m\big(\Pi_{t^-}\big)}-1\right)\Big(d\Nm_t + d\Np_t \Big).
\end{equation}
In between jumps, each path of $\Pi_t$ evolves according to an ODE. If enough time passes by without another order arriving, we can consider the asymptotic behaviour for $t\to T$. For $c=0$, this leads to
$$
\frac{d\Pi_t}{dt}\approx \hat q(\Pi_t) + w(\Pi_t)\frac{a}{e}\beta,\quad\mbox{with }\beta=1\mbox{ or }\beta=2,
$$
depending on the current inventory level. The right-hand side of this ODE is a concave parabola in $\Pi_t$ with a negative root and another one, $\pi^*$, between 0 and 1. For the parameter values in Table \ref{table_parameters}, it is either $\pi^*\approx 0.572$ for $\beta=1$, or $\pi^*\approx 0.636$ for $\beta=2$. Therefore, $\pi^*$ is an attractor (i.e., a stable equilibrium) for the filter, which explains the tendency of $\Pi_t$ towards 0.6 and the central role of this value. Indeed, as the filter approaches one of its attractors, the MM exhausts all the information she has and can make no further assessment on the state of the market until she is hit by another order. In a sense, this is when she faces the most uncertainty, and she increases her spreads accordingly to manage regime risk.


\chapter[A policy iteration algorithm for nonzero-sum impulse games]{A policy iteration algorithm for nonzero-sum stochastic impulse games}
\chaptermark{An algorithm for impulse control games}
\label{c:2}

\phantomsection
\section*{Introduction}
\addcontentsline{toc}{section}{Introduction}

Stochastic differential games model the interaction between players whose objective functions depend on the evolution of a certain continuous-time stochastic process. The subclass of impulse games focuses on the case where the players only act at discrete (usually random) points in time by shifting the process. In doing so, each of them incurs into costs and possibly generates `gains' (not necessarily nonnegative) for the others at the same time. They constitute a generalization of the well-known (single-player) optimal impulse control problems \cite[Chpt.7-10]{OS}, which have found a wide range of applications in finance, energy markets and insurance \cite{B,CCTZ,EH,K}, among plenty of other fields. 

From a deterministic numerical viewpoint, an impulse control problem entails the resolution of a differential quasi-variational inequality (QVI) to compute the value function and, when possible, retrieve an optimal strategy. Policy-iteration-type algorithms \cite{AF,COS,CMS} undoubtedly occupy an ubiquitous place in this respect, and more so in the infinite horizon case.   

The presence of a second player makes matters much more challenging, as one needs to find two optimal (or \textit{equilibrium}) payoffs dependent on one another, and the optimal strategies take the form of \textit{Nash equilibria} (NEs), i.e., strategies such that no player can benefit from unilaterally changing her own. And while impulse controls give a more realistic setting than `continuous' controls in applications as the aforementioned ones, they normally lead to less tractable and very technical models.

It is not surprising then, that the literature in impulse games is limited and mainly focused on the zero-sum case \cite{A1,C,EM}: when the cost for one player equals the gain for the other, and the same is true for the objective functions. The more general and flexible nonzero-sum instance, which dispenses with the former constraints, has only recently began to receive more attention. The authors of \cite{ABCCV} consider for the first time a general two-player game where both participants act through impulse controls,\footnote{\cite{CWW,WW} also consider nonzero-sum impulse games but assuming the intervention times of the players are known from the outset.} and characterize certain type of equilibrium payoffs and NEs via a system of QVIs by means of a Verification Theorem. Using this result, they provide the first example of an (almost) fully analytically solvable game, motivated by central banks competition over the exchange rate. The result is generalized to arbitrary $N$ players in \cite{BCG}, which also gives a semi-analytical solution (i.e., depending on several parameters found numerically) to a concrete cash management problem.\footnote{\cite{BCG} also studies the mean field limit game.} A different, more probabilistic, approach is taken in \cite{FK} to find a semi-analytical solution of a strategic pollution control problem and to prove another Verification Theorem.  

The previous examples, and the lack of others,\footnote{\cite{ABCCV} also gives semi-analytical solutions to modifications of the linear game when changing the payoffs in a non-symmetric way. To the best of the author's knowledge, these are the collection of examples available at the time of writing.} give testimony of how difficult it is to explicitly solve nonzero-sum impulse games. The analytical approaches require an educated guess to start with and (with the exception of the henceforth referred to as \textit{the linear game} in \cite{ABCCV}) several parameters need to be solved for in general from highly-nonlinear systems of equations coupled with order conditions. All of this can be very difficult, if not prohibitive, when the structure of the game is not simple enough. Further, all of them (as well as the majority of the concrete examples in the impulse control literature) assume linear costs. In general, for nonlinear costs, the state to which each player wants to shift the process when intervening is not unique, but depends on the starting point. This effectively means that infinite parameters may need to be solved for, drastically discouraging this methodology. 

While the need for numerical schemes able to handle nonzero-sum impulse games is obvious, unlike in the single-player case, this is an utterly underdeveloped line of research. Focusing on the purely deterministic approach, solving the system of QVIs derived in \cite{ABCCV} involves handling coupled free boundary problems, further complicated by the presence of nonlinear, nonlocal and noncontractive operators. Additionally, solutions will typically be irregular even in the simplest cases such as the linear game. Moreover, the absence of a viscosity solutions framework such as that of impulse control \cite{S} means that it is not possible to know whether the system of QVIs should have a solution or not (not to mention some form of uniqueness) unless one can explicitly solve it. This is further exacerbated by the fact that even defining such system requires a priori assumptions on the solutions (the \textit{unique impulse property}). This is also the case in \cite{FK}.

In this chapter (based on \cite{ABMZZ,Z}), we make the first attempt (to the best of our knowledge) at numerically solving nonzero-sum impulse games. In a nutshell, the iterative algorithm we put forward treats the game at each iteration as a combination of a fixed point problem and a slowly relaxing single-player impulse control problem. This allows us to take advantage of the machinery for the latter. Because no convergence analysis is provided in this chapter, the proposed algorithm is admittedly heuristic. Instead, we report on a range of numerical experiments which show convergence on fixed grids as well as convergence of the error with respect to the discretization. In fact, errors are quite satisfactory and, in the examples we have tackled, the relative errors easily drop well below $0.1\%$. The algorithm can thus be used to assist further development of the field, as well as to gain insight into applications modelled by nonzero-sum impulse games. In Chapter \ref{c:3}, which focuses on \textit{symmetric} games, a modified algorithm will be presented together with a rigorous convergence analysis. 

Let us outline the organization of the rest of the chapter. We start in Section \ref{s:analytical_problem} by properly setting up nonzero-sum stochastic impulse games and recalling the main result of \cite{ABCCV}, namely, the Verification Theorem with the corresponding system of QVIs. For the sake of completeness, the linear game and the underlying problem of competing central banks are also presented. Section \ref{s:general_algo} motivates, lists, and discusses our algorithm. Section \ref{s:experiments} presents numerical evidence supporting it, in the light of which we draw some conclusions in Section \ref{s:conclusions_chapter2}.

\section{Analytical continuous-space problem}
\label{s:analytical_problem}
In this section we start by reviewing a general formulation of two-player nonzero-sum stochastic differential games with impulse controls, as considered in \cite{ABCCV}, together with the main theoretical result of the authors: a characterization of certain NEs via a deterministic system of QVIs.
The indexes of the players are denoted $i,j\in\{1,2\}$ with $i\neq j$. Since no other type of games is considered in this thesis, we will often speak simply of `games' for brevity.

Throughout the thesis, we restrict our attention to the one-dimensional infinite-horizon case. Some of the most technical details, concerning the well-posedness of the model, are left out for the sake of briefness and can be found in \cite[Sect.2]{ABCCV}.

\subsection{General two-player nonzero-sum impulse games}
\label{s:general_game}
Let $(\Omega, \mathcal F, (\mathcal F_t)_{t\geq 0}, \mathbb P)$ be a filtered probability space under the usual conditions supporting a standard one-dimensional Wiener process $W$. We consider two players that observe the evolution of a \textit{state variable} $X$, modifying it when convenient through \textit{controls} of the form $u_i=\{(\tau_i^k,\delta_i^k)\}_{k=1}^\infty$ for $i=1,2$. The stopping times $(\tau_i^k)$ are their \textit{intervention times} and the $\mathcal F_{\tau_i^k}$-measurable random variables $(\delta_i^k)$ are their \textit{intervention impulses}. Given controls $(u_1,u_2)$ and a starting point $X_{0^-}=x\in\mathbb R$, we assume $X=X^{x;u_1,u_2}$ has dynamics
\begin{equation}
\label{X.1}
X_t=x+ \int_0^t\mu(X_s)ds + \int_0^t\sigma(X_s)dW_s + \displaystyle\sum_{k:\ \tau_1^k\leq t}\delta_1^k + \displaystyle\sum_{k:\ \tau_2^k\leq t}\delta_2^k,
\end{equation}
for some given drift and volatility functions $\mu,\sigma:\mathbb R\to \mathbb R$, locally Lipschitz with linear growth.\footnote{See \cite[Def.2.2]{ABCCV} for a precise recursive definition in terms of the strategies.}

Equation (\ref{X.1}) states that $X$ evolves as an It\^o diffusion in between the intervention times, and that each intervention consists in shifting $X$ by applying an impulse. It is assumed that the players choose their controls by means of threshold-type \textit{strategies} of the form 
$\varphi_i=(\mathcal I_i,\delta_i)$, where $\mathcal I_i\subseteq\mathbb R$ is a closed set called \textit{intervention (or action) region} and $\delta_i:\mathbb R\to \mathbb R$ is an \textit{impulse function} assumed to be continuous. The complement $\mathcal C_i=\mathcal I_i^c$ is called \textit{continuation (or waiting) region}.\footnote{In \cite{ABCCV}, strategies are described in terms of continuation regions instead.} 
That is, player $i$ intervenes if and only if the state variable reaches her intervention region, by applying an impulse $\delta_i(X_{t^-})$ (or equivalently, shifting $X_{t^-}$ to $X_{t^-}+\delta_i(X_{t^-})$). Further, we impose a priori constraints on the impulses: for each $x\in\mathbb R$ there exists a set $\emptyset\neq\mathcal Z_i(x)\subseteq\mathbb R$ such that $\delta_i(x)\in \mathcal Z_i(x)$ if $x\in\mathcal I_i$.\footnote{In \cite{ABCCV}, $\mathcal Z_i(x)$ is the same for every $x\in\mathbb R$. The generalization in this chapter is a standard one in impulse control and will prove useful in the sequel. The results in \cite{ABCCV} still hold with the same proofs.}
We also assume the game has no end and player 1 has the priority should they both want to intervene at the same time.

Given a starting point and a pair strategies, the (expected) \textit{payoff} of player $i$ is given by
\begin{equation*}
\label{Ji}
J_i(x;\varphi_1,\varphi_2) \defeq \mathbb E \left[ \int_0^\infty e^{-\rho_i s} f_i(X_s)ds - \sum_{k=1}^\infty e^{-\rho_i \tau_i^k} c_i \big( X_{(\tau_i^k)^-}, \delta_i^k \big) + \sum_{k=1}^\infty e^{-\rho_i \tau_j^k} g_i \big( X_{(\tau_j^k)^-}, \delta_j^k \big)\right],
\end{equation*}
with $X=X^{x;u_1,u_2}=X^{x;\varphi_1,\varphi_2}$. For player $i$, $\rho_i>0$ represents her (subjective) \textit{discount rate}, $f_i:\mathbb R\to\mathbb R$ her \textit{running payoff}, 
$c_i:\mathbb R^2\to(0,+\infty)$ her \textit{cost of intervention} and $g_i:\mathbb R^2\to\mathbb R$ her \textit{gain due to her opponent's intervention} (not necessarily non-negative).
The functions $f_i,c_i,g_i$ are assumed to be continuous.

Throughout this thesis, only admissible strategies are considered. Briefly, $(\varphi_1,\varphi_2)$ is \textit{admissible} if it gives well-defined payoffs for all $x\in\mathbb R$, $\|X\|_\infty$ has finite moments and, although each player can intervene immediately after the other, infinite simultaneous interventions are precluded.\footnote{More precisely, these would be $\mathbb R$-admissible strategies. See \cite[Def.2.5]{ABCCV} for more details.} As an example, if the running payoffs have polynomial growth, the `never intervene strategies'  $\varphi_1=\varphi_2=(\emptyset,\emptyset\hookrightarrow\mathbb R)$ are admissible and the game can be played.

Given a game, we want to know whether it admits some Nash equilibrium and how to compute it. Recall that a pair of strategies $(\varphi_1^*,\varphi_2^*)$ is a \textit{Nash equilibrium} (NE) if for every admissible $(\varphi_1,\varphi_2)$,
$$
J_1(x;\varphi_1^*,\varphi_2^*)\geq J_1(x;\varphi_1,\varphi_2^*)\quad\mbox{and}\quad J_2(x;\varphi_1^*,\varphi_2^*)\geq J_2(x;\varphi_1^*,\varphi_2),
$$
i.e., no player can gain from a unilateral change of strategy. If one such NE exists, we refer to $(V_1,V_2)$, with $V_i(x)= J_i(x;\varphi_1^*,\varphi_2^*)$, as a pair of \textit{equilibrium payoffs}.  

\subsection{General system of quasi-variational inequalities}
\label{s:QVIs}
To present the system of QVIs derived in \cite{ABCCV}, we need to define first the \textit{intervention operators}. For any $V_1, V_2:\mathbb R\to\mathbb R$ and $x\in\mathbb R$, the \textit{loss operator} of player $i$ is defined as 
\begin{equation}
\label{Mi}
\mathcal M_i V_i(x)\defeq\displaystyle\sup_{\delta\in \mathcal Z_i(x)}\{V_i(x+\delta) - c_i(x,\delta)\}.\footnote{Although we could have $\mathcal M_iV_i(x)=+\infty$, this will be excluded when enforcing the system of QVIs (\ref{QVIs}).}
\end{equation}
When applied to an equilibrium payoff, the loss operator $\mathcal M_i$ gives a recomputed present value for player $i$ due to the cost of her own intervention. Given the optimality of the NEs, one would intuitively expect that $\mathcal M_i V_i\leq V_i$ for equilibrium payoffs and that the equality is attained only when it is optimal for player $i$ to intervene. Under this logic:
\begin{definition}
\label{UIP}
We say that the pair $(V_1,V_2)$ has the \textit{unique impulse property (UIP)} if for each $i=1,2$ and $x\in \{\mathcal M_i V_i= V_i\}$, there exists a unique impulse, denoted $\delta_i^*(x)=\delta_i^*(V_i)(x)\in \mathcal Z_i(x)$, that realizes the supremum in (\ref{Mi}).\footnote{We do not require the UIP to hold outside of $\{\mathcal M_i V_i= V_i\}$, as this is not the case for equilibrium payoffs in many examples, such as the linear game with constant costs/gains. Proofs in \cite{ABCCV} carry through unaltered.}
\end{definition}
If $(V_1, V_2)$ enjoys the UIP, we define the \textit{gain operator} of player $i$ as
\begin{equation}
\label{Hi}
\mathcal H_i V_i(x)\defeq V_i(x+\delta_j^*(x)) + g_i(x,\delta_j^*(x)),\quad\mbox{ for }x\in \{\mathcal M_j V_j= V_j\}
\end{equation}
When applied to equilibrium payoffs, the gain operator $\mathcal H_i$ gives a recomputed present value for player $i$ due to her opponent's intervention.

Finally, let us denote by $\mathcal A$ the infinitesimal generator of $X$ when uncontrolled, i.e., 
\begin{equation*}
\mathcal AV(x)\defeq \frac{1}{2}\sigma^2(x) V''(x) + \mu(x)V'(x),
\end{equation*}
for any $V:\mathbb R\to\mathbb R$ which is $C^2$ at some open neighborhood of a given $x\in\mathbb R$. We assume this regularity holds whenever we compute $\mathcal AV(x)$ for some $V$ and $x$. The following Verification Theorem, due to \cite[Thm.3.3]{ABCCV}, states that if a regular enough solution $(V_1,V_2)$ to a certain system of QVIs exists, then it must be a pair of equilibrium payoffs, and a corresponding NE can be retrieved. We state here a simplified version that applies to the one-dimensional infinite-horizon games at hand.\footnote{Unlike in \cite{ABCCV}, there is no terminal condition in the system of QVIs and the assumption that $\partial \mathcal C^*_i$ be a Lipschitz surface is trivially satisfied for an open $\mathcal C^*_i\subseteq\mathbb R$, as it is a countable union of disjoint open intervals.}

\begin{theorem}[\textbf{General system of QVIs}]
\label{verification}
Given a game as in Section \ref{s:general_game}, let $V_1,V_2:\mathbb R\to\mathbb R$ be pair of functions with the UIP, such that for any $i,j \in \{1,2\}$, $i \neq j$:
\begin{equation}
\label{QVIs}
\begin{cases}
	\begin{aligned}
		& \mathcal M_j V_j - V_j \leq 0 && \text{on} \,\,\, \mathbb R   \\
		& \mathcal H_i V_i- V_i=0 && \text{on} \,\,\, \{\mathcal M_j V_j -  V_j = 0\}\backdefeq\mathcal I^*_j  \\
		& \max\big\{\mathcal A V_i -\rho_i V_i + f_i, \mathcal M_i V_i- V_i \}=0 && \text{on} \,\,\, \{\mathcal M_j V_j - V_j < 0\}\backdefeq\mathcal C^*_j
	\end{aligned}
\end{cases}
\end{equation}
and $V_i\in C^2(\mathcal C^*_j\backslash\partial\mathcal C^*_i)\cap C^1(\mathcal C^*_j)\cap C(\mathbb R)$ has polynomial growth and bounded second derivative on some reduced neighbourhood of $\partial\mathcal C^*_i$. Suppose further $\big((\mathcal I^*_i,\delta_i^*)\big)_{i=1,2}$ are admissible strategies.\footnote{For consistency with the strategies' definition, one should assume that $\delta^*$ has been continuously extended to $\mathbb R$. The conclusion is unaffected by the choice of the extension.} 
$$\mbox{Then,}\quad
(V_1,V_2)\mbox{ are equilibrium payoffs attained at a NE }\big((\mathcal I^*_i,\delta_i^*)\big)_{i=1,2}.$$
\end{theorem}
The first equation of system (\ref{QVIs}) states that at an equilibrium, a player cannot increase her own payoff by a unilateral intervention. One therefore expects that the equality $\mathcal M_j V_j = V_j$ will only hold when player $j$ intervenes, or in other words, when the value she gains can exactly compensate the cost of her intervention. Consequently, the second equation says that a gain results from the opponent's intervention. Finally, the last one, means that when the opponent does not intervene, each player faces a single-player impulse control problem. 

We conclude this section with some final observations that we will come back to in Chapter \ref{c:3}: 

\begin{remark}
\label{r:zero_impulse}
An immediate consequence of assuming strictly positive costs is that intervening at any state with a null impulse reduces the payoff of the acting player and is therefore suboptimal. This is also displayed in system (\ref{QVIs}): if at some state $x$, $\mathcal M_jV_j(x)$ was realized for $\delta=0$, then $\mathcal M_jV_j(x)= V_j(x) + c_j(x,0)<V_j(x)$. At the same time, allowing for vanishing costs often leads to degenerate games in the current framework \cite[Sect.4.4]{ABCCV}. Hence, assuming $c_j>0$ is quite reasonable.
\end{remark}

\begin{remark}
\label{r:concave_costs}
Consider the case of nonegative impulses and cost functions being strictly concave in the impulse, as in \cite{C}. (Concave costs are also assumed in \cite{EM}.) That is, $c_i(x,\delta+\bar\delta)<c_i(x,\delta)+c_i(x+\delta,\bar\delta)$ for all $x\in\mathbb R,\ \delta,\bar\delta\geq 0$. This models the situation in which simultaneous interventions are more expensive than a single one to the same effect. In such cases, it is easy to see that in the context of Theorem \ref{verification}, player $i$ will only shift the state variable towards her continuation region.\footnote{Let $x\in\mathcal I^*_i$ and suppose $y_i^*\defeq x+\delta^*_i(x)\in\mathcal I^*_i$. Set $y_i^{**}\defeq y_i^* + \delta^*_i(y_i^*)$. Then, by the UIP, the definitions of $\delta^*_i(x)$  and $\mathcal I^*_i$, and the concavity of the cost: 
$V_i(y_i^{**})-c_i(x,\delta^*_i(x)+\delta^*_i(y_i^*))\leq V_i(y_i^*)-c_i(x,\delta^*_i(x))=V_i(y_i^{**})-c_i(y_i^*,\delta^*_i(y_i^*))-c_i(x,\delta^*_i(x))<V_i(y_i^{**})-c_i(x,\delta^*_i(x)+\delta^*_i(y_i^*))$, which is a contradiction.}
\end{remark}


\subsection{Linear game and motivation}
\label{s:linear_game}

For completeness, we finish this Section by explicitly recalling the linear game of \cite{ABCCV}, the only almost fully analytically solvable game in the literature at the time of writing, and the application that motivates it. Recall that the choice of a linear structure (as opposed to piecewise linear-quadratic as in the original impulse control formulations \cite{AF,CZ,J,MO}) is made for tractability reasons. 
\bigskip

\noindent \textbf{Central banks competing over the foreign exchange rate (FEX):} Consider two central banks that compete to influence the FEX between their respective currencies, both seeking to devalue (resp. appreciate) their own one. In this application, either central bank intervenes when it deems its currency too strong (resp. too weak). Each bank's strategy is made up of the ensuing devaluation (resp. appreciation) amount and the threshold FEX triggering it. Both of these are to be optimally determined in advance, by solving the game. Let the process $X$ model the FEX, evolving as a Brownian motion when there are no interventions. In a simplified model, we assume the running payoffs of the banks are $X-s_1$ and $s_2-X$ respectively ($s_2>s_1$) and that the costs/gains upon interventions have a fixed component and another one proportional to the size of the change induced. We further assume that both banks discount their winnings/losses at the same rate $\rho>0$ and they have the same cost/gain parameters. More specifically,
\begin{equation*}
X_t=x+ \sigma W_t + \displaystyle\sum_{k:\ \tau_1^k\leq t}\delta_1^k + \displaystyle\sum_{k:\ \tau_2^k\leq t}\delta_2^k,\quad\mbox{ and }
\end{equation*}
\begin{equation*}
J_i(x;\varphi_1,\varphi_2) \defeq \mathbb E_x \left[ \int_0^\infty e^{-\rho s} (-1)^{i-1}(X_s - s_i)ds - \sum_{k=1}^\infty e^{-\rho \tau_i^k} (c+\lambda |\delta_i^k|)+ \sum_{k=1}^\infty e^{-\rho \tau_j^k} (\tilde{c}+\tilde{\lambda} |\delta_j^k|)\right],
\end{equation*}
where $0\leq \tilde c\leq c,\ 0\leq\tilde\lambda\leq\lambda$ and $(c,\lambda)\neq (\tilde c,\tilde\lambda)$. These parametric restrictions ensure the existence of a NE (see \cite{ABCCV} and Section \ref{s:games_without_NE} in Chapter \ref{c:3}).\footnote{$1-\rho\lambda>0$ is also imposed in \cite{ABCCV}, but this is only necessary to preclude the case in which the players never intervene.}
\bigskip
 
An exact solution of this game can be found by an application of Theorem \ref{verification}. To this purpose, the system of QVIs (\ref{QVIs}) is heuristically solved, yielding candidates for equilibrium payoffs and a NE. This is done, first, by making some educated guesses. Second, by solving the ordinary differential equations (ODEs) in (\ref{QVIs}) where appropriate. Finally, by imposing the regularity requirements of Theorem \ref{verification} through pasting conditions. Upon verification of the remaining hypotheses, the following turns out to be a solution of the game:

\begin{align*}
&V_2(x)= \left\{
\begin{array}{ll}
\varphi^{A_{1},A_{2}}(x_1^*) + {\tilde c} + {\tilde{\lambda}}(x_1^*-x) &\textrm{ if } x\in(-\infty,{\bar x}_1]\\
\varphi^{A_{1},A_{2}}(x) &\textrm{ if } x\in({\bar x}_1,{\bar x}_2)\\
\varphi^{A_{1},A_{2}}(x_2^*) - c - \lambda(x-x_2^*) &\textrm{ if } x\in[{\bar x}_2,+\infty),
\end{array}
\right.
&V_1(x)= V_2(2{\tilde s}-x),	
\end{align*}
where:
\begin{align*}
&\varphi^{A_{1},A_{2}}(x)= A_{1}e^{\theta x} + A_{2}e^{-\theta x} + \frac{1}{\rho}(s_2-x),\\
&{\tilde s}:= \frac{s_1+s_2}{2}, \qquad
\theta:= \sqrt{2\rho/\sigma^2},\qquad
\eta:= (1-\lambda\rho)/\rho,\\
&{\bar x}_i:= {\tilde s} + \frac{(-1)^i}{\theta}\log{\Bigg(\sqrt{\frac{\eta+\xi}{\eta-\xi}}\Big(\sqrt{\Gamma+1}+\sqrt{\Gamma}\Big)\Bigg)},\\
& x^*_i:= {\tilde s} + \frac{(-1)^i}{\theta}\log{\Bigg(\sqrt{\frac{\eta-\xi}{\eta+\xi}}\Big(\sqrt{\Gamma+1}+\sqrt{\Gamma}\Big)\Bigg)},\\
&A_{i}:=\exp{\Big((-1)^{i}\theta {\tilde s}\Big)}
\frac{\sqrt{\eta^2-\xi^2}}{2\theta}
\Big((-1)^{i+1}\sqrt{\Gamma+1}-\sqrt{\Gamma}\Big),\\
&\Gamma:= \frac{\theta(c-{\tilde c})}{4\xi}+
\frac{\theta c(\lambda-{\tilde{\lambda}})}{4\eta\xi}+
\frac{\lambda-{\tilde\lambda}}{2\eta},
\end{align*}
and $\xi\in[0,\eta]$ is the unique zero of $F(y):= 2y  - \eta\log{\Big(\frac{\eta+y}{\eta-y}\Big)} + \theta c$.

\section{Iterative algorithm for impulse games} 
\label{s:general_algo}
\setcounter{subsection}{1}
Compared with the single-value-function problems of impulse control, general two-player games are distinctly more challenging. Recall that a deterministic approach based on Theorem \ref{verification} entails that:   
\begin{itemize}
	\item two equilibrium payoffs, governed by a system of QVIs, must be solved for,	
	\item the dependence between a pair of solutions $(V_1,V_2)$ is highly nonlinear due to the presence of the gain operators $\mathcal H_i(V_j)V_i$,
	\item the gain operators are only well-defined for pairs with the UIP,
	\item solutions will typically be less regular. For example, if $(V_1,V_2)$ is the solution in Section \ref{s:linear_game} of the linear game, then $V_i$ is singular at $\bar{x}_j$ (i.e., each equilibrium payoff is non differentiable at the border of the opponent's intervention region), in spite of the game having linear running payoffs, costs and gains. Note that Theorem \ref{verification} contemplates this lack of regularity within the smoothness assumptions (compare to the classical Verification Theorems for single-player problems \cite{OS} where greater regularity is assumed). 
\end{itemize}

Algorithm \ref{general_algo} below is, as far as we know, the first ever numerical attempt at two-player nonzero-sum stochastic impulse games. It is admittedly heuristic and supported only by the numerical evidence reported in Section \ref{s:experiments}. A rigorous analysis concerning an improved algorithm, specific to symmetric games, will be carried out in Chapter \ref{c:3}. 

The remainder of this Section is organized as follows. We start by giving a quick overview of the discretization scheme we shall be using for system of QVIs (\ref{QVIs}). (A more detailed description can be found in Chapter \ref{c:3}.) We then explain the idea and motivate the underlying heuristics of our iterative method. Finally, we list the new algorithm.\newline

\begin{notation*}
$\mathbb R^\grid$ denotes the set of functions $v:\grid\to\mathbb R$. We shall identify grid points with indexes, functions in $\mathbb R^\grid$ with vectors and linear operators with matrices. The (partial) order considered in $\mathbb R^\grid$ and $\mathbb R^{\grid\times\grid}$ is the usual pointwise order for functions (elementwise for vectors and matrices), and the same is true for the supremum, maximum and arg-maximum induced by it. All set complements are taken with respect to the grid.
\end{notation*}

\textbf{Discretization:} From now on we localize the domain and work on a discrete finite grid $\grid\subseteq \mathbb R$. In all of the numerical experiments described in the sequel we have taken $\grid$ as an equispaced grid of $M$ steps between certain $x_{\min}<0<x_{\max}$ with $|x_{\min}|,|x_{\max}|$ and $M$ large enough (more about this in Chapter \ref{c:3}).

We approximate the impulse constraint sets $\mathcal Z(x)$ with $x\in\mathbb R$ by finite sets $\emptyset\neq Z(x)$ with $x\in\grid$ and we set $Z\defeq\prod_{x\in\mathbb\grid}Z(x)$. For $i,j\in\{1,2\},\ i\neq j$, $v_1,v_2\in\mathbb R^{\grid}$ and $x\in\grid$, we define discretized versions of the loss and gain operators as
\begin{gather*}
M_iv_i(x)\defeq\max_{\delta\in Z}\big\{v_i[\![x+\delta(x)]\!]-c_i(x,\delta(x))\big\},\qquad H_i(v_j)v_i(x)\defeq v_i[\![x+\delta^*_j(x)]\!]+g_i(x,\delta^*_j(x))\\
\mbox{with}\quad\delta^*_j(x)=\delta^*_j(v_j)(x)\defeq \min\left( \argmax_{\delta\in Z}\big\{v_j[\![x+\delta(x)]\!]-c_j(x,\delta(x))\big\}\right),
\end{gather*}
where $v[\![y]\!]$ denotes linear interpolation of $v$ on $y$ using the closest nodes on the grid, and $v[\![y]\!]=v(x_{\min})$ (resp. $v[\![y]\!]=v(x_{\max})$) if $y<x_{\min}$ (resp. $y>x_{\max}$); i.e., `no extrapolation'. For the experiments presented in Section \ref{s:experiments} we assumed $\mathcal Z(x)=\mathbb R$ for all $x\in\mathbb R$ and we set $Z(x)=\grid-x=\{y-x:\ y\in\grid\}$ for simplicity.

Next, we choose a finite difference scheme for the ODEs
$$
0=\mathcal A V_i -\rho_i V_i + f_i =\frac{1}{2}\sigma^2 V_i''+\mu V_i'-\rho_i V_i + f_i
$$
which is consistent, monotone and stable, adding artificial Dirichlet or Neumann-type boundary conditions (BCs). In all of our experiments, we have chosen an upwind (or positive coefficients) finite difference scheme, where
$$
V_i'(x)\approx \frac{V_i\big(x+\sgn(\mu(x)) h\big)-V_i(x)}{\sgn(\mu(x))h},\qquad V_i''(x)\approx \frac{ V_i(x+ h) - 2 V_i(x) + V_i(x-h)}{h^2}
$$
and $V_i(x_{\min}-h),\tilde V_i(x_{\max}+h)$ were solved for using artificial Neumann conditions. (See Chapter \ref{c:3} for more details on this choice as well as that of the grid extension.) The described procedure leads to a discretization of the ODEs as 
$$L_iv_i+f_i=0,$$ 
with $-L_i$ strictly diagonally dominant (SDD) $\mbox{L}_0$-matrices (see definitions in Appendix \ref{appendix:matrices}). Note that the values of $f_i$ at $x_{\min},x_{\max}$ need to be modified to account for the boundary conditions. 
\bigskip

{\bf Heuristics:} Having discretized the system of QVIs (\ref{QVIs}), we start from an initial guess $(v_1^0,v_2^0)$ to approximate its solution. We seek an iterative procedure to consecutively compute $(v_1^{k+1},v_2^{k+1})$, given  $(v_1^k,v_2^k)$ at the $k$-th iteration.

Let $i,j\in\{1,2\},\ i\neq j$. A natural idea is, first, to partition the grid into the `approximate intervention region' of player $j$, $\{M_jv_j^k-v_j^k\geq 0\}$ (equivalent to $\{M_jv_j^k-v_j^k= 0\}$ for a solution), and its complement, the `approximate continuation region'; and then to compute $v_i^{k+1}$ either by calculating a gain as $H_i(v_j^k)v_i^k$ in the former region, or by solving an impulse control problem in the latter. Unfortunately, our numerical experiments with these hard threshold definitions lead to unstable behaviours more often than not, and a successive relaxation procedure proved to be more useful. Consequently, the approximate intervention region of player $j$ is set as $\{M_jv_j^k-v_j^k\geq -r^k\}$ instead, where $(r^k)$ are small positive numbers decreasing to zero. 

It remains to schedule the relaxation procedure and to specify a stopping criteria for the algorithm. Regarding the former, having computed $(v_1^{k+1},v_2^{k+1})$, $r^k$ is linearly relaxed. (This relaxation procedure is chosen for simplicity.). Then the largest pointwise residual to the system of QVIs (incurred by either payoff) is calculated across the grid, taking into account a numerical tolerance $tol>0$. We denote this residual $R^{k+1}$ and we consider the algorithm has `converged' when the residual drops below a certain threshold. We use the largest residual (instead of the distance between consecutive approximations) because in the absence of a convergence analysis it is a practical metric reflecting whether a solution of the discrete system of QVIs has been found (as opposed to the algorithm stagnating or converging to something other than a solution). Additionally, the residuals give valuable information at each grid node. It is not without its caveats nonetheless, as it will be seen in the experiments that floating point arithmetic can often lead to a `spiking' of the residuals at the nodes where the discrete solutions reproduce singularities of the analytical ones. Thus, the largest residual on its own can at times be artificially big and misleading, and further inspection over all grid nodes is recommended in such cases (see Section \ref{s:experiments} and more details in Chapter \ref{c:3}). We should also mention that in Chapter \ref{c:3}, aided by a convergence analysis, we will take the opposite view and the largest residual will become a tool for secondary assessment.

As a final remark, we note that numerical experiments show Algorithm \ref{general_algo} does not enjoy global convergence (i.e. it is not guaranteed to converge from arbitrary initial guesses). Providing a good enough pair $(v_1^0,v_2^0)$ is thus a practical issue; in Section \ref{s:experiments}, a natural way of constructing educated guesses is explained. 

In the following, $a^+\defeq\max\{a,0\}$ denotes the positive part of a given $a\in\mathbb R$, $\mathbbm 1_{S}$ denotes the indicator function of $S\subseteq\mathbb R$ and $\|\cdot\|_\infty$ is computed over $\grid$. We take for granted that we have a convergent solver for single-player impulse control problems. In Chapter \ref{c:3} this point will be addressed in detail, showing how one can use either classical policy iteration or fixed-point policy iteration (the latter was used in the experiments presented in Section \ref{s:experiments}). Additionally, we note that in practice a maximum number of iterations should be used as an additional stopping criteria. We neglect this to simplify the presentation.

\begin{algorithm}[H]
	\caption{Iterative algorithm for general impulse games}	\label{general_algo}
	
	Set $tol>0$ and relaxation parameters $0<\alpha<1,\ r^0>0$.
	
	\begin{algorithmic}[1]
		\STATE Choose initial guess: $(v_1^0,v_2^0)\in\mathbb R^\grid\times \mathbb R^\grid$
		\STATE Set $R^0=+\infty$ 
		\FOR{$k=0,1,\dots$}
			\FOR{i=$1,2$}
				\STATE $j=3-i$
				\STATE $I_j^k=\{M_jv_j^k-v_j^k\geq-r^k\}$
				\STATE $v_i^{k+1}= H_i(v_j^k)v_i^k$ on $I_j^k$
				\STATE Solve impulse control: $\max{\big\{L_iv_i^{k+1}+ f_i,M_iv_i^{k+1}-v_i^{k+1}\big\}=0}$ on $(I_j^k)^c$
			\ENDFOR
			\STATE $r^{k+1}=\alpha r^k$
			\STATE $I_j^{k+1,tol}=\{M_jv_j^{k+1}-v_j^{k+1}\geq-tol\}$
			\STATE $
			R^{k+1} = \Big\|\max_{\substack{i,j\in\{1,2\},i\neq j}}
			\Big\{\big(M_iv_i^{k+1}-v_i^{k+1}\big)^+, \big|H_i(v_j^{k+1})v_i^{k+1}-v_i^{k+1}\big|\mathbbm 1_{I_j^{k+1,tol}}\allowbreak
			 +\big|\max\big\{L_iv_i^{k+1}+ f_i, M_iv_i^{k+1}-v_i^{k+1}\big\}\big|\mathbbm 1_{(I_j^{k+1,tol})^c}
			\Big\}\Big\|_\infty
			$
		\IF{$R^{k+1}<tol$}
			\STATE break from the iteration
		\ENDIF
	\ENDFOR
	\end{algorithmic}
\end{algorithm}	

Line 8 of Algorithm \ref{general_algo} deserves special attention. Although at this step we want to solve an impulse control problem restricted to the subgrid $D=(I_j^k)^c$ (for fixed $k,j$), we still need the information in $I_j^k$ in two ways; namely: 
\begin{itemize}
\item To compute the non-local operator $M_i$. 
\item To restrict to $D$ the system $L_iv_i^{k+1}+f_i=0$.
\end{itemize}
It is also for this reason that the gain $v_i^{k+1}= H_i(v_j^k)v_i^k(x)$ is computed in advance in line 7, effectively making the QVI in line 8 a constrained QVI over the whole grid. Thus, one needs to algebraically modify it into a classical QVI on $D$ before an impulse control solver can be applied. See more details in Chapter \ref{c:3}.   

\section{Numerical results}
\label{s:experiments}

In this section we assess the performance of Algorithm \ref{general_algo}. The system of QVIs (\ref{QVIs}) is discretized as explained in Section \ref{s:general_algo} on an equispaced grid of $M+1$ nodes, and the computational domain is always the plotted one. We assume $\mathcal Z(x)=\mathbb R$ for all $x\in\mathbb R$ and set $Z(x)=\grid-x=\{y-x:\ y\in\grid\}$. Throughout this section, $tol=10^{-8}$, $\alpha=0.8$ and $r_0=1$. The largest pointwise residual at convergence is denoted by $R^{\infty}$.

As mentioned before, the only almost fully analytically solvable game currently available is the linear game of Section \ref{s:linear_game}. Therefore, we shall eventually focus on that problem for validation purposes. However, we introduce two other games first for which we do not have an analytical solution. For these games, we solve the corresponding QVIs systems on a fixed grid and find approximate NEs (as much as this can be claimed from a numerical viewpoint). Additionally, they illustrate how an initial guess for the linear game can be constructed and provide further numerical evidence supporting Algorithm \ref{general_algo}.

\subsection{Non-symmetric parabolic game}
This is a version of the linear game where the running payoffs are replaced by (non-symmetric) concave parabolae with roots $r_i^L$ and $r_i^R$:
\begin{equation}
\label{eq:parabolic_f}
f_i=-(x-r_i^L)(x-r_i^R),\mbox{ with } r_i^L<r_i^R.
\end{equation}

Let us motivate this game. We seek to use as initial guess the pair of value functions of the single-player versions of the game. That is, the impulse control problems in which one of the players never intervenes. Removing the action of one of the players, however, may not always lead to a well-posed problem. For example, for running payoffs without maxima one can easily end up with `infinite-valued value functions' or `infinite-valued optimal impulses'. This is indeed the case for the linear game. In order to skirt that difficulty it is convenient to consider modifications with running payoffs that attain a maximum value, like those in (\ref{eq:parabolic_f}).

Figure \ref{F:Parabolic_V} shows a numerical solution, $(v_1,v_2)$, to a parabolic game (pair of solid curves) along with the initial guess (pair of dashed curves). The latter are, in turn, the value functions of the corresponding single-player games. It is intuitively clear that $(v_1,v_2)$ approximate, over the grid, functions which indeed satisfy the assumptions of the Verification Theorem \ref{verification}. The NE exhibited in this Theorem, $(\varphi_1^*,\varphi_2^*)$, can be retrieved from the pair of equilibrium payoffs, giving $\varphi_1^*=\big([1.068,+\infty), -1.848-x\big)$ and $\varphi_2^*=\big((-\infty,-3.048],-0.120-x\big)$. 

\begin{figure}[H]
	\centering
	\includegraphics[scale=.4]{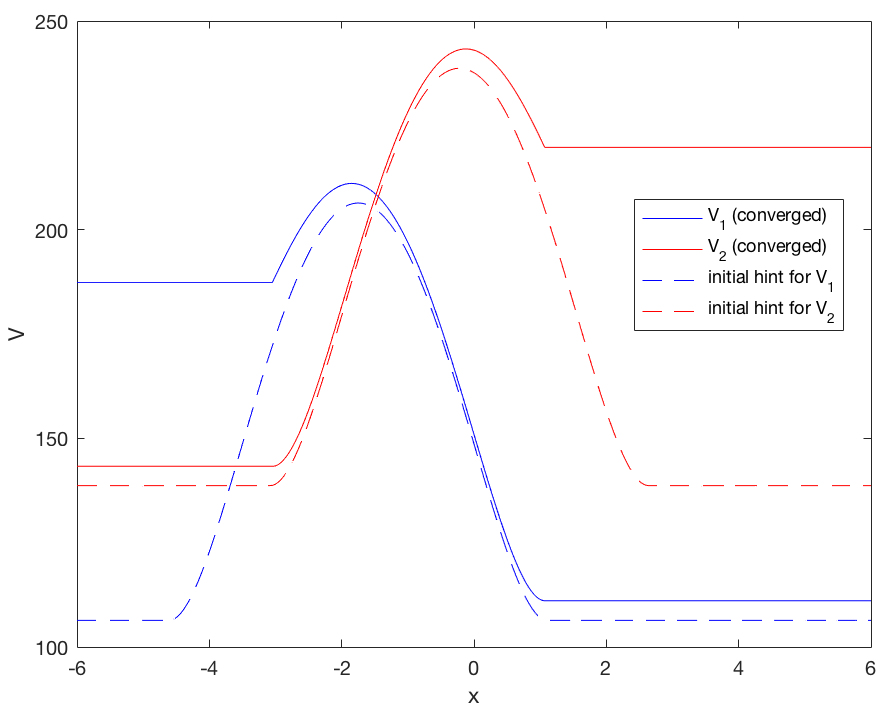}
	\caption[Solution of non-symetric parabolic game]{Equilibrium payoffs of non-symmetric parabolic game: $f_1=-(x+4.5)(x-1)$ and $f_2=-(x+\pi)(x-2.7)$, $\rho=.03$, $\sigma=.25$, $c=100$, ${\tilde c}=30$, $\lambda={\tilde \lambda}=0$. Overlaid, initial guesses (solutions of the respective single-player games). Here, $M=1000$; check Table \ref{t:parabolic} for other values.}
	\label{F:Parabolic_V}
\end{figure}

As it turns out, for this parabolic game Algorithm \ref{general_algo} also converges from the `zero guess' ($v_1^0=v_2^0=0$). On Table \ref{t:parabolic}, the value of $R^\infty$ over a wide range of finite difference grids (with $301,601,\ldots,3001$ nodes) is compiled. In all cases, either initial guess leads to convergence (up to tolerance $tol$). Nonetheless, the algorithm takes in general fewer iterations when starting from the value functions of the single-player parabolic games. (We stress that those on Table \ref{t:parabolic} are the outer iterations of Algorithm \ref{general_algo}. Within each one there is an inner loop of an impulse control solver as in Chapter \ref{c:3}.)

\begin{SCtable}[][h!]
	\centering
	\begin{tabular}{llll}
		M & $R^\infty$ & its.$^a$ &  its.$^b$ \\ 
		\noalign{\smallskip}\hline\noalign{\smallskip}
		300  &  1.4$\times 10^{-12}$ & 54 & 53 \\
		600 &   4.0$\times 10^{-9}$ &  74 & 77 \\
		1200 & 3.3$\times 10^{-9}$ & 144 & 77 \\
		1800 & 9.7$\times 10^{-9}$& 95 & 77 \\
		2400 & 7.3$\times 10^{-9}$ & 123 & 77 \\
		3000 & 5.9$\times 10^{-9}$ & 215 & 103\\
		\noalign{\smallskip}\hline
	\end{tabular}
	\caption[Convergence tests: non-symmetric parabolic game]{Largest residual to QVIs at convergence ($R^\infty$) vs. number of grid nodes ($M+1$) for parabolic game in Figure \ref{F:Parabolic_V}. Iterations to convergence within $tol=10^{-8}$ starting from: zero guess (its.$^a$) and value functions of single-player problems (its.$^b$).}
\label{t:parabolic}
\end{SCtable}

We conclude this example by illustrating the interplay between the equilibrium payoffs found with Algorithm \ref{general_algo}, the NE derived from them and the evolution of the optimally controlled state variable. Once the optimal strategies $(\varphi_1^*,\varphi_2^*)$ are available, they can be executed on specific realizations of the game. Sticking to the parameters and numerical solution in Figure \ref{F:Parabolic_V}, Figure \ref{f:path} depicts one exemplary path of the state process in the time interval $[0,1000]$, starting from $x=0$ and subjected to the pair of optimal strategies. For numerical purposes, let us consider the finite horizon payoffs:
\begin{equation}
\label{J_aux}
\hspace*{-.2cm}
\hat J_i(x;\varphi_1,\varphi_2,T) \defeq \mathbb E \left[ \int_0^T e^{-\rho_i s} f_i(X_s)ds - \sum_{\substack{k:\\ \tau_i^k\leq T}} e^{-\rho_i \tau_i^k} c_i \big( X_{(\tau_i^k)^-}, \delta_i^k \big) + \sum_{\substack{k:\\ \tau_j^k\leq T}} e^{-\rho_i \tau_j^k} g_i \big( X_{(\tau_j^k)^-}, \delta_j^k \big)\right].
\end{equation}

\noindent For `good enough' games and strategies, one intuitively expects that $\hat J_i(x;\varphi_1,\varphi_2,T)\to J_i(x;\varphi_1,\varphi_2)$ as $T\to\infty$. In fact, after $T\gtrsim 300$, the integrals in (\ref{J_aux}) for the parabolic game and NE of Figure \ref{F:Parabolic_V} have essentially attained their asymptotic value. Thus, we simply take $J_i(x;\varphi^*_1,\varphi^*_2)\approx \hat J_i(x;\varphi^*_1,\varphi^*_2,T=1000)$. With this clarification, Figure \ref{F:Ex} shows in particular $v_1(0)=J^1(0;\varphi_1^*,\varphi_2^*)$, $v_2(0)=J^2(0;\varphi_1^*,\varphi_2^*)$,
	$v_1(-1)=J^1(-1;\varphi_1^*,\varphi_2^*)$, and
	$v_2(-1)=J^2(-1;\varphi_1^*,\varphi_2^*)$, obtained by Monte Carlo simulations.\footnote{The expected values in (\ref{J_aux}) are approximated by the mean over $N=200$ realizations integrated with the Euler-Maruyama method with time step $\Delta t=0.001$.} They compare fairly well with the values of $v_1(0)$, $v_2(0)$, $v_1(-1)$ and $v_2(-1)$ in Figure \ref{F:Parabolic_V} obtained by our algorithm. (Even better agreement could be obtained by increasing $M$ in that figure and reducing the discretization bias and statistical error of the Monte Carlo simulation, but this is good enough to make our point.)

\begin{figure}[H]
	\centering
	\includegraphics[scale=.25]{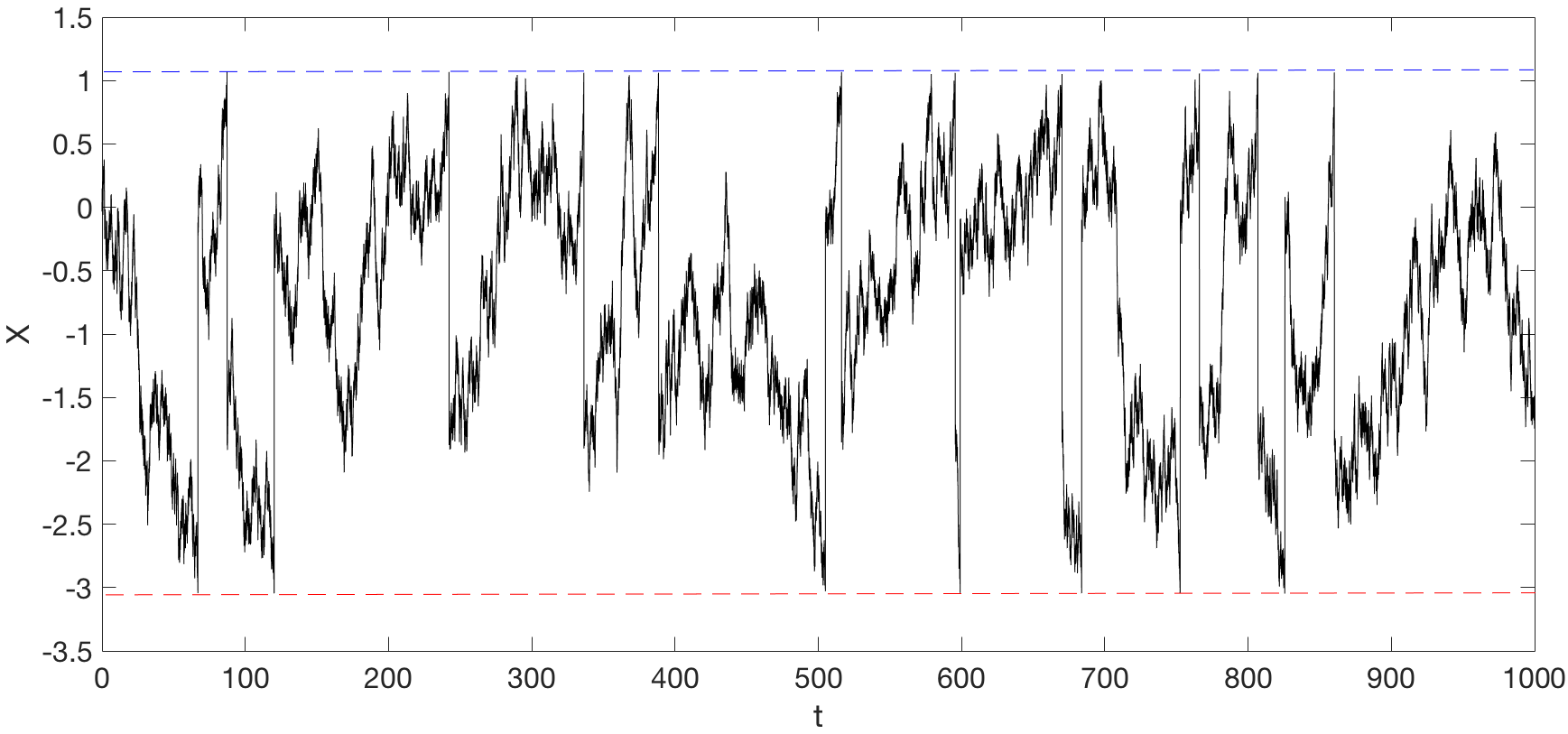}
	\caption[Sample paths of nonsymetric parabolic game]{Sample path from $x=0$ with parameters and solution from Figure \ref{F:Parabolic_V}. Blue and red dashed lines are intervention thresholds for players 1 and 2, respectively.}
	\label{f:path}
\end{figure}

\begin{figure}[H]
	\centering
	\includegraphics[scale=.4]{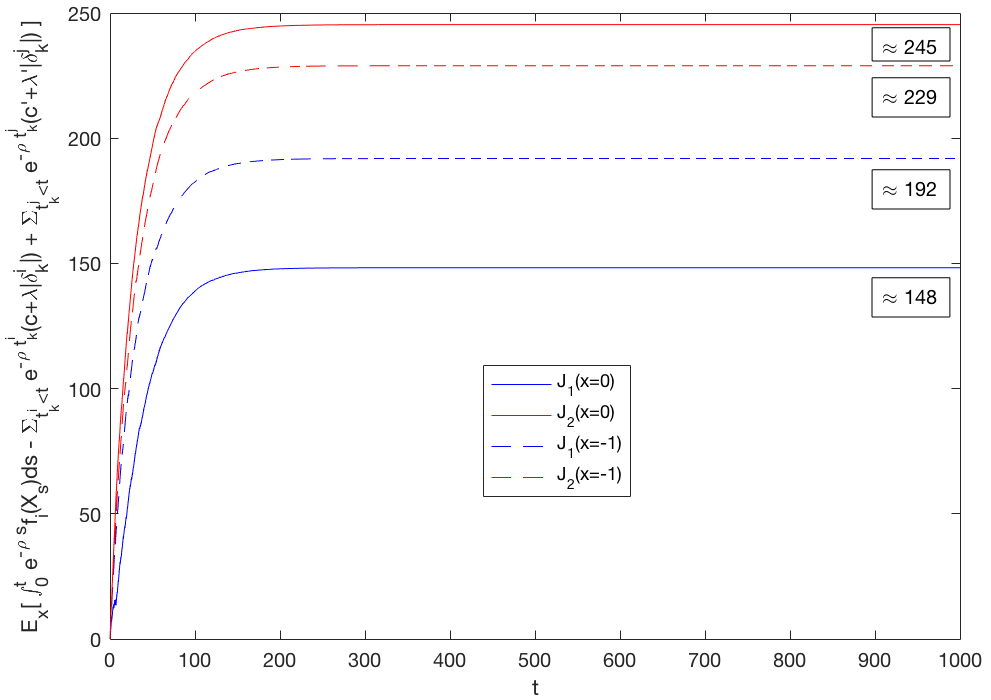}
	\caption[Monte Carlo simulations: equilibrium payoffs of non-symmetric parabolic game]{Approximations of the equilibrium payoffs at $x=0$ (solid curves) and $x=-1$ (dashed curves), obtained by Monte Carlo simulation (see text for details). Parameters and optimal strategies from Figure \ref{F:Parabolic_V}. Compare with $v_1(0)$, $v_2(0)$, $v_1(-1)$ and $v_2(-1)$ there.}
	\label{F:Ex}
\end{figure}

The NE itself can be visually explored in the following way. For a given starting point $x$, we keep the optimal strategy for one of the two players, and slightly alter the strategy of the other one. For concreteness, let us assume that player 1 changes her strategy to $\varphi_1=(1\pm 0.25{\cal U})\varphi^*_1$ while player 2 maintains $\varphi_2=\varphi^*_2$, where ${\cal U}$ is the uniform distribution, `$\pm$' means `with equal chance'.\footnote{Changes are applied to each of the two parameters defining $\varphi^*_1$.} Then, we proceed to calculate $J^1(x;\varphi_1,\varphi_2^*)$ by Monte Carlo simulations as before. By definition of the NE, $J^1(x;\varphi_1,\varphi_2^*)$ cannot be larger than $v_1(x)$. Within numerical tolerance, this is indeed observed in Figure \ref{F:Nash_eq}, where the blue empty circles ($J^1(x;\varphi_1,\varphi_2^*)$) do not lie over the blue curve ($v_1(x)$). Full red circles represent $J^2(x;\varphi_1,\varphi_2^*)$: note that the player who sticks to her optimal strategy may indeed improve over $v_2(x)$, should her opponent depart from a NE. (When player 2 is the one who changes, the red empty circles, red curve and full blue circles apply instead.) We stress, however, that Monte Carlo simulations such as these cannot {\em prove} that a pair of strategies form a NE. (At most, they can disprove it.) This is why Algorithm \ref{general_algo} becomes a valuable tool in solving the underlying system of QVIs and applying the Verification Theorem \ref{verification}.

\begin{figure}[H]
	\centering
	\includegraphics[scale=.32]{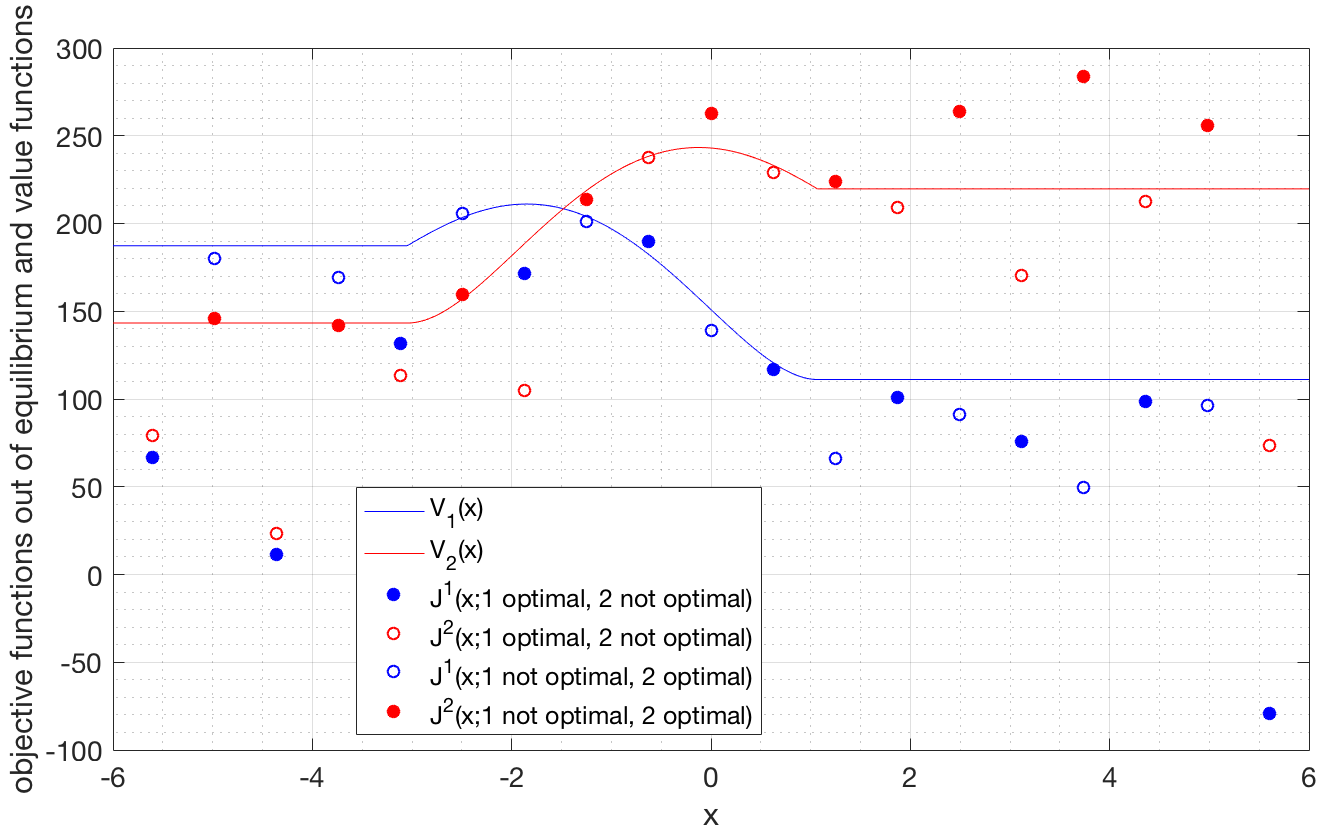}
	\caption[Monte Carlo simulations: Nash equilibrium of non-symmetric parabolic game]{Empty blue (resp. red) circles mean that the player 1 (resp. 2) has departed from her optimal strategy while her opponent has not. By virtue of NE, she cannot be (within numerical tolerance) better off than $v_1$ (resp. $v_2$). Full circles: payoff of the player who does not change her optimal strategy. Parameters and optimal strategies from Figure \ref{F:Parabolic_V}.  (Note that results are subject to numerical error.)}
	\label{F:Nash_eq}
\end{figure}

\setcounter{subsection}{1}
\subsection{Capped linear game}

In this game, we cap the running payoffs of the linear game at $K>0$:
\begin{equation}
\label{eq:capped_f}
f_i=(-1)^{i-1}(x-s_i)\wedge K.\footnote{$a\wedge b$ denotes the minimum of $a$ and $b$ for any $a,b\in\mathbb R$.}
\end{equation}
We shall always take $K=5$. Once again, the corresponding single-player games are well-posed and their solutions can be used as initial guess for the capped linear game. As in the previous example, the capped game also seems to converge from the zero guess. Some convergence results are compiled in Table \ref{T:Capped}. Note that convergence falters with $M=600$ and $M=3300$; we will come back to this later.

\begin{SCtable}[][h!]
	\centering
	\begin{tabular}{lll}
		M & $R^\infty$ & its.\\
		
		\noalign{\smallskip}\hline\noalign{\smallskip}
		600 &---& $\infty$ \\
		900 & $6.5\times 10^{-10}$ & 100 \\
		1200 & $8.8\times 10^{-9}$ & 124 \\
		1500 & $7.4\times 10^{-9}$ & 78 \\
		2700 & $4.1\times 10^{-9}$ & 96 \\
		3000 & $6.7\times 10^{-9}$ & 125 \\
		3300 &---& $\infty$\\
		\noalign{\smallskip}\hline
	\end{tabular}
	\caption[Convergence tests: capped linear game]{Iterations needed for convergence and largest terminal residual ($R^\infty$) in capped linear game using zero initial guess (a hyphen stands for lack of convergence within $tol$). Same parameters as in Figure \ref{F:Capped_error}.}
	\label{T:Capped} 
\end{SCtable}

\begin{figure}[H]
	\centering
	\includegraphics[scale=.33]{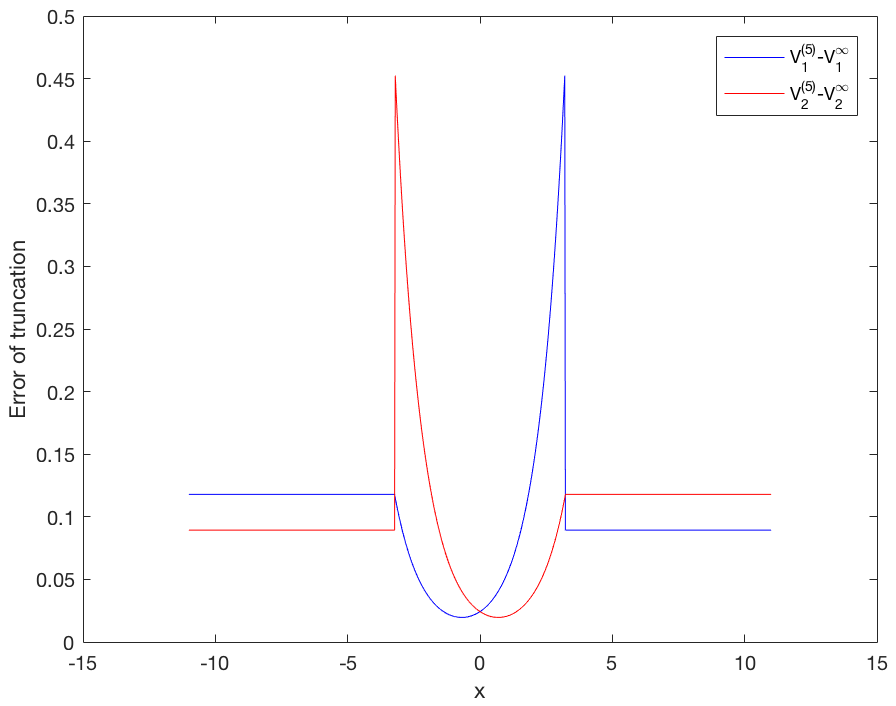}
	\caption[Capped game as approximation of linear game]{Difference (in absolute value, less than $1\%$) between (numerical, with $M=1200$) equilibrium payoffs of the capped linear game and those (exact) of the linear game from Section \ref{s:linear_game}, justifying the former payoffs being used as initial guess for the latter game. Parameters: $\sigma=.25,\ \rho=.03,\ c=100,\ {\tilde c}=30,\ \lambda=0.5,\ {\tilde\lambda}=0.3,\ s_1=-\pi/3,\ s_2=\pi/3,\ K=5$.}
	\label{F:Capped_error}
\end{figure}

The equilibrium payoffs found for the capped linear game are a good approximation to those of the linear game shown in Section \ref{s:linear_game} (see Figure \ref{F:Capped_error}). This seems to make sense: due to the action of the opponent and for $\sigma\ll K$, the discarded portion of the payoff is not very relevant in practice.

\subsection{Linear game with educated initial guess}
Finally, we tackle the linear game, for which an exact solution is available (see Section \ref{s:linear_game}). Contrary to the previous examples, Algorithm \ref{general_algo} does not seem to enjoy unconditional convergence here. In fact, when the zero guess was used, it failed to converge more often than not (not reported). In order to construct an adequate initial guess, we first solve for some equilibrium payoffs of the capped linear game. Using them as the initial guess, convergence was achieved in every experiment.

\begin{figure}[H]
	\hspace*{-.85cm}
	\includegraphics[scale=.33]{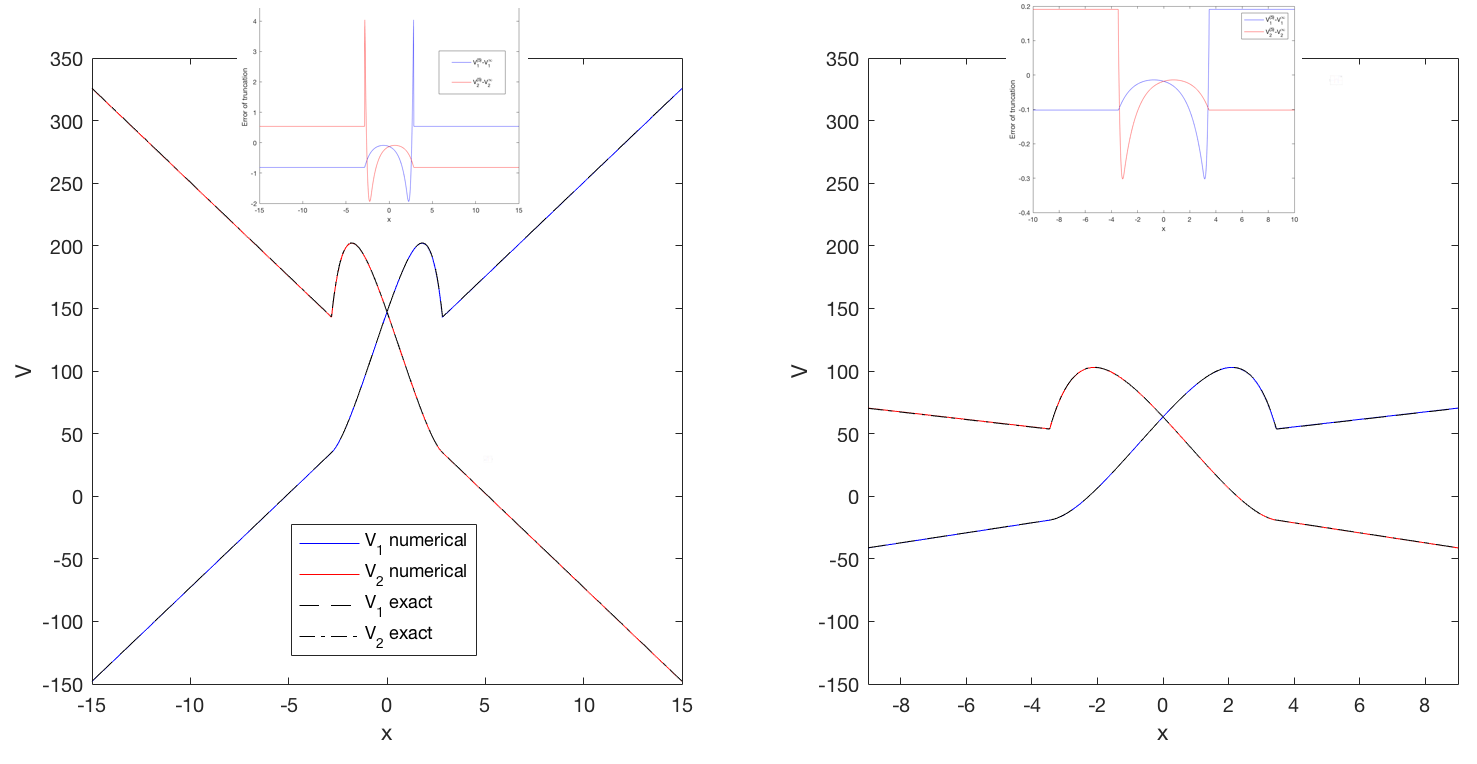}
	\caption[Numerical solution of linear game]{Equilibrium payoffs of two instances of linear game. Initial guess: solutions of capped linear games ($K=5$). Overlaid, error of the initial guess.
		Parameters: $\rho=.02$, $\sigma=.15$, $s_1=-3$, $s_2=3$, $c=100$, ${\tilde c}=0$, $\lambda={\tilde \lambda}=15$,  
		$M=1000$  (left); $\rho=.03$, $\sigma=.25$, $s_1=-2$, $s_2=2$, $c=100$, ${\tilde c}=30$, $\lambda=4$, ${\tilde \lambda}=3$, 
		$M=1000$  (right).} 
	\label{F:Benchmark_V}
\end{figure}

The results of two experiments are plotted on Figure \ref{F:Benchmark_V}. The numerical approximations can hardly be distinguished from the exact solutions with the naked eye. As it was done with the parabolic game, once again we can retrieve an approximate NE from the numerical solution. For the left figure, this gives $\varphi_1^*=\big((-\infty,-2.82],1.53-x\big)$ and $\varphi_2^*=\big([2.82,+\infty),-1.53-x\big)$. When compared with the exact equilibrium, the errors on the corresponding abscissae are smaller than the grid step size.

\begin{SCtable}[][h!]
	\centering
	\begin{tabular}{llll}
		M & $R^\infty$ & $|$error$|$ & its.\\
		\noalign{\smallskip}\hline\noalign{\smallskip}
		500 & 3.8$\times 10^{-10}$ & 0.687 & 126 \\
		1000 & 1.2$\times 10^{-9}$ & 0.805 & 154 \\
		1500 & 7.2$\times 10^{-9}$ & 0.512 & 157 \\
		2000 & 3.5$\times 10^{-9}$ & 0.365 & 172 \\
		2500 &---&---& $\infty$\\ 
		\noalign{\smallskip}\hline
	\end{tabular}
	\caption[Convergence tests: linear game (first)]{Convergence of Algorithm \ref{general_algo} for linear game. Same parameters as in Figure \ref{F:Benchmark_V} (left).}
	\label{T:Benchmark_ejemploB} 
\end{SCtable}

\begin{SCtable}[][h!]
	\centering
	\begin{tabular}{llll}
		M & $R^\infty$ & $|$error$|$ & its.\\
		\noalign{\smallskip}\hline\noalign{\smallskip}
		600 &---&---& $\infty$ \\
		800 & 9.5$\times 10^{-10}$ & 0.023 & 183 \\
		1000 & 3.7$\times 10^{-9}$ & 0.330 & 177 \\
		1400 & 8.7$\times 10^{-9}$ & 0.196 & 159 \\
		1800 & 6.8$\times 10^{-9}$ & 0.121 & 226 \\
		2200 & 5.6$\times 10^{-9}$ & 0.073 & 224 \\
		2600 &---&---& $\infty$\\ 
		\noalign{\smallskip}\hline
	\end{tabular}
	\caption[Convergence tests: linear game (second)]{Convergence of Algorithm \ref{general_algo} for linear game (same parameters as in Figure \ref{F:Benchmark_V} (right).}
	\label{T:Benchmark_ejemploD} 
\end{SCtable}

On Tables \ref{T:Benchmark_ejemploB} and \ref{T:Benchmark_ejemploD}, the convergence of the numerical approximation provided by Algorithm \ref{general_algo} to the true solution is demonstrated. However, $R^\infty$ fails to drop below $tol$ for some discretizations as the algorithm stagnates, with the highest residuals found on and around the nodes where the payoffs display a singularity. Figure \ref{F:Stagnation} illustrates this situation. It can be seen that errors continue to be acceptable (far less than $1\%$ relative error in the worst case), but further refinements of the grid did not yield any major improvements. The algorithm for symmetric games, presented and studied in Chapter \ref{c:3}, is far better-behaved in this respect as well. It will be seen that for games in which one intuitively expects a solution to exist, cases of stagnation were always resolved by refining the grid. 

\begin{figure}[H]
\centering
	\includegraphics[scale=.5]{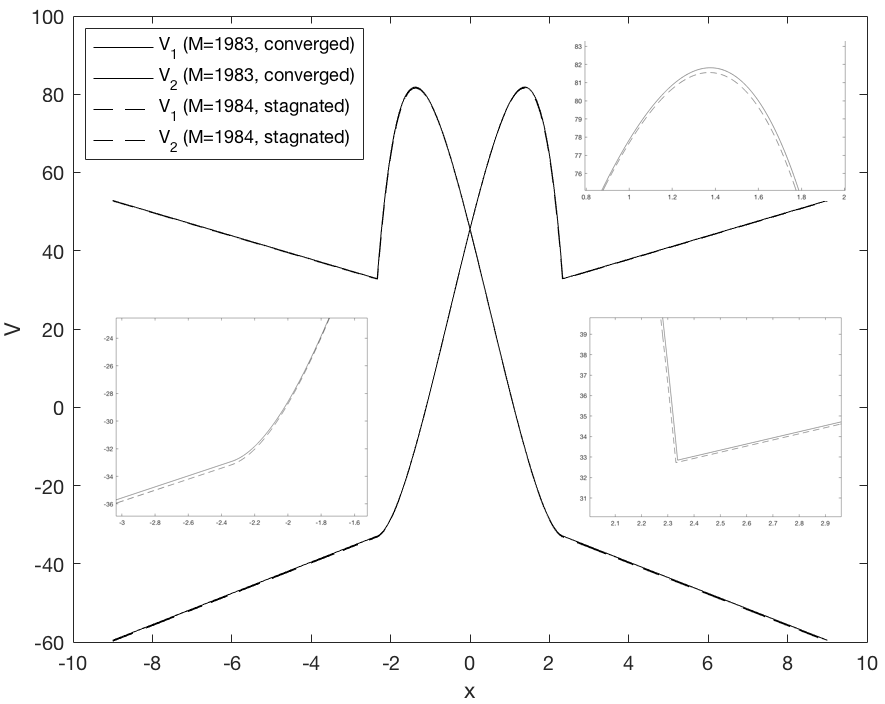}
	\caption[Stagnation of accuracy of non-symmetric algorithm]{Stagnation of accuracy: the pair with $M=1983$ is fully converged (within numerical tolerance); the pair with $M=1984$ (overlaid) is not. The insets zoom in on both pairs of functions. Parameters: $\rho=.02$, $\sigma=.15$, $s_1=-1$, $s_2=1$, $c=100$, ${\tilde c}=30$, $\lambda=4$, ${\tilde\lambda}=3$.} 
	\label{F:Stagnation}
\end{figure}

\section{Concluding remarks}
\label{s:conclusions_chapter2}
In this chapter, we have designed and tested a novel policy iteration algorithm (the first one as far as we know) to numerically solve nonzero-sum stochastic impulse control games. The approach consists in solving a system of QVIs which characterizes certain equilibrium payoffs and NEs, exploiting the Verification Theorem of \cite{ABCCV}.

Our algorithm iteratively computes the approximate
solution by partitioning, for each of the players, the discretized localized spatial domain into approximate continuation and intervention regions of the opponent. These are defined through a relaxation
parameter that evolves along the iterations. In the continuation
region, we solve a constrained single-player impulse control problem, whereas in the complement, a gain is computed. A way of producing an educated initial guess to start the iterations, which relies on solving standard impulse control problems, has been presented along with the new algorithm.
 
The algorithm was validated numerically with satisfactory results, by considering different games and both fixed and refining grids, but it was noted that cases of stagnation could not be resolved by simple grid refinements. In particular, we have tested the convergence to an analytical solution of the linear game, the only one almost fully tractable at the time of writing. Results show that the algorithm offers a means of gaining quantitative insight into applications modelled by general nonzero-sum stochastic impulse games. In the case of symmetric games however, we shall see in the next chapter a modified algorithm which substantially improves the current one, and we will provide the missing convergence analysis.


\chapter[A fixed-point policy-iteration-type algorithm for symmetric impulse games]{A fixed-point policy-iteration-type algorithm for symmetric nonzero-sum stochastic impulse games}
\chaptermark{Algorithm for symmetric impulse games}
\label{c:3}

\phantomsection
\section*{Introduction}
\addcontentsline{toc}{section}{Introduction}

To the author's best knowledge, the only numerical method available in the literature of nonzero-sum stochastic impulse games is our algorithm in \cite{ABMZZ} (Algorithm \ref{general_algo} in Chapter \ref{c:2}), which tackles the system of QVIs (\ref{QVIs}) by sequentially solving single-player impulse control problems combined with a relaxation scheme. Unfortunately, the choice of the latter scheme is not obvious in general and it was verified that the convergence of the algorithm was reliant on a good initial guess. It was also observed that stagnation could put a cap on the accuracy of the results, without any simple solution to it. Lastly, while a numerical validation was performed, no rigorous convergence analysis was provided at the time.

Restricting attention to the one-dimensional infinite horizon two-player case, this chapter (based on \cite{Z}) puts the focus on certain nonzero-sum impulse games which display a symmetric structure between the players. This class is broad enough to include many interesting applications; no less than the competing central banks problem (whether in its linear form \cite{ABCCV} or others considered in the single bank formulation \cite{AF,CZ,J,MO}), the cash management problem \cite{BCG} (reducing its dimension by a simple change of variables) and the generalization of many impulse control problems to the two-player case. 

For this class of games, an iterative algorithm is presented which substantially improves Algorithm \ref{general_algo} by harnessing the symmetry of the problem, removing the strong dependence on the initial guess and dispensing with the relaxation scheme altogether. The result is a simpler and more intuitive, precise and efficient routine, for which a convergence analysis is provided. It is shown that the overall routine admits a representation that strongly resembles, both algorithmically and in its properties, that of the combined fixed-point policy-iteration methods \cite{HFL1,Cl0}, albeit with nonexpansive operators. Still, a certain contraction property can still be established.

To perform the analysis, assumptions are imposed on the discretization scheme used on the system of QVIs and the discrete admissible strategies. These naturally generalize those of the impulse control case \cite{AF} and admit graph-theoretic interpretations in terms of weakly chained diagonally dominant (WCDD) matrices and their matrix sequences counterpart \cite{A}. We establish a clear parallel between these discrete type assumptions, the behaviour of the players and the mimicking of the Verification Theorem. 

Section \ref{s:symmetric_games} deals with the analytical problem. We give a precise definition of the class of symmetric nonzero-sum impulse games and establish some preliminary results.

Section \ref{s:discrete_problem} moves on the analogue discrete problem. Section \ref{s:dQVIs} specifies a general and abstract discrete version of the system of QVIs, such that any discretization scheme compliant with the assumptions to be imposed will be subject to the same results. Section \ref{s:sym_algo} presents the iterative algorithm subject of this chapter, and shows how the impulse control problems that need to be sequentially solved have a unique solution and can be handled by policy iteration. An alternative general solver for impulse control problems is also provided in Section \ref{s:solve_one_player}. It consists of an instance of fixed-point policy-iteration that is noncompliant with the standard assumptions \cite{HFL1} and, as far as the author knows, it was not used in the context of impulse control before, other than heuristically in Chapter \ref{c:2}. We prove its convergence under our framework. 

Section \ref{s:FPPI} characterizes the overall iterative algorithm as a fixed-point policy-iteration-type method, allowing for reformulations of the original problem and results pertaining the solutions. (The necessary matrix and graph-theoretic definitions and results are collected in Appendix \ref{appendix:matrices} for the reader's convenience.) Section \ref{s:convergence} carries on with the convergence analysis, and shows to which extent different sets of reasonable assumptions are enough to guarantee convergence to solutions, convergence of strategies and boundedness of iterates. A result of theoretical interest, giving sufficient conditions for convergence, is proved. Discretization schemes within the standing framework are provided in Section \ref{s:discretization}.

Section \ref{s:numerics_games} presents all the numerical results. In Section \ref{s:games_on_fixed_grid}, a variety of symmetric nonzero-sum impulse games, many seemingly too complicated to be handled analytically, are explicitly solved for equilibrium payoffs and NEs strategies with great precision. This is done on a fixed grid, while considering different performance metrics and addressing practical matters of implementation. In the absence of a viscosity solutions framework to establish convergence to analytical solutions as the grid is refined, Section \ref{cv_analytical_sol} performs a numerical validation using the only examples of symmetric solvable games in the literature. Section \ref{s:games_without_NE} addresses the case of games without NEs. Section \ref{s:beyond_verif_theo} tackles games such that the results go beyond the scope the currently available theory, displaying discontinuous impulses and very irregular payoffs. The latter give insight and motivate further research into this field. Lastly, Section \ref{s:conclusions} concludes.


\section{Analytical continuous-space problem:\\the symmetric case}
\label{s:symmetric_games}
\setcounter{subsection}{1}
In Section \ref{s:analytical_problem} we established a general framework for two-player nonzero-sum stochastic impulse control games. We now want to focus our study on games which present a certain type of symmetric structure between the players, generalising the linear game in Section \ref{s:linear_game} and the cash management game \cite{BCG}.\footnote{The latter can be reduced to one dimension with the change of variables $x=x_1-x_2$. Additionally, we will restrict attention to unidirectional impulses, as these yield the `most relevant' NE found in \cite{BCG}.} We shall make a slight abuse of terminology with respect to the more common use of the term `symmetric' in games theory, but this will be consistent all throughout.

\begin{notation*}
The type of games presented in Section \ref{s:general_game} are fully defined by setting the drift, volatility, impulse constraints, discount rates, running payoffs, costs and gains. In other words, any such game can be represented by a tuple $\mathcal G=(\mu,\sigma,\mathcal Z_i,\rho_i,f_i,c_i,g_i)_{i=1,2}$.
\end{notation*} 

\begin{definitions}
\label{sym_game}
We say that a game $\mathcal G=(\mu,\sigma,\mathcal Z_i,\rho_i,f_i,c_i,g_i)_{i=1,2}$ is \textit{symmetric (with respect to zero)} if
	\begin{enumerate}[label=(S\arabic*)]
		\item \label{sym_dynamics} $\mu$ is odd and $\sigma$ is even (i.e., $\mu(x)=-\mu(-x)$ and $\sigma(x)=\sigma(-x)$ for all $x\in\mathbb R$).
		\item \label{sym_constraints} $-\mathcal Z_2(-x)=\mathcal Z_1(x)\subseteq [0,+\infty)$ for all $x\in\mathbb R$ and $\mathcal Z_1(x)=\{0\}=\mathcal Z_2(-x)$ for all $x\geq 0$.
		\item \label{sym_rest} $\rho_1=\rho_2$, $f_1(x)=f_2(-x)$, $c_1(x,\delta)=c_2(-x,-\delta)$ and $g_1(x,-\delta)=g_2(-x,\delta)$, for all $\delta\in\mathcal Z_1(x),\ x\in\mathbb R$.
	\end{enumerate}
We say that the game is \textit{symmetric with respect to }$s$ (for some $s\in\mathbb R$), if the $s$-\textit{shifted} game $(\mu(x+s),\sigma(x+s),\mathcal Z_i(x+s),\rho_i,f_i(x+s),c_i(x+s,\delta),g_i(x+s,\delta))_{i=1,2}$ is symmetric. We refer to $x=s$ as a \textit{symmetry line} of the game. 
\end{definitions}
Condition \ref{sym_dynamics} is necessary for the state variable to have symmetric dynamics. In particular, together with \ref{sym_rest}, it guarantees symmetry between solutions of the Hamilton--Jacobi--Bellman (HJB) equations of the players when there are no interventions, i.e.,
$$
V=V^*(x)\mbox{ solves }\mathcal AV -\rho_1V +f_1=0\quad\mbox{ if and only if }\quad V=V^*(-x)\mbox{ solves }\mathcal AV -\rho_2V +f_2=0.
$$
\begin{examples}
The most common examples of It\^o diffusions satisfying this assumption are the scaled Brownian motion (symmetric with respect to zero) and the Ornstein--Uhlenbeck (OU) process (symmetric with respect to its long term mean). 
\end{examples}

Condition \ref{sym_rest} is self-explanatory, while \ref{sym_constraints} is only partly so. Indeed, although symmetric constraints on the impulses $\mathcal Z_1(x)=-\mathcal Z_2(-x)$ should clearly be a requirement, the rest of \textit{(ii)} is in fact motivated by the numerical method to be presented and the type of problems it can handle. On the one hand, the third equation of the QVIs system (\ref{QVIs}) implies that a stochastic impulse control problem for player $i$ needs to be solved on $\mathcal C^*_j$. The unidirectional impulses assumption is a common one for the convergence of policy iteration algorithms in impulse control.\footnote{See this assumption in \cite[Sect.4]{CMS} or \cite[Sect.10.4.2]{OS}, its graph-theoretic counterpart in \cite[Asm.(H2) and Thm.4.3]{AF}, and a counterexample of convergence in its absence in \cite[Ex.4.9]{AF}.} However, it is often too restrictive for many interesting applications, such as when the controller would benefit the most from keeping the state variable within some bounded interval instead of simply keeping it `high' or `low' (see, e.g., \cite{B1} and \cite[Sect.6.1]{AF}). Interestingly enough, assuming unidirectional impulses turns out to be less restrictive when there is a second player present, with an opposed objective. Indeed, it is often the case that each player needs not to intervene in one of the two directions, and can instead rely on her opponent doing so, while capitalising a gain rather than paying a cost. See examples in Section \ref{s:games_on_fixed_grid} with quadratic and degree four running payoffs. 

On the other hand, $\mathcal Z_1(x)=\{0\}=\mathcal Z_2(-x)$ for all $x\geq 0$ means that we can assume without loss of generality that the admissible intervention regions do not cross over the symmetry line; i.e., $\mathcal I_1\subseteq(-\infty,0)$ and $\mathcal I_2\subseteq (0,+\infty)$ for every pair of strategies. (See Remark \ref{r:zero_impulse}.) This guarantees in particular that the players never want to intervene at the same time and the priority rule can be disregarded.   

There are different reasons why the last mentioned condition is less restrictive than it first appears to be. It is not uncommon to assume connectedness of either intervention or continuation regions (or other conditions implying them) both in impulse control \cite{E} and nonzero-sum games \cite[Sect.1.2.1]{DFM}. The same can be said for assumptions that prevent the players from intervening in unison \cite[Sect.1.2.1]{DFM},\cite[Rmk.6.5]{C}.
\footnote{In \cite[Rmk.6.5]{C}, assumptions are given to guarantee the zero-sum analogue to $\mathcal M_iV_i\leq \mathcal H_iV_i$, with a strict inequality if such assumptions are slightly strengthened. These inequalities, in the context of Theorem \ref{verification}, imply that the equilibrium intervention regions cannot intersect.} In the context of symmetric games and payoffs (see Lemma \ref{sym_lemma}) such assumptions would necessarily imply the intervention regions need to be on opposed sides of the symmetry line. Additionally, without any further requirements, strategies such that $\mathcal I_1\supseteq(-\infty,0]$ and $\mathcal I_2\supseteq [0,+\infty)$ would be inadmissible in the present framework, as per yielding infinite simultaneous impulses.

\begin{definitions}
\label{sym_strategies}
Given a symmetric game, we say that $\big((\mathcal I_i,\delta_i)\big)_{i=1,2}$ are \textit{symmetric strategies (with respect to zero)} if $\mathcal I_1=-\mathcal I_2$ and $\delta_1(x)=-\delta_2(-x)$. Given a symmetric game with respect to some $s\in\mathbb R$, we say that$\big((\mathcal I_i,\delta_i)\big)_{i=1,2}$ are \textit{symmetric strategies with respect to $s$} if $\big((\mathcal I_i-s,\delta_i(x+s))\big)_{i=1,2}$ are symmetric, and we refer to $x=s$ as a \textit{symmetry line} of the strategies.
\end{definitions}

\begin{definition}
\label{sym_functions}
We say that $V_1,V_2:\mathbb R\to\mathbb R$ are \textit{symmetric functions (with respect to zero)} if $V_1(x)=V_2(-x)$. We say that they are \textit{symmetric functions with respect to $s$} (for some $s\in\mathbb R$) if $V_1(x+s),V_2(x+s)$ are symmetric, and we refer to $x=s$ as a \textit{symmetry line} for $V_1,V_2$. 
\end{definition}

\begin{remark}
\label{r:sym_NE}
Definition \ref{sym_strategies} singles out strategies that share the same symmetry line with the game. For the linear game, for example, the authors find infinitely many NEs \cite[Prop.4.7]{ABCCV}, each presenting symmetry with respect to some point $s$, but only one for $s=0$ (hence, symmetric in the sense of Definition \ref{sym_strategies}). At the same time, the latter is the only one for which the corresponding equilibrium payoffs $V_1,V_2$ have a symmetry line as per Definition \ref{sym_functions}. The same is true for the cash management game \cite{BCG}.
\end{remark}

\begin{remark}
Throughout the chapter we will work only with games symmetric with respect to zero, to simplify the notation. Working with any other symmetry line amounts simply to shifting the game and results back and forth.
\end{remark}

\begin{lemma}
\label{sym_lemma}
For any symmetric game, strategies $(\varphi_1,\varphi_2)$ and functions $V_1,V_2:\mathbb R\to\mathbb R$:
\begin{enumerate}[label=(\roman*)]
\item If $V_1,V_2$ are symmetric, then $\mathcal M_1V_1,\mathcal M_2V_2$ are symmetric.
\item If $V_1,V_2$ are symmetric and have the UIP, then $\delta_1^*(x)=-\delta_2^*(-x)$ and $\mathcal H_1V_1,\mathcal H_2V_2$ are symmetric. 
\item If $(\varphi_1,\varphi_2)$ are symmetric, then $J_1(\cdot;\varphi_1,\varphi_2),J_2(\cdot;\varphi_1,\varphi_2)$ are symmetric. 
\item If $V_1,V_2$ are as in Theorem \ref{verification} and $(\varphi^*_1,\varphi^*_2)$ is the corresponding NE of the theorem, then $(\varphi^*_1,\varphi^*_2)$ are symmetric if and only if $V_1,V_2$ are symmetric. 
\end{enumerate}
\end{lemma}

\begin{proof}
\textit{(i)} and \textit{(ii)} are straightforward from the definitions. 

To see \textit{(iii)}, one can check with the recursive definition of the state variable \cite[Def.2.2]{ABCCV} that $X^{-x;\varphi_1,\varphi_2}$ has the same law as $-X^{x;\varphi_1,\varphi_2}$ (recall that the continuation regions are simply disjoint unions of open intervals). Noting also that intervention times and impulses are nothing but jump times and sizes of $X$, one concludes that $J_1(x;\varphi_1,\varphi_2)=J_2(-x;\varphi_1,\varphi_2)$, as intended.  

Finally, \textit{(iv)} is a consequence of \textit{(i)}, \textit{(ii)} and \textit{(iii)}. 
\end{proof}

\begin{convention}
\label{convention}
In light of Lemma \ref{sym_lemma}, for any symmetric game we will often lose the player index from the notations and refer always to 
quantities corresponding to player 1,\footnote{Note that $g$ will denote $g(x,\delta)\defeq g_1(x,-\delta)$, as $\delta\geq 0$ for player 1, yet $g_1$ depends on the (negative) impulse of player 2.} henceforth addressed simply as `the player'. Player 2 shall be referred to as `the opponent'. Statements like `$V$ has the UIP' or `$V$ is a symmetric equilibrium payoff' are understood to refer to $(V(x),V(-x))$. Likewise, `$(\mathcal I,\delta)$ is admissible' or `$(\mathcal I,\delta)$ is a NE' refer to the pair $(\mathcal I,\delta(x)),(-\mathcal I,-\delta(-x))$.
\end{convention}

Due to their general lack of uniqueness, it is customary in game theory to restrict attention to specific type of NEs, depending on the problem at hand (see for instance \cite{HS} for a treatment within the classical theory). Motivated by Lemma \ref{sym_lemma} \textit{(iii)} and \textit{(iv)}, and by Remark \ref{r:sym_NE}, one can arguably state that symmetric NEs are the most meaningful for symmetric games. Furthermore, Lemma \ref{sym_lemma} implies that for symmetric games, one can considerably reduce the complexity of the full system of QVIs (\ref{QVIs}) provided the conjectured NE (or equivalently, the pair of payoffs) is symmetric. Using Convention \ref{convention}, Theorem \ref{verification} and Lemma \ref{sym_lemma} give:
\begin{corollary}[\textbf{Symmetric system of QVIs}]
\label{coro_sym_QVIs}
Given a symmetric game as in Definition \ref{sym_game}, let $V:\mathbb R\to\mathbb R$ be a function with the UIP, such that:
\begin{equation}
\label{sym_QVIs}
\begin{cases}
	\begin{aligned}
		& \mathcal H V- V=0 && \text{on} \,\,\, -\{\mathcal M V -  V = 0\}\backdefeq -\mathcal I^*\\
		& \max\big\{\mathcal A V -\rho V + f, \mathcal M V- V \}=0 && \text{on} \,\,\, -\{\mathcal M V - V < 0\}\backdefeq-\mathcal C^*
	\end{aligned}
\end{cases}
\end{equation} 
and $V\in C^2(-\mathcal C^*\backslash\partial\mathcal C^*)\cap C^1(-\mathcal C^*)\cap C(\mathbb R)$ has polynomial growth and bounded second derivative on some reduced neighbourhood of $\partial\mathcal C^*$. Suppose further that $(\mathcal I^*,\delta^*)$ is an admissible strategy.
$$\mbox{Then,}\quad V\mbox{ is a symmetric equilibrium payoff attained at a symmetric NE }(\mathcal I^*,\delta^*).$$
\end{corollary}
Note that system (\ref{sym_QVIs}) also omits the equation $\mathcal M V - V \leq 0$ as per being redundant. Indeed, by Definition \ref{sym_game} and Remark \ref{r:zero_impulse}, at a NE the player does not intervene above 0, nor the opponent below it. Thus, $\mathcal M V - V \leq \max\big\{\mathcal A V -\rho V + f, \mathcal M V- V \}=0$ on $-\mathcal C^*\supset (-\infty,0]$ and $\mathcal M V- V < 0$ on $[0,+\infty)$. 

System (\ref{sym_QVIs}) simplifies a numerical problem which is very challenging even in cases of linear structure (see Chapter \ref{c:2}). In light of the previous, we will focus our attention on symmetric NEs only and numerically solving the reduced system of QVIs (\ref{sym_QVIs}).

\section{Numerical discrete-space problem}
\label{s:discrete_problem}
In this section we consider a discrete version of the symmetric system of QVIs (\ref{sym_QVIs}) over a fixed grid, and propose and study an iterative method to solve it. As it is often done in numerical analysis for stochastic control, for the sake of generality we proceed first in an abstract fashion without making reference to any particular discretization scheme. Instead, we give some general assumptions any such scheme should satisfy for the results presented to hold. Explicit discretization schemes within our framework are presented in Section \ref{s:discretization} and used in Section \ref{s:numerics_games}.  

\subsection{Discrete system of quasi-variational inequalities}
\label{s:dQVIs}

From now on we work on a discrete symmetric grid 
\begin{equation*}
\grid:\ x_{-N}=-x_N<\dots<x_{-1}=-x_1<x_0=0<x_1<\dots<x_N.
\end{equation*} 
Recall that $\mathbb R^\grid$ denotes the set of functions $v:\grid\to\mathbb R$ and $S:\mathbb R^\grid\to \mathbb R^\grid$ denotes the symmetry operator, $Sv(x)=v(-x)$. In general, by an `operator' we simply mean some $F:\mathbb R^\grid\to\mathbb R^\grid$, not necessarily linear nor affine unless explicitly stated. We shall identify grid points with indexes, functions in $\mathbb R^\grid$ with vectors and linear operators with matrices; e.g., $S=(S_{ij})$ with $S_{ij}=1$ if $x_i=-x_j$ and 0 otherwise. The (partial) order considered in $\mathbb R^\grid$ and $\mathbb R^{\grid\times\grid}$ is the usual pointwise order for functions (elementwise for vectors and matrices), and the same is true for the supremum, maximum and arg-maximum induced by it.

We want to solve the following discrete nonlinear system of QVIs for $v\in\mathbb R^\grid$:
\begin{equation}
\label{dQVIs}
\begin{cases}
	\begin{aligned}
		& H v- v=0 && \text{on} \,\,\, -\{ M v -  v = 0\}\backdefeq -I^*\\
		& \max\big\{ Lv + f, Mv- v \}=0 && \text{on} \,\,\, -\{ M v - v < 0\}\backdefeq -C^*,
	\end{aligned}
\end{cases}
\end{equation} 
where $f\in\mathbb R^\grid$ and $L:\mathbb R^\grid\to \mathbb R^\grid$ is a linear operator. The nonlinear operators $M,H:\mathbb R^\grid\to\mathbb R^\grid$ are as follows: let $\emptyset\neq Z(x)\subseteq\mathbb R$ be a finite set for each $x\in\grid$, with $Z(x)=\{0\}$ if $x\geq 0$. Set $Z\defeq\prod_{x\in\mathbb\grid}Z(x)$ and for each $\delta\in Z$ let $B(\delta):\mathbb R^\grid\to\mathbb R^\grid$ be a linear operator, $c(\delta)\in\mathbb (0,+\infty)^\grid$ and $g(\delta)\in\mathbb R^\grid$, the three of them being \textit{row-decoupled} (i.e., row $x$ of $B(\delta),c(\delta),g(\delta)$ depends only on $\delta(x)\in Z(x)$). Then 
\begin{gather}
Mv \defeq \max_{\delta\in Z}\big\{B(\delta)v-c(\delta)\big\},\quad Hv=H(\delta^*)v\defeq SB(\delta^*)Sv+g(S\delta^*)\label{discrete_intervention_operators}\\
\mbox{and}\quad\delta^*=\delta^*(v)\defeq\max\Big(\argmax_{\delta\in Z}\big\{B(\delta)v-c(\delta)\big\}\Big).\label{delta_star}
\end{gather}

Some remarks are in order. Firstly, in the same fashion as the continuous-space case, the sets $I^*,C^*$ form a partition of the grid and represent the (discrete) intervention and continuation regions of the player, while $-I^*,-C^*$ are such regions for the opponent. 

Secondly, the general representation of $M$ follows \cite{CMS,AF}. For the standard choices of $B(\delta)$, our definition of $H$ is the only one for which a discrete version of Lemma \ref{sym_lemma} holds true (see Section \ref{s:discretization}). However, since $B$ and $g$ are row-decoupled, $SB(\delta^*)S$ and $g(S\delta^*)$ cannot be, as each row $x$ depends on $\delta^*(-x)$. For this reason and the lack of maximization over $-I^*$, there is no obvious way to reduce problem (\ref{dQVIs}) to a classical Bellman problem:
\begin{equation}
\label{Bellman_problem}
\sup_{\varphi}\big\{-A(\varphi)v+b(\varphi)\big\}=0,
\end{equation}
like in the impulse control case \cite{AF}, to apply Howard's policy iteration \cite[Ho-1]{BMZ}. Furthermore, unlike in the control case, even with unidirectional impulses and good properties for $L$ and $B(\delta)$, system (\ref{dQVIs}) may have no solution as in the analytical case \cite{ABCCV}. 

Thirdly, we have defined $\delta^*$ in (\ref{delta_star}) by choosing one particular maximizing impulse for each $x\in\grid$. The main motivation behind fixing one is to have a well defined discrete system of QVIs for every $v\in\mathbb R^\grid$. (This is not the case for the analytical problem (\ref{sym_QVIs}) where the gain operator $\mathcal H$ is not well defined unless $V$ has the UIP.) Being able to plug in any $v$ in (\ref{dQVIs}) and obtain a residual will be useful in practice, when assessing the convergence of the algorithm (see Section \ref{s:numerics_games}). Whether a numerical solution verifies, at least approximately, a discrete UIP (and the remaining technical conditions of the Verification Theorem) becomes something to be checked separately. 

\begin{remark}
\label{r:maxargmax1}
Choosing the maximum arg-maximum in (\ref{delta_star}) is partly motivated by ensuring a discrete solution will inherit the property of Remark \ref{r:concave_costs}. (The proof remains the same, for the discretizations of Section \ref{s:discretization}.) We will also motivate it in terms of the proposed numerical algorithm in Remark \ref{r:maxargmax2}. Note that in \cite{ABMZZ} (Chapter \ref{c:2}) the minimum arg-maximum was used instead for both players. Nevertheless, the replication of property \textit{(ii)}, Lemma \ref{sym_lemma}, dictates that it is only possible to be consistent with \cite{ABMZZ} for one of the two players (in this case, the opponent).
\end{remark}

\subsection{Iterative algorithm for symmetric games}
\label{s:sym_algo}
This section introduces the iterative algorithm developed to solve system (\ref{dQVIs}), which builds on Algorithm \ref{general_algo} by harnessing the symmetry of the problem and dispenses with the need for a relaxation scheme altogether. It is presented with a pseudocode that highlights the mimicking of system (\ref{dQVIs}) and the intuition behind the algorithm; namely:
\begin{itemize}
\item The player starts with some suboptimal strategy $\varphi^0=(I^0,\delta^0)$ and payoff $v^0$, to which the opponent responds symmetrically, resulting in a gain for the player (first equation of (\ref{dQVIs}); lines 1, 2 and 4 of Algorithm \ref{sym_algo}).
\item The player improves her strategy by choosing the optimal response, i.e., by solving a single-player impulse control problem through a policy-iteration-type algorithm (second equation of ({\ref{dQVIs}}); line 5 of Algorithm \ref{sym_algo}).
\item This procedure is iterated until reaching a stopping criteria (lines 6-8 of Algorithm \ref{sym_algo}).
\end{itemize}
\begin{notation}
In the following: $\grid_{<0}$ and $\grid_{\leq 0}$ represent the sets of grid points which are negative and nonpositive respectively, and $\Phi$ the set of \textit{(discrete) strategies} 
\begin{equation}
\label{discrete_strategies}
\Phi\defeq\{\varphi=(I,\delta):\ I\subseteq\grid_{<0}\mbox{ and }\delta\in Z\}.
\end{equation}
Set complements are taken with respect to the whole grid, $Id:\mathbb R^G\to\mathbb R^G$ is the identity operator; and given a linear operator $O:\mathbb R^\grid\to\mathbb R^\grid\simeq\mathbb R^{\grid\times\grid}$, $v\in\mathbb R^\grid$ and subsets $I,J\subseteq\grid$, $v_I\in\mathbb R^I$ denotes the restriction of $v$ to $I$ and $O_{IJ}\in\mathbb R^{I\times J}$ the submatrix/operator with rows in $I$ and columns in $J$. 
\end{notation} 

\begin{algorithm}[H]
    \caption{Iterative algorithm for symmetric games} \label{sym_algo}
Set $tol,scale>0$. 
		\begin{algorithmic}[1]
	    \STATE Choose initial guess: $v^0\in\mathbb R^\grid$
			\STATE Set $I^0=\{Lv^0+f\leq Mv^0-v^0\}\cap\grid_{<0}$ and $\delta^0=\delta^*(v^0)$
			\FOR{$k=0,1,\dots$}
				\STATE $
									v^{k+1/2}=\begin{cases} 
														H(\delta^k)v^k & \mbox{ on }-I^k\\
													  v^k            & \mbox{ on }(-I^k)^c
														\end{cases}
								$						
				\STATE $(v^{k+1},I^{k+1},\delta^{k+1})=\textsc{SolveImpulseControl}(v^{k+1/2},(-I^k)^c)$
				\IF{$\|(v^{k+1}-v^k)/\max\{|v^{k+1}|,scale\}\|_{\infty}<tol$} 
					\STATE break from the iteration
				\ENDIF	
			\ENDFOR	
    \end{algorithmic}
\end{algorithm}

The $scale$ parameter in line 5 of Algorithm \ref{sym_algo}, used throughout the literature by Forsyth, Labahn and coauthors \cite{AF,FL,HFL1,HFL2}, prevents the enforcement of unrealistic levels of accuracy for points $x$ where $v^{k+1}(x)\approx 0$. Additionally, note that having chosen the initial guess for the payoff $v^0$, the initial guess for the strategy is taken as the one induced by $v^0$. (The alternative expression for the intervention region gives the same as $\big\{M v^0- v^0=0\big\}$ for a solution of (\ref{dQVIs}).)

Line 5 of Algorithm \ref{sym_algo} assumes we have a subroutine $\textsc{SolveImpulseControl}(w,D)$ that solves the constrained QVI problem:
\begin{equation}
\label{constrained_QVI}
\max\{Lv+f,Mv-v\}=0\mbox{ on }D,\quad\mbox{subject to }v=w\mbox{ on }D^c
\end{equation}
for fixed $\grid_{\leq 0}\subseteq D\subseteq\grid$ (approximate continuation region of the opponent) and $w\in\mathbb R^\grid$ (approximate payoff due to the opponent's intervention). Although we only need to solve for $\tilde v=v_D$, the value of $v_{D^c}=w_{D^c}$ impacts the solution both when restricting the equations and when applying the nonlocal operator $M$. Hence, the approximate payoff $v^{k+1/2}$ fed to the subroutine serves to pass on the gain that resulted from the opponent's intervention and as an initial guess if desired (more on this on Remark \ref{r:initial_guess}).

The remaining of this section consists in establishing an equivalence between problem (\ref{constrained_QVI}) and a classical (unconstrained) QVI problem of impulse control. This allows us to prove the existence and uniqueness of its solution. In particular, we will see that \textsc{SolveImpulseControl} can be defined, if wanted, by policy iteration. However, we will present in the next section an alternative method that performs better in many practical situations and, in particular, in the examples treated in Section \ref{s:numerics_games}. Let us suppose from here onwards that the following assumptions hold true (see Appendix \ref{appendix:matrices} for the relevant Definitions \ref{matrix_definitions}):
\begin{enumerate}[label=(A\arabic*), start=0]
\item \label{A0} 
For each strategy $\varphi=(I,\delta)\in\Phi$ and $x\in I$, there exists a walk in graph$B(\delta)$ from row $x$ to some row $y\in I^c$.
\item \label{A1} 
$-L$ is a strictly diagonally dominant (SDD) $\mbox{L}_0$-matrix and, for each $\delta\in Z$, $Id-B(\delta)$ is a weakly diagonally dominant (WDD) $\mbox{L}_0$-matrix.
\end{enumerate}

\begin{remark}[\textit{Interpretation}]
\label{r:interpretation}
Assumptions \ref{A0},\ref{A1} are (H2),(H3) in \cite{AF}. For an impulse operator (say, `$B(\delta)v(x)=v(x+\delta)$'), \ref{A0} asserts that the player always wants to shift states in her intervention region to her continuation region through finitely many impulses. (This does not take into account the opponent's response.) On the other hand, if problem (\ref{constrained_QVI}) was rewritten as a fixed point problem, \ref{A1} would essentially mean that the uncontrolled operator is contractive while the controlled ones are nonexpansive (see \cite{CMS} and \cite[Sect.4]{AF}). 
\end{remark}

\begin{theorem}
\label{solution_constrained_QVI_1}
Assume \emph{\ref{A0}},\emph{\ref{A1}}. Then, for every $\grid_{\leq 0}\subseteq D\subseteq\grid$ and $w\in\mathbb R^\grid$, there exists a unique $v^*\in\mathbb R^\grid$ that solves the constrained QVI problem (\ref{constrained_QVI}). Further, $v^*_D$ is the unique solution of 
\begin{equation}
\label{restricted_QVI}
\max\Big\{\tilde L \tilde v +\tilde f, \tilde M \tilde v - \tilde v\Big\}=0,
\end{equation}
where $\tilde L\defeq L_{DD},\ \tilde f\defeq f_D+L_{DD^c}w_{D^c},\ \tilde Z\defeq \prod_{x\in D}Z(x),\ \tilde c(\tilde\delta)\defeq c(\tilde\delta)-B(\tilde\delta)_{DD^c}w_{D^c}\mbox{ and }\tilde B(\tilde\delta)\defeq B(\tilde\delta)_{DD}\mbox{ for }\tilde\delta\in\tilde Z;\mbox{ and } \tilde M\tilde v\defeq\max_{\tilde\delta\in\tilde Z}\big\{\tilde B(\tilde\delta)\tilde v -\tilde c(\tilde\delta)\big\}\mbox{ for }\tilde v\in\mathbb R^D$.

Additionally, for any initial guess, the sequence $(\tilde v^k)\subseteq\mathbb R^D$ defined by policy iteration \cite[Thm.4.3]{AF} applied to problem (\ref{restricted_QVI}), converges exactly to $v^*_D$ in at most $|\tilde\Phi|$ iterations, with $\tilde\Phi\defeq\{\tilde\varphi=(I,\tilde\delta):\ I\subseteq\grid_{<0}\mbox{ and }\tilde\delta\in\tilde Z\}$  the set of restricted admissible strategies.\footnote{$|A|$ denotes the cardinal of set $A$.}
\end{theorem}

\begin{proof}
The equivalence between problems (\ref{constrained_QVI}) and (\ref{restricted_QVI}) is due to simple algebraic manipulation and $B(\delta),c(\delta)$ being row-decoupled for every $\delta\in Z$. ($B(\tilde\delta),c(\tilde\delta)$ are thus defined in the obvious way for each $\tilde\delta\in\tilde Z$.) 

The rest of the proof is mostly as in \cite[Thm.4.3]{AF}. Let $\tilde{Id}=Id_{DD}$. 
Each intervention region $I$, can be identified with its indicator $\tilde\psi=\mathbbm 1_I\in\{0,1\}^D$ since $D\supseteq I$, and each $\tilde\psi$ can be identified in turn with a diagonal matrix having $\tilde\psi$ as main diagonal: $\tilde\Psi=\mbox{diag}(\tilde\psi)\in\mathbb R^{D\times D}$. Then problem (\ref{restricted_QVI}) takes the form of the classical Bellman problem 
\begin{equation}
\label{equivalent_Bellman}
\max_{\tilde\varphi\in\tilde\Phi}\big\{-A(\tilde\varphi)v+b(\tilde\varphi)\big\}=0,
\end{equation}
if we take 
$$A(\tilde\varphi)=-(\tilde{Id}-\tilde\Psi)\tilde L + \tilde\Psi(\tilde{Id}-\tilde B(\tilde\delta))\quad\mbox{ and }\quad b(\tilde\varphi)=(\tilde{Id}-\tilde\Psi)\tilde f - \tilde\Psi\tilde c.$$ 
Note that $\tilde\Phi$ can be identified with the Cartesian product 
$$\tilde\Phi=\Big(\prod_{x\in\grid_{<0}}\{0,1\}\times Z(x)\Big)\times\Big(\prod_{x\in D\backslash\grid_{<0}}\{0\}\times Z(x)\Big)$$
and $A(\tilde\varphi),b(\tilde\varphi)$ are row-decoupled for every $\tilde\varphi\in\tilde\Phi$. Since $\tilde\Phi$ is finite, all we need to show is that the matrices $A(\tilde\varphi)$ are monotone (Definitions \ref{matrix_definitions} and \cite[Thm.2.1]{BMZ}). Let us check the stronger property (Thm. \ref{characterization_theorem} and Prop. \ref{M-matrix_characterization}) of being weakly chained diagonally dominant (WCDD) $\mbox{L}_0$-matrices (see Definitions \ref{matrix_definitions}).

If \ref{A0} and \ref{A1} also held true for the restricted matrices and strategies, the conclusion would follow. While \ref{A1} is clearly inherited, \ref{A0} may fail to do so, but only in non-problematic cases. To see this, let $\tilde\varphi=(I,\tilde\delta)\in\tilde\Phi,\ x\in I\subseteq D$ and let $\delta\in Z$ be some extension of $\tilde\delta$. Note that row $x$ of $A(\tilde\varphi)$ is WDD. We want to show that there is a walk in graph$A(\tilde\varphi)$ from $x$ to an SDD row.

By \ref{A0} there must exist some walk $x\to y_1\to\dots\to y_n\in I^c$ in graph$B(\delta)$. If this is in fact a walk from $x$ to $I^c\cap D$ in graph$\tilde B(\tilde\delta)$, then it verifies the desired property (just as in \cite[Thm.4.3]{AF}). If not, then there must be a first $1\leq m\leq n$ such that the subwalk $x\to y_1\to\dots\to y_m$ is in graph$\tilde B(\tilde\delta)$ but $y_{m+1}\notin D$. Since $y_m\to y_{m+1}$ is an edge in graph$B(\delta)$, we have $B(\delta)_{y_m,y_{m+1}}\neq 0$ and the WDD row (by \ref{A1}) $y_m$ of $\tilde{Id}-\tilde B(\tilde\delta)$ is in fact SDD. Meaning that the subwalk $x\to y_1\to\dots\to y_m$ verifies the desired property instead.
\end{proof}

\begin{remark}(\textit{Practical considerations}) 
\label{practical_considerations}
\begin{enumerate*}
\item While convergence is guaranteed to be exact, floating point arithmetic can bring about stagnation \cite{HFL2}. A stopping criteria like that of Algorithm \ref{sym_algo} should be used in those cases, with a tolerance $\ll tol$.   
\item \label{lambda} The solution of system (\ref{restricted_QVI}) does not change if one introduces a scaling factor $\lambda>0$: $\max\big\{\tilde L \tilde v +\tilde f, \lambda\big(\tilde M \tilde v - \tilde v\big)\big\}=0$ \cite[Lem.4.1]{AF}. This problem-specific parameter is typically added in the implementation to enhance performance \cite{HFL1,AF}. It can intuitively be thought as a units adjustment.
\end{enumerate*}
\end{remark}

\subsection{Iterative subroutine for impulse control}
\label{s:solve_one_player}
Due to Theorem \ref{solution_constrained_QVI_1}, a sensible choice for \textsc{SolveImpulseControl} is the classical policy iteration algorithm \cite[Thm.4.3]{AF} applied to (\ref{restricted_QVI}) (i.e., \cite[Ho-1]{BMZ} applied to (\ref{equivalent_Bellman})), adding an appropriately chosen scaling factor $\lambda$ to improve efficiency (Remark \ref{practical_considerations} \ref{lambda}). It does, however, bear some drawbacks in practice. At each iteration, one needs to solve the system $-A(\tilde\varphi^k)v^{k+1}+b(\tilde\varphi^k)=0$ for some $\tilde\varphi^k\in\tilde\Phi$. While the matrix $\tilde L$ typically has a good sparsity pattern in applications (often tridiagonal), the presence of $\tilde B(\tilde\delta^k)$ prevents $A(\tilde\varphi^k)$ from inheriting the same structure and makes the resolution of the previous system a lot more costly, even when using a good ordering technique. An exact resolution often becomes prohibitive and an iterative method, such as GMRES or BiCGSTAB with preconditioning, is relied upon.

Motivated by the previous observation, this section considers an alternative choice for \textsc{SolveImpulseControl}: an instance of a very general class of algorithms known as \textit{fixed-point policy iteration} \cite{HFL1,Cl0}. As far as the author knows, this application to impulse control was never done in the past other than heuristically in Chapter \ref{c:2}. Instead of solving $-A(\tilde\varphi^k)v^{k+1}+b(\tilde\varphi^k)=0$ at the $k$-th iteration, we will solve 
\begin{equation}
\label{VI}
\underbrace{\left((\tilde{Id}-\tilde\Psi^k)\tilde L - \tilde\Psi^k\right)}_{-\tilde{\mathbb A}(\tilde\varphi^k)}v^{k+1} + \underbrace{\tilde\Psi^k \tilde B(\tilde\delta^k)}_{\tilde{\mathbb B}(\tilde\varphi^k)}v^k + \underbrace{b(\tilde\varphi^k)}_{\tilde{\mathbb C}(\tilde\varphi^k)}=0,
\end{equation}
(scaled by $\lambda$) where the previous iterate value $v^k$ is given and $\tilde\Psi^k$ is the diagonal matrix with $\psi^k$ as diagonal. In other words, we split the original policy matrix $A(\tilde\varphi)=\tilde{\mathbb A}(\tilde\varphi)-\tilde{\mathbb B}(\tilde\varphi)$ and we apply a one-step fixed-point approximation,
\begin{equation}
\label{FPPI}
\tilde{\mathbb A}\big(\tilde\varphi^k\big)\tilde v^{k+1}=\tilde{\mathbb B}\big(\tilde\varphi^k\big)\tilde v^k +\tilde{\mathbb C}\big(\tilde\varphi^k\big),
\end{equation}
at each iteration of Howard's algorithm. The resulting method can be expressed as follows ($tol$ and $scale$ are as in Algorithm \ref{sym_algo}):

\makeatletter\renewcommand{\ALG@name}{Subroutine}
\begin{algorithm}[H]
    \caption{$(v,I,\delta)=\textsc{SolveImpulseControl}(w,D)$} \label{solve_one_player}
		\textbf{Inputs} $w\in\mathbb R^\grid$ and solvency region $\grid_{\leq 0}\subseteq D\subseteq\grid$
		
		\textbf{Outputs} $v\in\mathbb R^\grid,\ I\subseteq\grid_{<0},\ \delta\in Z$
		\newline
		
		Set scaling factor $\lambda>0$ and $0<\widetilde{tol}\ll tol$. 
		
		\qquad // Restrict constrained problem
		\begin{algorithmic}[1]
			\STATE $\tilde L\defeq L_{DD},\ \tilde f\defeq f_D+L_{DD^c}w_{D^c}$
			\STATE $\tilde Z\defeq \prod_{x\in D}Z(x),\quad \tilde c(\tilde\delta)\defeq c(\tilde\delta)-B(\tilde\delta)_{DD^c}w_{D^c},\ \tilde B(\tilde\delta)\defeq B(\tilde\delta)_{DD}\mbox{ for }\tilde\delta\in\tilde Z$
			\STATE $\tilde M\tilde v\defeq\max_{\tilde\delta\in\tilde Z}\big\{\tilde B(\tilde\delta)\tilde v -\tilde c(\tilde\delta)\big\}\mbox{ for }\tilde v\in\mathbb R^D,\ \tilde{Id}\defeq Id_{DD}$
			\newline
					
// Solve restricted problem
			\STATE Choose initial guess: $\tilde v^0\in\mathbb R^D$, $I^0\subseteq\grid_{<0}$
	  	\FOR{$k=0,1,2,\dots$}
				\STATE $\tilde L^k_{ij}=\begin{cases} 
										\tilde L_{ij} & \mbox{ if }x_i\in D\backslash I^k\\
										-\tilde{Id}_{ij} & \mbox{ if }x_i\in I^k	
										\end{cases}
										\qquad
								\tilde f^k =\begin{cases} 
										\tilde f & \mbox{ on }\in D\backslash I^k\\
										\tilde M\tilde v^k & \mbox{ on }I^k
										\end{cases}
							 $	
				\STATE $\tilde v^{k+1} \mbox{ solution of }\tilde L^k \tilde v + \tilde f^k=0$
				\STATE $I^{k+1}=\big\{\tilde L\tilde v^{k+1}+\tilde f\leq\lambda\big(\tilde M\tilde v^{k+1}-\tilde v^{k+1}\big)\big\}$
				\IF{$\|(\tilde v^{k+1}-\tilde v^k)/\max\{\tilde v^{k+1},scale\}\|_{\infty}<\widetilde{tol}$} 
				\STATE $v=\begin{cases}
									\tilde v^{k+1} & \mbox{ on }D\\
									w_{D^c} & \mbox{ on }D^c
									\end{cases},
				\ I=I^{k+1},\ \delta=\delta^*(v)$ and break from the iteration 
				\ENDIF	
			\ENDFOR
    \end{algorithmic}
\end{algorithm}
\makeatother

Lines 1-3 of Subroutine \ref{solve_one_player} deal with restricting the constrained problem, while the rest give a routine that can be applied to any QVI of the form (\ref{restricted_QVI}). Starting from some suboptimal $\tilde v^0$ and $I^0$, one computes a new payoff $\tilde v^1$ by solving the coupled equations $\tilde M\tilde v^0-\tilde v^1=0$ on $I^0$ and $\tilde L\tilde v^1+\tilde f=0$ outside $I^0$. A new intervention region $I^1=\big\{\tilde L\tilde v^1+\tilde f\leq\lambda\big(\tilde M\tilde v^1-\tilde v^1\big)\big\}$ is defined and the procedure is iterated.

Algorithmically, the difference with classical policy iteration is that $\tilde v^{k+1}$ is computed in Line 7 with a fixed obstacle $\tilde Mv^k$, changing a quasivariational inequality for a variational one. The resulting method is intuitive and simple to implement, and the linear system (\ref{VI}) (Line 7) inherits the sparsity pattern of $\tilde L$. For example, for an SDD tridiagonal $\tilde L$, the system can be solved (exactly in exact arithmetic, and stably in floating point one) in $O(n)$ operations, with $n=|D|$ \cite[Sect.9.5]{H}. The matrix-vector multiply $\tilde B(\tilde\delta^k)\tilde v^k$ can take at most $O(n^2)$ operations, but will reduce to $O(n)$ for standard discretizations of impulse operators.

It is also worth mentioning that Subroutine \ref{solve_one_player} differs from the so-called iterated optimal stopping \cite{COS,OS} in that the latter solves $\max\big\{\tilde L \tilde v^{k+1} +\tilde f, \tilde M \tilde v^k - \tilde v^{k+1}\big\}=0$ exactly at the $k$-th iteration (by running a full subroutine of Howard's algorithm with fixed obstacle), while the former only performs one approximation step. 

To establish the convergence of Subroutine \ref{solve_one_player} in the present framework, we add the following assumption:
\begin{enumerate}[label=(A\arabic*), start=2]
\item \label{A2} $B(\delta)$ has nonnegative diagonal elements for all $\delta\in Z$.
\end{enumerate}

\begin{remark}
\ref{A2} and the requirement of \ref{A1} that $Id-B(\delta)$ be a WDD $\mbox{L}_0$-matrix are equivalent to $B(\delta)$ being substochastic (see Appendix \ref{appendix:matrices}). This is standard for impulse operators (see Section \ref{s:discretization}) and other applications of fixed-point policy iteration \cite[Sect.4-5]{HFL1}.
\end{remark}

\begin{theorem}
\label{solution_constrained_QVI_2}
Assume \emph{\ref{A0}--\ref{A2}} and set $I^0=\emptyset$. Then, for every $\grid_{\leq 0}\subseteq D\subseteq\grid$ and $w\in\mathbb R^\grid$, the sequence $(\tilde v^k)$ defined by \textsc{SolveImpulseControl}$(w,D)$ is monotone increasing for $k\geq 1$ and converges to the unique solution of (\ref{restricted_QVI}).
\end{theorem}

\begin{proof}
We can assume without loss of generality that $\lambda=1$. Subroutine \ref{solve_one_player} takes the form of a fixed-point policy iteration algorithm as per (\ref{FPPI}). Assumptions \ref{A0},\ref{A1} ensure (\ref{restricted_QVI}) has a unique solution (Theorem \ref{solve_one_player}) and that this scheme satisfies \cite[Cond.3.1 (i),(ii)]{HFL1}. That is, $\tilde{\mathbb A}(\tilde\varphi)$ and $\tilde{\mathbb A}(\tilde\varphi)-\tilde{\mathbb B}(\tilde\varphi)$ are nonsingular $M$-matrices (see proof of Theorem \ref{solve_one_player} and Appendix \ref{appendix:matrices}) and all coefficients are bounded since $\tilde\Phi$ is finite. In \cite[Thm.3.4]{HFL1} convergence is proved under one additional assumption of $\|\cdot\|_{\infty}$-contractiveness \cite[Cond.3.1 (iii)]{HFL1}, which is not verified in our case. However, the same computations show that the scheme satisfies
\begin{equation}
\tilde{\mathbb A}(\tilde\varphi^k)(\tilde v^{k+1}-\tilde v^k)\geq \tilde{\mathbb B}(\tilde\varphi^{k-1})(\tilde v^k-\tilde v^{k-1})\mbox{ for all }k\geq 1.
\end{equation}
Since $I^0=\emptyset$, and due to \ref{A1} and \ref{A2}, $\tilde{\mathbb B}(\tilde\varphi^0)=0$ and $\tilde{\mathbb B}(\tilde\varphi^k)\geq 0$ for all $k$. Thus, $(\tilde v^k)_{k\geq 1}$ is increasing by monotonicity of $\tilde{\mathbb A}(\tilde\varphi^k)$. Furthermore, it must be bounded, since for all $k\geq 1$:
$$
\tilde{\mathbb A}\big(\tilde\varphi^k\big)\tilde v^{k+1}=\tilde{\mathbb B}\big(\tilde\varphi^k\big)\tilde v^k +\tilde{\mathbb C}\big(\tilde\varphi^k\big)\leq\tilde{\mathbb B}\big(\tilde\varphi^{k}\big)\tilde v^{k+1} +\tilde{\mathbb C}\big(\tilde\varphi^k\big),
$$
which gives $\tilde v^{k+1}\leq (\tilde{\mathbb A}(\tilde\varphi^k)-\tilde{\mathbb B}(\tilde\varphi^k))^{-1}\tilde{\mathbb C}\big(\tilde\varphi^k\big)\leq \max_{\tilde\varphi\in\tilde\Phi}(\tilde{\mathbb A}(\tilde\varphi)-\tilde{\mathbb B}(\tilde\varphi))^{-1}\tilde{\mathbb C}\big(\tilde\varphi\big)$. Hence, $(\tilde v^k)$ converges.
\end{proof}

\begin{remark}
\label{r:initial_guess}
In light of Theorem \ref{solution_constrained_QVI_2}, moving forward we will set $I^0=\emptyset$ in Subroutine \ref{solve_one_player}. It is natural however to choose $\tilde v^0=w_D$ and $I^0=\big\{\tilde L\tilde v^0+\tilde f\leq\lambda\big(\tilde M\tilde v^0-\tilde v^0\big)\big\}$. The experiments performed with the latter choice displayed (non-monotone) convergence and usually a faster one; but this is not proved here. Additionally, exact convergence was often observed.
\end{remark}

\subsection{Overall routine as a fixed-point policy-iteration-type method}
\label{s:FPPI}

The system of QVIs (\ref{dQVIs}) cannot be reduced in any apparent way to a Bellman formulation (\ref{Bellman_problem}) (see comments preceding equation). Notwithstanding, we shall see that Algorithm \ref{sym_algo} does take a very similar form to a fixed-point policy iteration algorithm as in (\ref{FPPI}) for some appropriate $\mathbb A,\mathbb B,\mathbb C$. Further, assumptions resembling those of the classical case \cite{HFL1} will be either satisfied or imposed to study its convergence. This is independent of whether \textsc{SolveImpulseControl} is chosen as in Subroutine \ref{solve_one_player} or Howard's algorithm (Theorem \ref{solution_constrained_QVI_1}), although we shall assume that the outputs $I^k,\delta^k$ in the latter case are defined in the same way as in the former. The matrix and graph-theoretic definitions and properties used throughout this section can be found in Appendix \ref{appendix:matrices}.

\begin{notation}
We identify 
each intervention region $I\subseteq\grid_{<0}$ with its indicator function $\psi=\mathbbm 1_I\in\{0,1\}^\grid$ and each $\psi$ with a diagonal matrix having $\psi$ as main diagonal: $\Psi=\mbox{diag}(\psi)\in\mathbb R^{\grid\times \grid}$.
The sequences $(v^k)$ and $(\varphi^k)$, with $\varphi^k=(\psi^k,\delta^k)$, are the ones generated by Algorithm \ref{sym_algo}. We consider $v^*\in\mathbb R^\grid$ fixed and $\varphi^*=(\psi^*,\delta^*(v^*))$ the induced strategy with $\psi^*\defeq\{Lv^*+f\leq Mv^*-v^*\}\cap\grid_{<0}$.
\end{notation}

\begin{proposition}
\label{FPPI-like_result}
Assume \emph{\ref{A0}--\ref{A2}}. Then, 
\begin{equation}
\label{FPPI-like}
\mathbb A\big(\varphi^k,\varphi^{k+1}\big)v^{k+1}=\mathbb B\big(\varphi^k\big)v^k +\mathbb C\big(\varphi^k,\varphi^{k+1}\big),\mbox{ where:}
\end{equation}
\begin{enumerate}[label=(\roman*)]
\item $\psi^k=\mathbbm 1_{\{Lv^k+f\leq Mv^k-v^k\}\cap\grid_{<0}}\mbox{ and } \delta^k\in\argmax_{\delta\in Z}\big\{B(\delta)v^k-c(\delta)\big\}$.   
\item $\mathbb A\big(\varphi,\overline\varphi\big)\defeq Id -\big(Id - \overline\Psi-S\Psi S\big)(Id+L) - \overline\Psi B(\overline\delta)$ is a WCDD $\mbox{L}_0$-matrix, and thus a nonsingular M-matrix.  
\item $\mathbb B\big(\varphi\big)\defeq S\Psi B(\delta)S=diag(S\psi) S B(\delta)S$ is substochastic.
\item $\mathbb C\big(\varphi,\overline\varphi\big)\defeq \big(Id-\overline\Psi-S\Psi S\big)f - \overline\Psi c(\overline\delta) + S\Psi S g(S\delta)$.
\end{enumerate}
\end{proposition}

\begin{proof}
Using that $(v^{k+1},I^{k+1},\delta^{k+1})=\textsc{SolveImpulseControl}(v^{k+1/2},(-I^k)^c)$ solves the constrained QVI problem (\ref{constrained_QVI}) for $D=(-I^k)^c$ and $w=v^{k+1/2}$ (Theorem \ref{solution_constrained_QVI_1} or \ref{solution_constrained_QVI_2}), the recurrence relation (\ref{FPPI-like}) results from simple algebraic manipulation.

Given $\varphi,\overline\varphi\in\Phi$, \ref{A0} and \ref{A1} ensure $\mathbb A\big(\varphi,\overline\varphi\big)$ is a WCDD $\mbox{L}_0$-matrix, while \ref{A1} and \ref{A2} imply $\mathbb B\big(\varphi\big)$ is substochastic.
\end{proof}

The following corollary is immediate by induction. It gives a representation of the sequence of payoffs in terms of the improving strategies throughout the algorithm.

\begin{corollary}
Assume \emph{\ref{A0}--\ref{A2}}. Then, 
\begin{equation}
\label{sequence_representation}
v^{k+1}=\left(\prod_{j=k}^0\mathbb A^{-1}\mathbb B\big(\varphi^j,\varphi^{j+1}\big)\right)v^0 + \sum_{n=0}^k\left(\prod_{j=k}^{n+1}\mathbb A^{-1}\mathbb B\big(\varphi^j,\varphi^{j+1}\big)\right)\mathbb A^{-1}\mathbb C\big(\varphi^n,\varphi^{n+1}\big).\footnote{For any index $i\leq k$, $\prod_{j=k}^iA^j\defeq A^kA^{k-1}\dots A^i$.}
\end{equation}
\end{corollary}

We now establish some properties of the strategy-dependent matrix coefficients that will be useful in the sequel. Given a WDD (resp. substochastic) matrix $A\in \mathbb R^{\grid\times\grid}$, we define its set of `non-trouble states' (or rows) as 
$$
J[A]\defeq \{x\in\grid:\mbox{ row }x\mbox{ of }A\mbox{ is SDD}\}\mbox{ (resp. }\hat J[A]\defeq \{x\in\grid:\mbox{ row }x\mbox{ of }A\mbox{ sums less than one}\}),
$$ 
and its \textit{index of connectivity} $\mbox{con}A$ (resp. \textit{index of contraction} $\widehat{\mbox{con}}A$) by computing for each state the least length that needs to be walked on graph$A$ to reach a non-trouble one, and then taking the maximum over all states (more details in Appendix \ref{appendix:matrices}). This recently introduced concept gives an equivalent charaterization of the WCDD property for a WDD matrix as con$A<+\infty$, and can be efficiently checked for sparse matrices in $O(|\grid|)$ operations \cite{A}. On the other hand, if $A$ is substochastic then $\widehat{\mbox{con}}A<\infty$ if and only if its spectral radius verifies $\rho(A)<1$ (Theorem \ref{contraction}). The proof of the following lemma can be found in Appendix \ref{appendix:matrices}.

\begin{lemma}
\label{coefficients_properties}
Assume \emph{\ref{A0}--\ref{A2}}. Then for all $\varphi,\overline\varphi\in\Phi$, $\mathbb A^{-1}\mathbb B\big(\varphi,\overline\varphi\big)$ is substochastic, $(\mathbb A-\mathbb B)\big(\varphi,\overline\varphi\big)$ is a WDD $L_0$-matrix and $\widehat{\emph{con}}\big[\mathbb A^{-1}\mathbb B\big(\varphi,\overline\varphi\big)\big]\leq\emph{con}\big[(\mathbb A-\mathbb B)\big(\varphi,\overline\varphi\big)\big]$.
\end{lemma}

As previously mentioned, system (\ref{dQVIs}) may have no solution. The matrix coefficients introduced in this section allow us to algebraically characterize the existence of such solutions through strategy-dependent linear systems of equations. For each strategy $\varphi\in\Phi$, let $\mathbb O\big(\varphi\big):\mathbb R^\grid\to\mathbb R^\grid$ be the operator that applies $Id+L,\ M$ and $H$ on the continuation, intervention and opponent's intervention regions, respectively. That is, $\mathbb O\big(\varphi\big)= Id-(\mathbb A-\mathbb B)(\varphi^*,\varphi^*)=\big(Id - \Psi-S\Psi S\big)(Id+L) + \Psi B(\delta) + S \Psi B(\delta) S$. Then the following equivalences are immediate.

\begin{proposition}
\label{equivalent_dQVIs1}
Assume \emph{\ref{A0}--\ref{A2}}. Then the following statements are equivalent:
\begin{enumerate}[label=(\roman*)]
\item $v^*$ solves the system of QVIs (\ref{dQVIs}).
\item $\mathbb A\big(\varphi^*,\varphi^*\big)v^*=\mathbb B\big(\varphi^*\big)v^* +\mathbb C\big(\varphi^*,\varphi^*\big)$.
\item $v^*=\mathbb O\big(\varphi^*\big)v^*+\mathbb C\big(\varphi^*,\varphi^*\big)$.
\end{enumerate}
\end{proposition}

As mentioned in Remark \ref{r:interpretation}, Assumption \ref{A0} constrains the type of strategies the player can use, but without taking into account the opponent's response. This is enough for the single-player constrained problems to have a solution and, therefore, for Algorithm \ref{sym_algo} to be well defined. But we cannot expect this restriction to be sufficient in the study of the two-player game and the convergence of the overall routine. 

In order to improve the result of Proposition \ref{equivalent_dQVIs1} let us consider the following stronger version of \ref{A0} reflecting the interaction between the player and the opponent.

\begin{enumerate}[label=(A\arabic*'),start=0]
\item \label{A0'} 
For each pair of strategies $\varphi,\overline\varphi\in\Phi$, and for each $x\in \overline I\cup (-I)$, there exists a walk in graph$(\overline\Psi B(\overline\delta)+S\Psi B(\delta)S)$ from row $x$ to some row $y\in \overline C\cap  C$, where $\overline C={\overline I}^c,\ C=I^c$.
\end{enumerate} 

\begin{remark}(\textit{Interpretation})
\label{r:interpretation2}
If $\overline\varphi,-\varphi$ are the strategies used by the player and the opponent respectively,\footnote{The slight abuse of notation $-\varphi$ stands for the strategy symmetric to $\varphi$, i.e., $-\varphi=(-I,-\delta(-x))$.} then \ref{A0'} asserts that states in their intervention regions will eventually be shifted to the common continuation region. This precludes infinite simultaneous interventions and emulates the admissibility condition of the continuous-state case. Fixing $I=\emptyset$ we recover \ref{A0}. Additionally, \ref{A0'} together with \ref{A1} imply that $(\mathbb A-\mathbb B)(\varphi,\overline\varphi)$ is a WCDD $\mbox{L}_0$-matrix, hence an $M$-matrix. This is another one of the assumptions of the classical fixed-point policy iteration \cite{HFL1}.  
\end{remark}

Under this new assumption, the $\varphi^*=\varphi^*(v^*)$-dependent systems of Proposition \ref{equivalent_dQVIs1} will admit a unique solution. Then solving the original problem (\ref{dQVIs}) amounts to finding $v^*\in\mathbb R^\grid$ that solves its induced linear system of equations.
\begin{proposition}
\label{equivalent_dQVIs2}
Assume \emph{\ref{A0'},\ref{A1},\ref{A2}}. In the context of Proposition \ref{equivalent_dQVIs1}, the following statements are also equivalent:
\begin{enumerate}[label=(\roman*),start=4]
\item $v^*=(\mathbb A-\mathbb B)^{-1}\mathbb C\big(\varphi^*,\varphi^*\big)$.
\item $v^*= (Id-\mathbb A^{-1}\mathbb B)^{-1}\mathbb A^{-1}\mathbb C\big(\varphi^*,\varphi^*\big)=\sum_{n\geq 0}\big(\mathbb A^{-1}\mathbb B\big)^n\mathbb A^{-1}\mathbb C\big(\varphi^*,\varphi^*\big)$. (cf. equation (\ref{sequence_representation}).)
\end{enumerate}
\end{proposition}

\begin{proof}
Both expressions result from rewriting and solving the systems of Proposition \ref{equivalent_dQVIs1}. Assumptions \ref{A0'},\ref{A1} guarantee that $(\mathbb A-\mathbb B)\big(\varphi^*,\varphi^*\big)$ is WCDD and, hence, nonsingular. Then {\textit{(v)}} is due to Lemma \ref{coefficients_properties}, Theorem \ref{contraction} and the matrix power series expansion $(1-X)^{-1}=\sum_{n\geq 0}X^n$, when $\rho(X)<1$.  
\end{proof}

\subsection{Convergence analysis}
\label{s:convergence}

We now study the convergence properties of Algorithm \ref{sym_algo}. Henceforth,  the UIP refers to the obvious discrete analogue of Definition \ref{UIP}, where we replace the domain $\mathbb R$, the impulse constraints $\mathcal Z$ and the operator $\mathcal M$ by their discretizations $\grid,\ Z$ and $M$ respectively.

The obvious first question to address is whether when Algorithm \ref{sym_algo} converges, it does so to a solution of the system of QVIs (\ref{dQVIs}). Unlike in the classical Bellman problem (\ref{Bellman_problem}), problem (\ref{dQVIs}) is intrinsically dependent on the particular strategy chosen by the player (see Propositions \ref{equivalent_dQVIs1} and \ref{equivalent_dQVIs2}). Accordingly, we start with a lemma addressing what can be said about the convergence of the strategies $(\varphi^k)$ when the payoffs $(v^k)$ converge.

\begin{notation}
$\partial I^*\defeq \{Lv^*+f=Mv^*-v^*\}\cap\grid_{<0}$ denotes the `border' of the intervention region $\{Lv^*+f\leq Mv^*-v^*\}\cap\grid_{<0}$ defined by $v^*$.
\end{notation}
  
\begin{lemma}
\label{convergence_strategies}
Assume \emph{\ref{A0}-\ref{A2}} and suppose $v^k\to v^*$. Then:
\begin{enumerate}[label=(\roman*)]
\item $\psi^k\to\psi^*$ in $(\partial I^*)^c$ and $Mv^k\to Mv^*$.
\item If $\overline\psi,\overline\delta$ are any two limit points of $(\psi^k),(\delta^k)$ resp.,\footnote{By `limit point' we mean the limit of a convergent subsequence.} then 
$$
\overline\delta\in\argmax_{\delta\in Z}\big\{B(\delta)v^*-c(\delta)\big\},\quad\overline\psi=0\mbox{ on }\grid_{>=0}\quad\mbox{and}\quad\overline\psi\in\argmax_{i\in\{0,1\}}\big\{O_iv^*\big\}\mbox{ on }\grid_{<0},
$$
with $O_0v=Lv+f$ and $O_1v=Mv-v$.
\item If $v^*$ has the UIP, then $\delta^k\to\delta^*(v^*)$ and $Hv^k\to Hv^*$.
\end{enumerate}
\end{lemma}

\begin{proof}
That $Mv^k\to Mv^*$ is clear by continuity of the operators $B(\delta)$ and finiteness of $Z$. 

Let $x\in(\partial I^*)^c$ and suppose $Lv^*(x)+f(x)<Mv^*(x)-v^*(x)$ (the other case being analogue). By continuity of $L$ and $M$ there must exist some $k_0$ such that $Lv^k(x)+f(x)<Mv^k(x)-v^k(x)$ for all $k\geq k_0$, which implies $\psi^k(x)=1=\psi^*(x)$ for $k\geq k_0$.

The statement about $\overline\psi,\overline\delta$ is proved as before by considering appropriate subsequences. Consequently, if $v^*$ has the UIP, then necessarily $\delta^k\to\delta^*(v^*)$ and $Hv^k\to Hv^*$.
\end{proof}

As a corollary we can establish that, should the sequence $(v^k)$ converge, its limit must solve problem (\ref{dQVIs}). If convergence is not exact however (i.e., in finite iterations), then we will ask that $v^*$ verifies some of the properties of the Verification Theorem in Corollary \ref{coro_sym_QVIs}. Namely, the UIP and a discrete analogue of the continuity in the border of the opponent's intervention region. 
We emphasize that our main motivation in solving system (\ref{dQVIs}) relies in Corollary \ref{coro_sym_QVIs} and its framework. Additionally, in most practical situations and for fine-enough grids, one can intuitively expect the discretization of an equilibrium payoff as in Corollary \ref{coro_sym_QVIs} to inherit the UIP. Lastly, we note that the exact equality $Lv^*+f=Mv^*-v^*$ will typically not be verified for any point in the grid in practice, giving $\partial I^*=\emptyset$.

\begin{corollary}
\label{convergence_to_solution}
Assume \emph{\ref{A0}--\ref{A2}} and suppose $v^k\to v^*$. Then:
\begin{enumerate}[label=(\roman*)]
\item If the convergence is exact, then $v^*$ solves the system of QVIs (\ref{dQVIs}).
\item If $v^*$ has the UIP and $Lv^*+f=Hv^*-v^*$ on $-\partial I^*$, then $v^*$ solves (\ref{dQVIs}).
\end{enumerate}
\end{corollary}

\begin{proof}
\textit{(i)} is immediate from the definition of Algorithm \ref{sym_algo}.

In the general case, since $\{0,1\}^\grid$ is finite, there is a subsequence of $\big(\psi^k,\psi^{k+1}\big)$ that converges to some pair $(\psi,\overline\psi)$. Passing to such subsequence, by Lemma \ref{convergence_strategies}, the UIP of $v^*$ and equation (\ref{FPPI-like}), we get that $v^*$ solves the system $\mathbb A\big(\varphi,\overline\varphi\big)v^*=\mathbb B\big(\varphi\big)v^* +\mathbb C\big(\varphi,\overline\varphi\big)$ for $\varphi=\big(\psi,\delta^*(v^*)\big),\overline\varphi=\big(\overline\psi,\delta^*(v^*)\big)$ and $\psi,\overline\psi$ coincide with $\psi^*$ except possibly on $\partial I^*$. Thus, it only remains to show that $v^*$ also solves the equations of the system (\ref{dQVIs}) for any $x\in\partial I^*\cup(-\partial I^*)$. 

For $x\in\partial I^*$, the previous is true by definition. Suppose now $x\in-\partial I^*\subseteq {\overline I}^c$. We have $\psi^*(-x)=1$. If $\psi(-x)=1$, there is nothing to prove. If $\psi(-x)=0$, then $x\in{\overline I}^c\cap(-I)^c$ and $0=Lv^*(x)+f(x)=Hv^*(x)-v^*(x)$, 
where the last equality holds true by assumption.
\end{proof}

Lemma \ref{convergence_strategies} shows to what extent the convergence of the payoffs imply the convergence of the strategies. The following theorem, of theoretical interest, establishes a reciprocal under the stronger assumption \ref{A0'}. In general, since the set of strategies $\Phi$ is finite, the sequence of strategy-dependent coefficients of the fixed-point equations (\ref{FPPI-like}) will always be bounded and with finitely many limit points. However, if the approximating strategies are such that the former coefficients convergence, then Algorithm \ref{sym_algo} is guaranteed to converge. Further, instead of looking at the convergence of $\big(\mathbb A,\mathbb B,\mathbb C\big)\big(\varphi^k,\varphi^{k+1}\big)$, we can instead consider the weaker condition of $\big(\mathbb A^{-1}\mathbb B,\mathbb A^{-1}\mathbb C\big)\big(\varphi^k,\varphi^{k+1}\big)$ converging.

\begin{theorem}
\label{convergence_payoffs}
Assume \emph{\ref{A0'},\ref{A1},\ref{A2}}. If $\big(\mathbb A^{-1}\mathbb B\big(\varphi^k,\varphi^{k+1}\big)\big)$ and $\big(\mathbb A^{-1}\mathbb C\big(\varphi^k,\varphi^{k+1}\big) \big)$ converge, then $(v^k)$ converges.
\end{theorem}
\begin{proof}
Set $b=\lim_k \mathbb A^{-1}\mathbb C\big(\varphi^k,\varphi^{k+1}\big)$. Since $\Phi$ is finite, there must exist $k_0\in\mathbb N$ and $\varphi,\overline\varphi\in\Phi$ such that $\mathbb A^{-1}\mathbb B\big(\varphi^k,\varphi^{k+1}\big)=\mathbb A^{-1}\mathbb B\big(\varphi,\overline\varphi\big)$ and $\mathbb A^{-1}\mathbb C\big(\varphi^k,\varphi^{k+1}\big)=b$ for all $k\geq k_0$. Moreover, under our assumptions, $(\mathbb A-\mathbb B)\big(\varphi,\overline\varphi\big)$ is a WCDD $\mbox{L}_0$-matrix. Then Lemma \ref{coefficients_properties} and Theorem \ref{contraction} imply that $\mathbb A^{-1}\mathbb B\big(\varphi,\overline\varphi\big)$ is contractive for some matrix norm. Lastly, note that the sequence of payoffs $(v^k)_{k\geq k_0}$ now satisfies the classical (constant-coefficients) contractive fixed-point recurrence $v^{k+1}= \mathbb A^{-1}\mathbb B\big(\varphi,\overline\varphi\big)v^k + b$, which converges to the unique fixed-point of the equation.
\end{proof}

The classical fixed-point policy-iteration framework \cite{HFL1,Cl0} assumes uniform contractiveness in $\|\cdot\|_\infty$ of the sequence of operators. This is a natural norm to consider in a context where matrices have properties defined row by row, such as diagonal dominance.\footnote{Recall that this norm can be computed as the maximum absolute value row sum.} However, the authors mention convergence in experiments where only $\|\cdot\|_\infty$-non-expansiveness held true. The latter is the typical case in our context, for the matrices $\mathbb A^{-1}\mathbb B\big(\varphi^k,\varphi^{k+1}\big)$, which is why Theorem \ref{convergence_payoffs} relies on the fact that a spectral radius strictly smaller than one guarantees contractiveness in some matrix norm. 

It is natural to ask whether there is some contractiveness condition that may account for the observations in \cite{HFL1,Cl0} and that can be generalized to our context to further the study of Algorithm \ref{sym_algo}. Imposing a uniform bound on the spectral radii would not only be hard to check, but also difficult to manipulate, as the spectral radius is not sub-multiplicative.\footnote{$\rho(AB)\leq \rho(A)\rho(B)$ does not hold in general when the matrices $A$ and $B$ do not commute.} Instead, we can consider the sequential indexes of contraction and connectivity, which naturally generalize those of the previous section by means of walks in the graph of a sequence of matrices (see Appendix \ref{appendix:matrices} for more details). As before, they can be identified with one another (see Lemma \ref{substochastic_WDD_link}) and, given substochastic matrices, the sequential index of contraction tells us how many we need to multiply before the result becomes $\|\cdot\|_\infty$-contractive (Theorem \ref{sequence_contraction}). Thus, let us consider a uniform bound on the following sequential indexes of connectivity: 

\begin{enumerate}[label=(A\arabic*''),start=0]
\item \label{A0''} 
There exists $m\in\mathbb N_0$ such that for any sequence of strategies $(\overline\varphi^k)\subseteq\Phi$, 
$$
\mbox{con}\left[\Big(\mathbb A-\mathbb B)\big(\overline\varphi^k,\overline\varphi^{k+1}\big)\Big)_k\right]\leq m.
$$
\end{enumerate} 

\begin{remark}
Given $\varphi,\overline\varphi\in\Phi$, by considering the sequence $\varphi,\overline\varphi,\varphi,\overline\varphi,\dots$, we see that \ref{A0''} implies \ref{A0'}. In fact, \ref{A0''} can be interpreted as precluding infinite simultaneous impulses even when the players can adapt their strategies (cf. Remark \ref{r:interpretation2}) and imposing that the number of shifts needed for any state to reach the common continuation region is bounded.
\end{remark}

Under this stronger assumption, we have:
\begin{proposition}
\label{bounded_iterates}
Assume \emph{\ref{A0''},\ref{A1},\ref{A2}}. Then $(v^k)$ is bounded.
\end{proposition}

\begin{proof}
In a similar way to Lemma \ref{coefficients_properties}, one can check that under \ref{A0''},\ref{A1},\ref{A2} we have the following uniform bound for the sequential indexes of contraction:
$$
\widehat{\mbox{con}}\left[\Big(\mathbb A^{-1}\mathbb B)\big(\varphi^k,\varphi^{k+1}\big)\Big)_k\right]\leq m,
$$
for any sequence of strategies $(\varphi^k)\subseteq\Phi$. In other words, multiplying any $m+1$ of the previous substochastic matrices results in a $\|.\|_{\infty}$-contractive one. Furthermore, since $\Phi^{m+1}$ is finite, there must a be uniform uniform contraction constant $C_1<1$. Let $C_2>0$ be a uniform bound for $\mathbb A^{-1}\mathbb C$. By the representation in Corollary \ref{sequence_representation},
\begin{equation*}
\|v^{k+1}\|_\infty\leq \|v^0\|_\infty + C_2\sum_{n=0}^k C_1^{\big[\frac{k-n}{m+1}\big]}
\leq\|v^0\|_\infty + (m+1)C_2\sum_{n=0}^\infty C_1^n<+\infty.\footnote{For any $x\in\mathbb R$, $[x]$ denotes its integer part.}
\end{equation*}
\end{proof}

Given $n_0\in\mathbb N$ and $k>(m+1)n_0$, the same argument of the previous proof shows that one can decompose $(v^k)$ as 
$
v^{k+1}=u^k + F(\varphi^{k-(m+1)n_0},\dots,\varphi^k) + w^k,
$
 for a fixed function $F$, $\|u^k\|_\infty\leq C_1^{[k/(m+1)]}\|v^0\|\to 0$ and $\|w^k\|_\infty\leq (m+1)C_2\sum_{n=n_0}^\infty C_1^n$. The latter is small if $n_0$ is large. Hence, one could heuristically expect that the trailing strategies are often the ones dominating the convergence of the algorithm. In fact, in all the experiments carried out with a discretization satisfying \ref{A0''},\ref{A1},\ref{A2}, a dichotomous behaviour was observed: the algorithm either converged or at some point reached a cycle between a few payoffs. In the latter case, and restricting attention to instances in which one heuristically expects a solution to exist (more details in Section \ref{s:numerics_games}), it was possible to reduce the residual to the QVIs and the distance between the iterates by refining the grid. 

The previous motivates the study of Algorithm \ref{sym_algo} when the grid is sequentially refined, instead of fixed. Such an analysis however, would likely entail the need of a viscosity solutions framework as in \cite{ABL,BS}, which does not currently exist in the literature of nonzero-sum stochastic impulse games. Consequently, this analysis and the stronger convergence results that may come out of it are inevitably outside the scope of this thesis. 

\subsection{Discretization schemes}
\label{s:discretization}
Let us conclude this section by showing how one can discretize the symmetric system of QVIs (\ref{sym_QVIs}) to obtain (\ref{dQVIs}) in a way that satisfies the assumptions present throughout the chapter. Recall that we work on a given symmetric grid $\grid:\ x_{-N}=-x_N<\dots<x_{-1}=-x_1<x_0=0<x_1<\dots<x_N$.

Firstly, we want a discretization $L$ of the operator $\mathcal A-\rho Id$ such that $-L$ is an SDD $L_0$-matrix as per \ref{A1}. A standard way to do this is to approximate the first (resp. second) order derivatives with forward and backward (resp. central) differences in such a way that we approximate the ODE $\frac{1}{2}\sigma^2 V''+\mu V' - \rho V + f=0$ with an \textit{upwind} (or \textit{positive-coefficients}) scheme. More precisely, for each $x=x_i\in\grid$ we approximate the first derivative with a forward (resp. backward) difference if its coefficient in the previous equation is nonegative (resp. negative) in $x_i$,
\begin{equation*}
V'(x_i)\approx\frac{V(x_{i+1})-V(x_i)}{x_{i+1}-x_i}\quad\mbox{ if }\mu(x_i)\geq 0\quad\mbox{ and }\quad V'(x_i)\approx\frac{V(x_i)-V(x_{i+1})}{x_i-x_{i+1}}\quad\mbox{ if }\mu(x_i)<0					
\end{equation*}
and the second derivative by
$$
V''(x_i)\approx\frac{V(x_{i+1})-V(x_i)}{(x_{i+1}-x_i)(x_{i+1}-x_{i-1})}-\frac{V(x_i)-V(x_{i-1})}{(x_i-x_{i-1})(x_{i+1}-x_{i-1})}.						
$$
In the case of an equispaced grid with step size $h$, this reduces to 
\begin{equation*}
V'(x)\approx\frac{V\big(x+\sgn(\mu(x))h\big)-V(x)}{\sgn(\mu(x))h}					
\quad\mbox{ and }\quad		
V''(x)\approx\frac{V(x+h)-2V(x)+V(x-h)}{h^2}.\footnote{$\sgn$ denotes the sign function, $\sgn(x)=1$ if $x\geq 0$ and $-1$ otherwise.}						
\end{equation*}
For the previous stencils to be defined in the extreme points of the grid, we consider two additional points $x_{-N-1},x_{N+1}$ and replace $V(x_{-N-1}),V(x_{N+1})$ in the previous formulas by some values resulting from artificial boundary conditions. A common choice is to impose Neumann conditions to solve for $V(x_{-N-1}),V(x_{N+1})$ using the first order differences from before. For example, in the equispaced grid case, 
given $\mbox{LBC},\mbox{RBC}\in\mathbb R$ we solve for $V(x_{-N}-h)$ (resp. $V(x_N+h)$) from the Neumann condition 
$$\mbox{LBC}=V'\left(x_{-N}-h\mathbbm 1_{\big\{\mu(x_{-N})\geq 0\big\}}\right)\quad \mbox{(resp. } \mbox{RBC}=V'\left(x_N+h\mathbbm 1_{\big\{\mu(x_N)< 0\big\}}\right)),$$
yielding $V(x_{-N}-h)\approx V(x_{-N})-\mbox{LBC}h$ (resp. $V(x_N+h)\approx  V(x_N)+ \mbox{RBC}h$). The choice of $\mbox{LBC},\mbox{RBC}$ is problem-specific and intrinsically linked to that of $x_N$, although it does not affect the properties of the discrete operators. See more details in Section \ref{s:numerics_games}.

The described procedure leads to a discretization of the ODE as $Lv+f=0$, with $L$ satisfying the properties we wanted. (The \emph{strict} diagonal dominance is a consequence of $\rho>0$.) Note that the values of $f$ at $x_{-N},x_N$ need to be modified to account for the boundary conditions. 

\begin{remark}
One could increase the overall order of approximation by using central differences as much as possible for the first order derivatives, provided the scheme remains upwind (see \cite{FL, WF} for more details). This is not done here in order to simplify the presentation. 
\end{remark}

We now approximate the impulse constraint sets $\mathcal Z(x)$ ($x\in\mathbb R$) by finite sets $\emptyset\neq Z(x)\subseteq[0,+\infty)$ ($x\in\grid$), such that $Z(x)=\{0\}$ if $x\geq 0$, and define the impulse operators 
$$B(\delta)v(x) = v[\![x+\delta(x)]\!],\mbox{ for }v\in\mathbb R^\grid,\ \delta\in Z,\ x\in\grid,$$
where $v[\![y]\!]$ denotes linear interpolation of $v$ on $y$ using the closest nodes on the grid, and $v[\![y]\!]=v(x_{\pm N})$ if $\pm y>\pm x_{\pm N}$ (i.e., `no extrapolation'). This univocally defines the discrete loss and gain operators $M$ and $H$ as per (\ref{discrete_intervention_operators}), as well as the optimal impulse $\delta^*$ according to (\ref{delta_star}). The set of discrete strategies $\Phi$ is defined as in (\ref{discrete_strategies}). 

This general discretization scheme satisfies assumptions \ref{A0}--\ref{A2} and one can impose some regularity conditions on the sets $\mathcal Z(x)$ and $Z(x)$ such that the solutions of the discrete QVI problems (\ref{constrained_QVI}) converge locally uniformly to the unique viscosity solution of the analytical impulse control problem, as the grid is refined.\footnote{Additional technical conditions include costs bounded away from zero and a comparison principle for the analytical QVI.} See \cite{A0,ABL} for more details.

\begin{example}
\label{ex:arbitrary_impulses}
In the case where $\mathcal Z(x)=[0,+\infty)$ for $x<0$, a natural and most simple choice for $Z(x)$ is $Z(x_i)=\{0,x_{i+1}-x_i,\dots,x_N-x_i\}$ for $i<0$. In this case, $B(\delta)v(x) = v(x+\delta(x))$ and $Hv(x)=v(x-\delta^*(-x))+g(x,-\delta^*(-x))$. This choice, however, does not satisfy \ref{A0'}. 
\end{example}

In order to preclude infinite simultaneous interventions it is enough to constrain the size of the impulses so that the symmetric point of the grid cannot be reached. That is, $Z(x)\subseteq[0,-2x)$ for any $x\in\grid_{<0}$. In this case, the scheme satisfies the stronger conditions \ref{A0''},\ref{A1},\ref{A2} (and in particular, \ref{A0'}). Note that we can take $m=N$ in \ref{A0''}, as each positive impulse will lead to a state which is at least one node closer to $x_0=0$, where no player intervenes. Practically, it makes sense to make this choice when one suspects (or wants to check whether) there is a symmetric NE with no `far-reaching impulses', in the previous sense.

\begin{example}
\label{ex:impulses_constrained_size}
If $\mathcal Z(x)=[0,+\infty)$ for $x<0$, the analogous of Example \ref{ex:arbitrary_impulses} is now $Z(x_i)=\{0,x_{i+1}-x_i,\dots,x_{-i-1}-x_i\}$ for $i<0$. 
\end{example}

\begin{remark}
\label{r:maxargmax2}
Consider Example \ref{ex:impulses_constrained_size} in the context of Remark \ref{r:maxargmax1}. As in Proposition \ref{bounded_iterates} and due to Theorem \ref{sequence_contraction}, the less impulses needed between the two players to reach the common continuation region, the faster that the composition of the fixed-point operators of Algorithm \ref{sym_algo} becomes contractive. Hence, one could intuitively expect that when close enough to the solution, the choice of the maximum arg-maximum in (\ref{delta_star}) improves the performance of Algorithm \ref{sym_algo}. This is another motivation for such choice.  
\end{remark}

\section{Numerical results}
\label{s:numerics_games}
This section presents numerical results obtained on a series of experiments. See the introduction of Chapter \ref{c:2}, Section \ref{s:linear_game} and Section \ref{s:symmetric_games} for the motivation and applications behind some of them. We do not assume additional constraints on the impulses in the analytical problem. 
All the results presented were obtained on equispaced grids with step size $h>0$ (to be specified) and with a discretization scheme as in Section \ref{s:discretization} and Example \ref{ex:impulses_constrained_size}. The extreme points of the grid are displayed on each graph. 

For the games with linear costs and gains of the form $c(x,\delta)=c^0+c^1\delta$ and $g(x,\delta)=g^0+g^1\delta$, with $c^0,c^1,g^0,g^1$ constant, the artificial boundary conditions were taken as LBC $=c^1$ and RBC $=g^1$ for a sufficiently extensive grid. They result from the observation that on a hypothetical symmetric NE of the form $\varphi^*=\big((-\infty, \overline x], \delta^*(x)=y^*-x\big)$, with $\overline x<0,\ \overline x< y^*\in\mathbb R$, the equilibrium payoff verifies $V(x)=V(y^*)-c^0-c^1(y^*-x)$ for $x<\overline x$ and $V(x)=V(-y^*)+g^0+g^1(x+y^*)$ for $x>-\overline x$. For other examples, LBC, RBC and the grid extension were chosen by heuristic guesses and/or trial and error. However, in all the examples presented the error propagation from poorly chosen LBC,RBC was minimal. 

The initial guess was set as $v_0=0$ and its induced strategy in all cases. \textsc{SolveImpulseControl} was chosen as Subroutine \ref{solve_one_player} with $\widetilde {tol}=10^{-15}$ and $\lambda=1$.\footnote{For very fine grids, one should increase the value of $\widetilde {tol}$ to avoid stagnation as per Remark \ref{practical_considerations}, and the same being true for Subroutine \ref{solve_one_player}.}  Its convergence was exact however, in all the examples, and faster (in terms of time elapsed and number of operations) when it was compared with Howard's policy iteration (not reported). Instead of fixing a terminal tolerance $tol$ beforehand, we display the highest accuracy that was attained in each case and the number of iterations needed for it.

Section \ref{s:games_on_fixed_grid} considers a fixed grid and games where the results point to the existence of a symmetric NE as per Corollary \ref{coro_sym_QVIs}. Not having an analytical solution to compare with, results are assessed by means of the percentage difference between the iterates 
$$\mbox{Diff}\defeq\left\|(v^{k+1}-v^k)/\max\{|v^{k+1}|,scale\}\right\|_{\infty},$$ 
with $scale=1$ as in \cite{AF}, and the maximum pointwise residual to the system of QVIs (\ref{dQVIs}), defined for $v\in\mathbb R^\grid$ by setting $I=\{Lv+f\leq Mv-v\}\cap\grid_{<0}$, $C=I^c$ and
$$
\mbox{maxResQVIs}(v)\defeq\left\|  \max\{Lv+f,Mv-v\}\mathbbm 1_{-C} + (Hv-v)\mathbbm 1_{-I} \right\|_{\infty}.
$$
Section \ref{cv_analytical_sol} considers the only symmetric games in the literature with (semi-) analytical solution: the central bank linear game of Section \ref{s:linear_game} and the cash management game \cite{BCG}, and computes the errors made by discrete approximations, showing in particular the effect of refining the grid. Not considered here is the strategic pollution control game \cite{FK}, due to its inherent non-symmetric nature. Section \ref{s:games_without_NE} comments on results obtained for games without NEs. Finally, Section \ref{s:beyond_verif_theo} shows results that go beyond the scope of the currently available theory for impulse games. 

\subsection{Convergence to discrete solution on a fixed grid}
\label{s:games_on_fixed_grid}

Throughout this section the grid step size is fixed as $h=0.01$, unless otherwise stated (although results where corroborated by further refinements). Each figure specifies the structure, $\mathcal G=(\mu,\sigma,\rho,f,c,g)$, of the symmetric game solved and shows the numerical solutions at the terminal iteration for the equilibrium payoff, $v^k$, and NE. Graphs plot payoff versus state of the process. The intervention region is displayed in red over the graph of the payoff for presentation purposes.

As a general rule, we focus on games with higher costs $c$ than gains $g$, as the opposite typically leads to players attempting to apply infinite simultaneous impulses \cite{ABCCV} (i.e., inducing a gain from the opponent's intervention is `cheap') leading to degenerate games. The following games resulted in exact convergence in finite iterations, which guarantees a solution of (\ref{dQVIs}) was reached (Corollary \ref{convergence_to_solution}), although very small acceptable errors were reached much sooner.
\begin{figure}[H]
\hspace*{.2cm}
	\includegraphics[scale=.36]{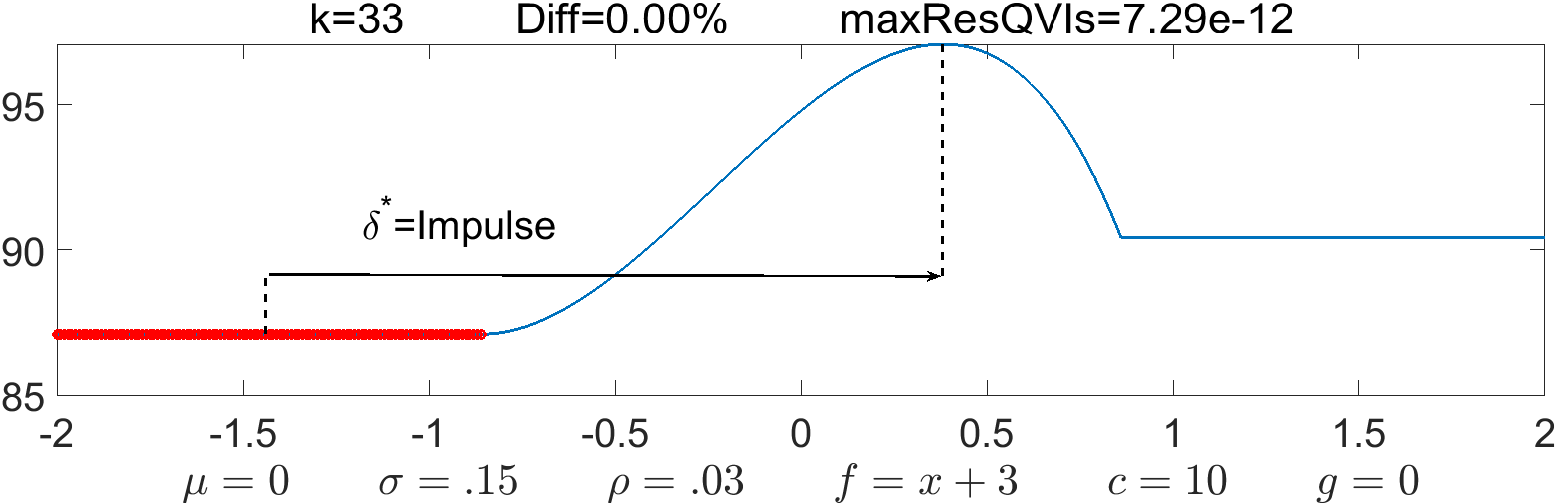}
	\caption[Nash equilibrium and payoff: linear game]{}
\end{figure}
\begin{figure}[H]
\hspace*{.2cm}
	\includegraphics[scale=.36]{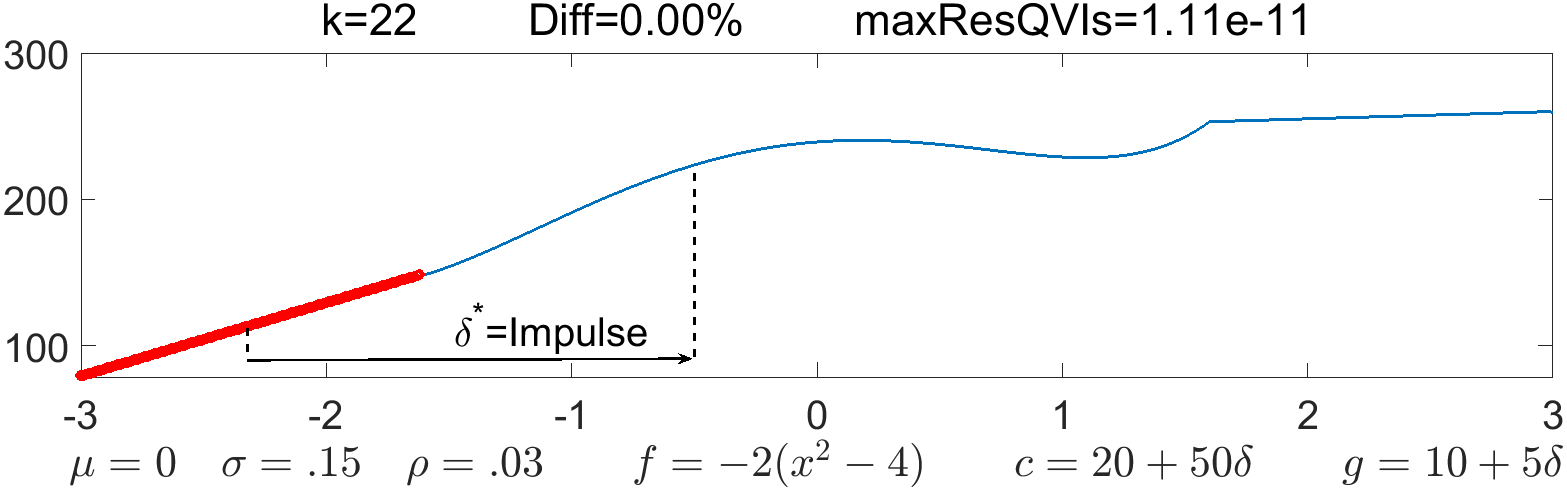}
	\caption[Nash equilibrium and payoff: quadratic running payoff]{}
\end{figure}
\begin{figure}[H]
\hspace*{.2cm}
	\includegraphics[scale=.36]{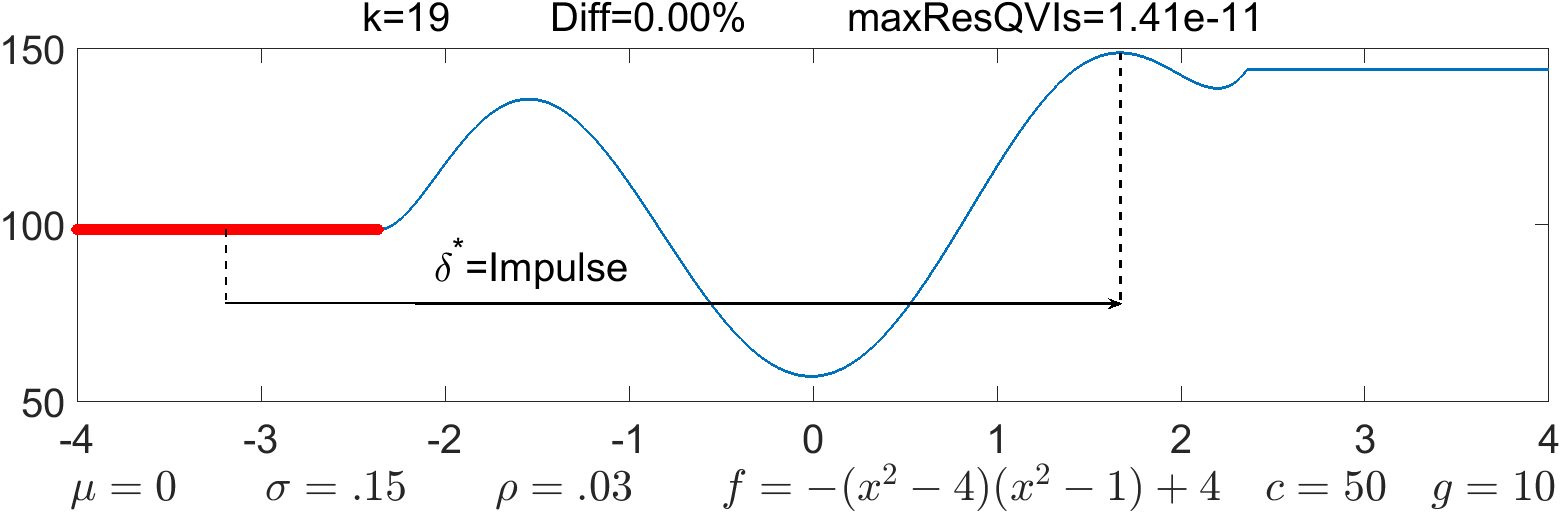}
	\caption[Nash equilibrium and payoff: degree four running payoff]{}
\end{figure}
The following is an example in which the accuracy stagnates. At that point, the iterates start going back and forth between a few values. Although we cannot guarantee that we are close to a solution of (\ref{dQVIs}), the results seem quite convincing, with both Diff and maxResQVIs reasonably low. In fact, simply halving the step size to $h=0.005$ produces a substantial improvement of Diff=9.16e-11\% and maxResQVIs=9.19e-11 in $k=$33 iterations.
\begin{figure}[H]
\hspace*{.3cm}
	\includegraphics[width=.9\linewidth]{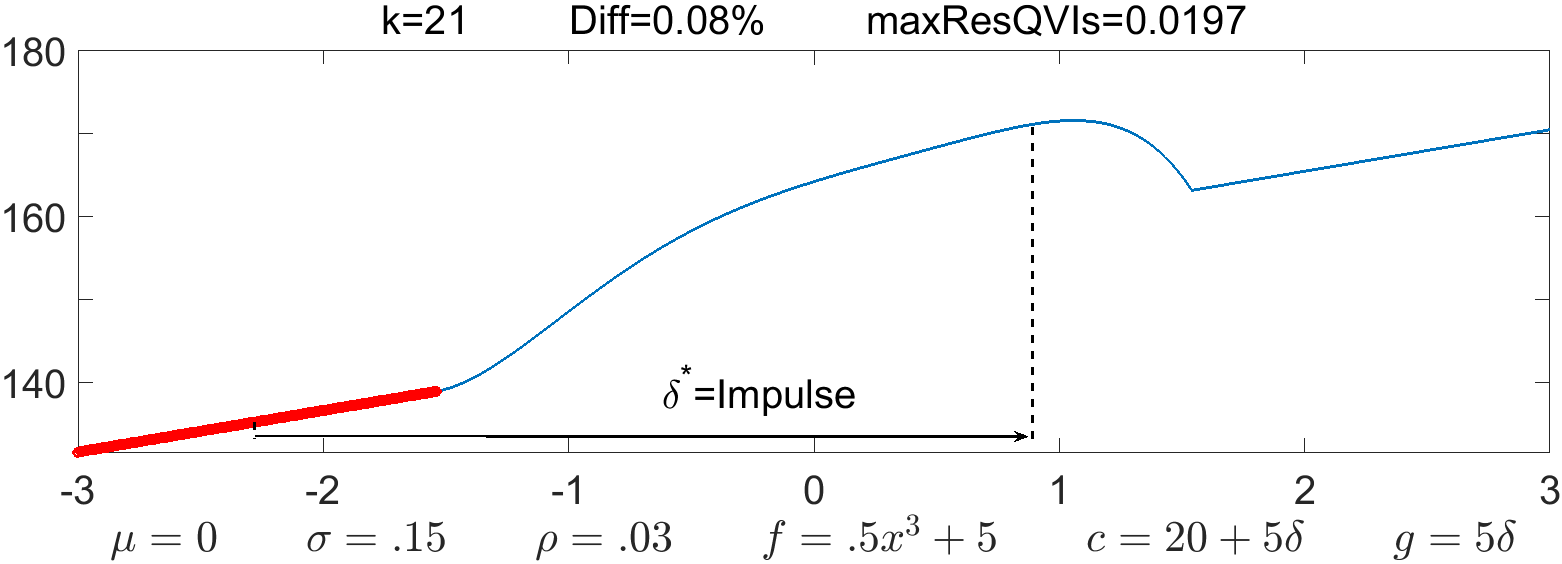}
	\caption[Nash equilibrium and payoff: cubic running payoff]{}
\end{figure}
The previous games have a state variable evolving as a scaled Brownian motion. We now move on to a mean reverting OU process with zero long term mean. (Recall that any other value can be handled simply by shifting the game.) In general, the experiments with these dynamics converged exactly in a very small number of iterations.
\begin{figure}[H]
\hspace*{.2cm}
	\includegraphics[scale=.36]{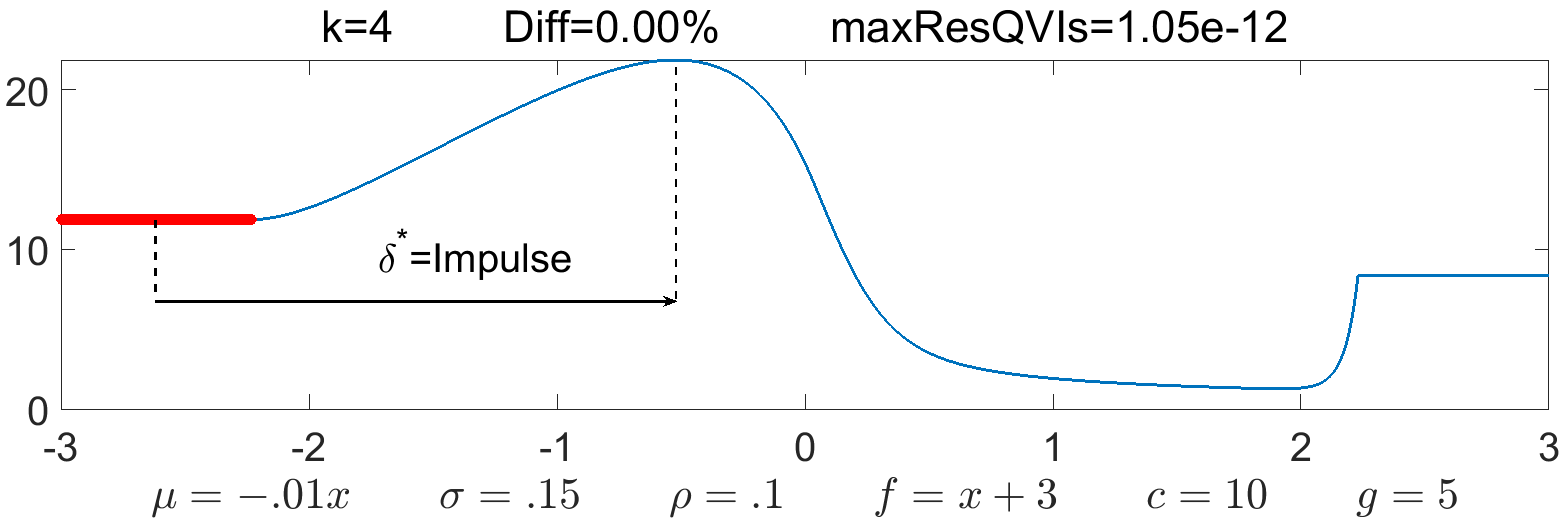}
	\caption[Nash equilibrium and payoff: Ornstein--Uhlenbeck dynamics]{}
\end{figure}
\begin{figure}[H]
\hspace*{.2cm}
	\includegraphics[scale=.36]{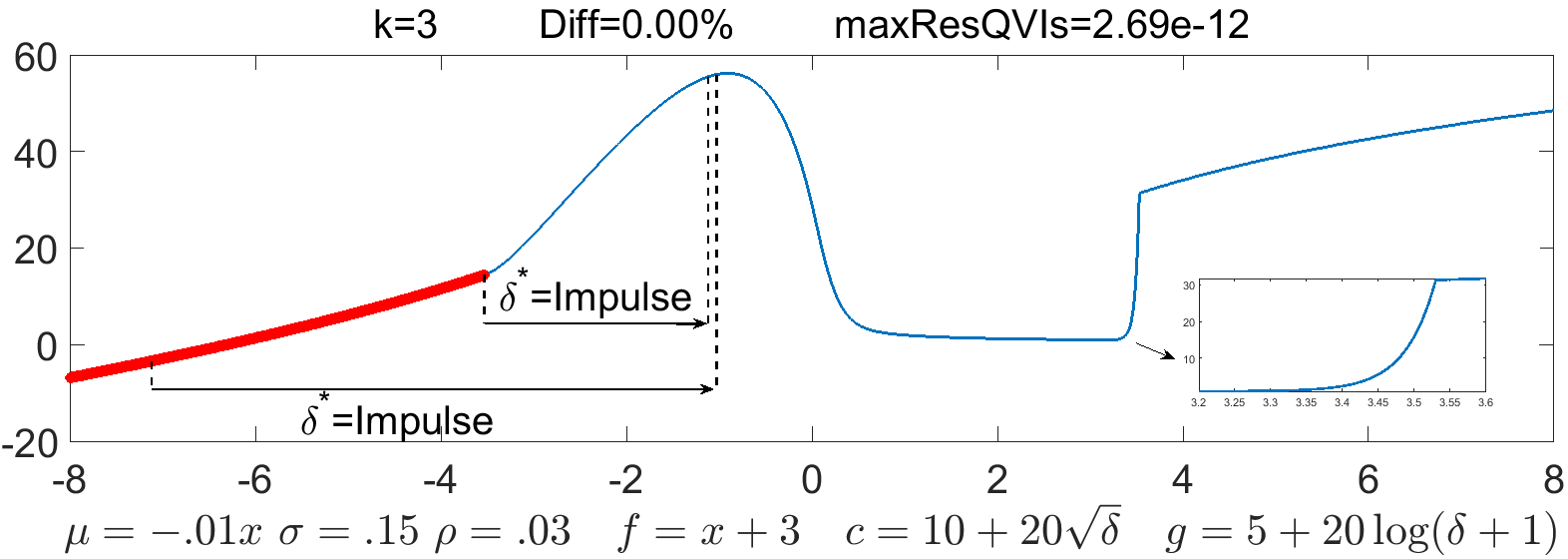}
		\caption[Nash equilibrium and payoff: Ornstein--Uhlenbeck dynamics with concave costs and gains]{}
\end{figure}
Note how all the games treated in this section exhibit a typical feature known to hold for simpler symmetric games (see Chapter \ref{c:2} and \cite{ABCCV,BCG}): the equilibrium payoff of the player is only $C^1$ at the border of the intervention region $\partial\mathcal I^*=\{\mathcal AV-\rho V+f = \mathcal MV-V \}\cap\mathbb R_{<0}$, and only continuous at the border of the opponent's intervention region $-\partial \mathcal I^*$. In floating point arithmetic, the former makes the discrete approximation of $\partial\mathcal I^*$ particularly elusive, while the latter can lead to high errors when close to $-\partial\mathcal I^*$. As a consequence, Subroutine \ref{solve_one_player} (or any equivalent) will often misplace a few grid nodes between intervention and continuation regions, which will in turn make the residual resQVIs `spike' on the opponent's side. Thus, a large value of maxResQVIs can at times be misleading, and further inspection of the pointwise residuals is advisable. 

As a matter of fact, halving the step size to $h=0.005$ in the last example results in exact convergence with terminal maxResQVIs$=2510$, but the residuals on all grid nodes other than the `border' of the opponent's intervention region and a contiguous one are less than 1.43e-11. This is an extreme example propitiated by the almost vertical shape of the solution close to such border. Thus, while it is useful in practice to consider a stopping criteria for Algorithm \ref{sym_algo} based on maxResQVIs, this phenomenon needs to be minded. 

The last example also shows how impulses at a NE can lead to different endpoints depending on the state of the process. This is often the case when costs are nonlinear. In fact:

\begin{lemma}
\label{nonunique_endpoint}
Let $(\mathcal I^*,\delta^*)$ and $V$ be a symmetric NE and equilibrium payoff as in Corollary \ref{coro_sym_QVIs}. Assume that $\mathcal Z(x)=[0,+\infty)$ for $x<0$ and $c=c(x,\delta)\in C^2\big(\mathbb R\times (0,+\infty)\big)$, and consider the re-parametrization $\overline c=\overline c(x,y) \defeq c(x,y-x)$. Suppose that $y^*\defeq x+\delta^*(x)$ is constant for all $x\in{(\mathcal I^*)}^\circ$.\footnote{$A^\circ$ denotes the interior of the set $A\subseteq \mathbb R$.} Then $\overline c_{xy}(x,y^*)\equiv 0$ on ${(\mathcal I^*)}^\circ$. 
\end{lemma}

The result is immediate since $\mathcal MV(x)=\sup_{y>x}\{V(y)-\overline c(x,y)\} = V(y^*)-\overline c(x,y^*)$ on $(\mathcal I^*)^\circ$ and $V\in C^1\big(-\mathcal C^*)$. ($y^*\notin -\mathcal I^*$ or there would be infinite simultaneous impulses.) While the sufficient condition $\overline c_{xy}(x,y^*)\equiv 0$ for some $y^*$ is verified for linear costs, it is not in general and certainly not for $c(x,\delta)=10+20\sqrt\delta$ as in the last example.

\subsection{Convergence to analytical solution with refining grids}
\label{cv_analytical_sol}
A convergence analysis from discrete to analytical solution with refining grids is outside the scope of this thesis, and seems to be far too challenging when a viscosity solutions framework is yet to be developed. Instead, we present here a numerical validation using the solutions of the linear and cash management games. We focus first and foremost in the former, as it has an almost fully analytical solution, with only one parameter to be found numerically as opposed to four for the latter. The structure of the game is defined with parameter values used in \cite{ABCCV} (also in Chapter \ref{c:2}). To minimize rounding errors from floating point arithmetic, we proceed as in \cite{AF} considering grids made up entirely of machine numbers. The results are displayed in Table \ref{table}. 

For each step size, Algorithm \ref{sym_algo} either converged or was terminated upon stagnation. Regardless, we can see the errors made when approximating the analytical solution are quite satisfactory in all cases, and overall decrease as $h\to 0$. Moreover, we see once again how the `spiking' of the residual can be misleading (the highest value was always at the `border' of the opponent's intervention region). 

\begin{SCtable}[][h!]
	\begin{tabular}{llll}
		h  &   $\%\mbox{error}$  &  Its.  &  maxResQVIs \\ 
		\noalign{\smallskip}\hline\noalign{\smallskip}
		1  &      6.67$\%$       &   17  &\quad    8.88e-16 \\
	 1/2 &      8.33$\%$       &   13  &\quad    5.33e-15 \\
	 1/4 &      0.23$\%$       &    4  &\quad     13.2    \\
	 1/8 &      0.21$\%$       &    8  &\quad     15.1    \\
	1/16 &      0.16$\%$       &    8  &\quad      30     \\
	1/32 &      0.07$\%$       &   21  &\quad     21.2    \\
	1/64 &      0.0043$\%$     &   37  &\quad     0.343   \\
		\noalign{\smallskip}\hline
	\end{tabular}
	\caption[Convergence of symmetric algorithm to analytical solution]{Convergence to analytical solution when refining equispaced symmetric grid with step size $h$ and endpoint $x_N=4$. Game: $\mu=0,\ \sigma=.15,\ \rho=.02,\ f=x+3,\ c=100+15\delta,\ g=15\delta$. \%error $\defeq \|(v-V)/V\|_{\infty}$, with $V$ exact solution and $v$ discrete approximation after Its iterations. $\|\cdot\|_\infty$ is computed over the grid.}
	\label{table} 
\end{SCtable}
An exact symmetric NE for this game (up to five significant figures) is given by the intervention region $\mathcal I^*=(-\infty,-2.8238]$ and impulse function $\delta^*(x)=1.5243-x$, while the approximation given by Algorithm \ref{sym_algo} with $h=1/64$ is $\big((-\infty, -2.8125],1.5313 - x\big)$, with absolute errors on the parameters of no more than the step size.

The cash management game \cite{BCG} with unidirectional impulses can be embedded in our framework by reducing its dimension with the change of variables $x=x_1-x_2$, changing minimization by maximization and relabelling the players. With the parameter values of \cite[Fig.1b]{BCG}, it translates into the game: $\mu=0,\ \sigma=1,\ \rho=.5,\ f=-|x|,\ c=3+\delta,\ g=-1$. The authors found numerically a symmetric NE approximately equal to $\big((-\infty, -5.658], 0.686 - x\big)$, while Algorithm \ref{sym_algo} with $x_N=8$ and $h=1/64$ gives $\big((-\infty, -5.6563], 0.6875 - x\big)$. The absolute difference on the parameters is once again below the grid step size.

\subsection{Games without Nash equilibria}
\label{s:games_without_NE}
It is natural to ask how Algorithm \ref{sym_algo} behaves on games without symmetric NEs, and whether anything can be inferred from the results. For the linear game, two cases without NEs (symmetric or not) are addressed in \cite{ABCCV}: `no fixed cost' and `gain greater than cost'. Both of them yield degenerate `equilibria' where the players perform infinite simultaneous interventions. When tested for several parameters, Algorithm \ref{sym_algo} converged in finitely many iterations (although rather slowly) and yielded the exact same type of `equilibria'.\footnote{There were also cases of stagnation, improved by refining the grid as in Section \ref{s:games_on_fixed_grid}.} For a fine enough grid, the previous can be identified heuristically by some node in the intervention region that would be shifted to its symmetric one (\textit{infinite alternated interventions}), or to its immediate successor over and over again, until reaching the continuation region (\textit{infinite one-sided interventions}).\footnote{More precisely, due to our choice of $Z(x)$ a node can be shifted at most to that immediately preceding its symmetric one.}

In the first case, we recovered the limit `equilibrium payoffs' of \cite[Props.4.10,4.11]{ABCCV}. Intuitively, the players in this game take advantage of free interventions, whether by no cost or perfect compensation, in order to shift the process as desired. Note that when $c=g$, the impulses that maximimize the net payoff are not unique.

In the second case, grid refinements showed the discrete payoffs to diverge towards infinity at every point. This is again consistent with the theory: each player forces the other one to act, producing a positive instantaneous profit. Iterating this procedure infinitely often leads to infinite payoffs for every state.

Tested games in which Algorithm \ref{sym_algo} failed to converge (and not due to stagnation nor a poor choice of the grid extension) were characterized by iterates reaching a cycle, typically with high values of Diff and maxResQVIs regardless of the grid step size. In many cases, the cycles would visit at least one payoff inducing infinite simultaneous impulses. While this might, potentially and heuristically, be indicative of the game admitting no symmetric NE, there is not much more than can be said at this stage.

\subsection{Beyond the Verification Theorem}
\label{s:beyond_verif_theo}
We now present two games in which the solution of the discrete QVIs system (\ref{dQVIs}) found with Algorithm \ref{sym_algo} (by exact convergence) does not comply with the continuity and smoothness assumptions of Corollary \ref{coro_sym_QVIs}. However, the results are sensible enough to heuristically argue they may correspond to NEs beyond the scope of the Verification Theorem. In both cases $h=0.01$. Finer grids yielded the same qualitative results. 

The first game considers costs convex in the impulse. When far enough from her continuation region, it is cheaper for the player to apply several (finitely many) simultaneous impulses instead of one, to reach the state she wishes to (cf. Remark \ref{r:concave_costs}). 
In this game, the optimal impulse $\delta^*$ becomes discontinuous, and its discontinuity points are those in the intervention region where the equilibrium payoff is non-differentiable. These, in turn, translate into discontinuities on the opponent's intervention region (one smoothness degree less, as with the border of the regions).  
\begin{figure}[H]
\hspace*{.2cm}
	\includegraphics[scale=.36]{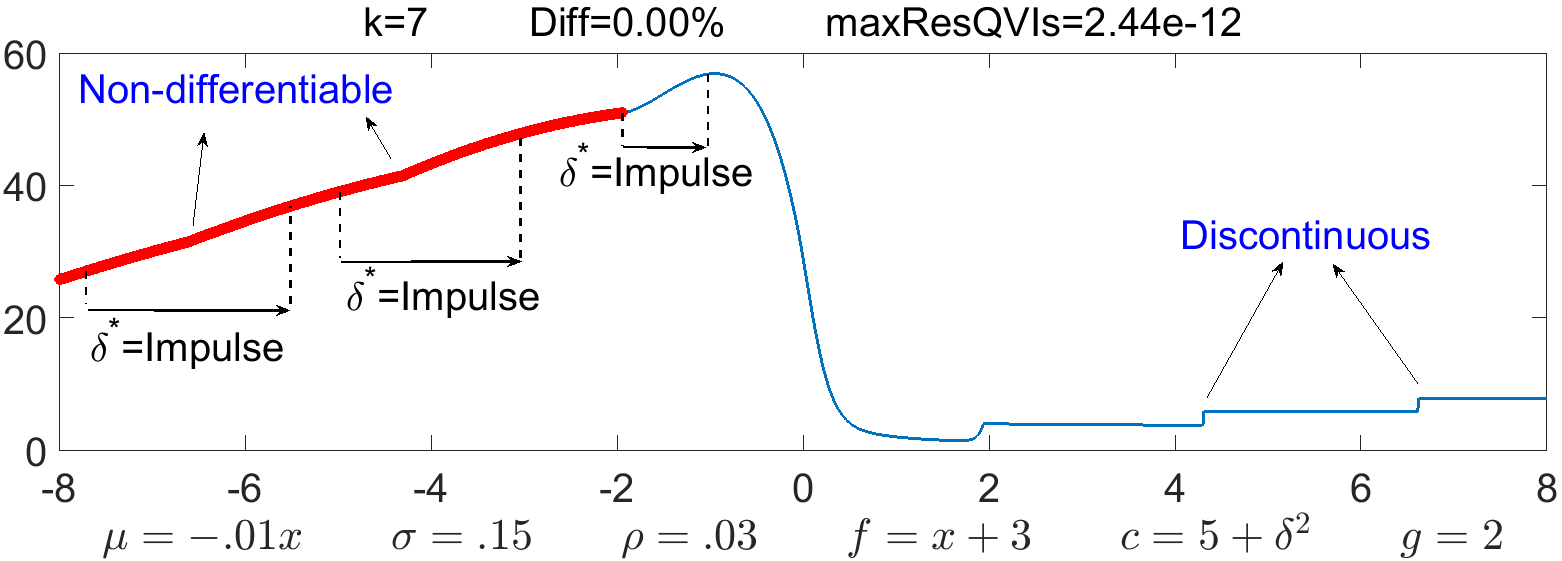}
	\caption[Discontinuous impulses and irregular discontinuous payoffs (first)]{}
\end{figure}

The second game considers linear gains, quadratic running payoffs and costs concave on the impulse. The latter makes the player shift the process towards her continuation region (cf. Remark \ref{r:concave_costs}). However, when far enough from the border, instead of shifting the process directly to her `preferred area', the player chooses to pay a bit more to force her opponent's intervention, inducing a gain and letting the latter pay for the final move. Once again, this causes $\delta^*$ to be discontinuous and leads to a non-differentiable (resp. discontinuous) point for the equilibrium payoff in the intervention (resp. opponent's intervention) region.
\begin{figure}[H]
\hspace*{.2cm}
	\includegraphics[scale=.36]{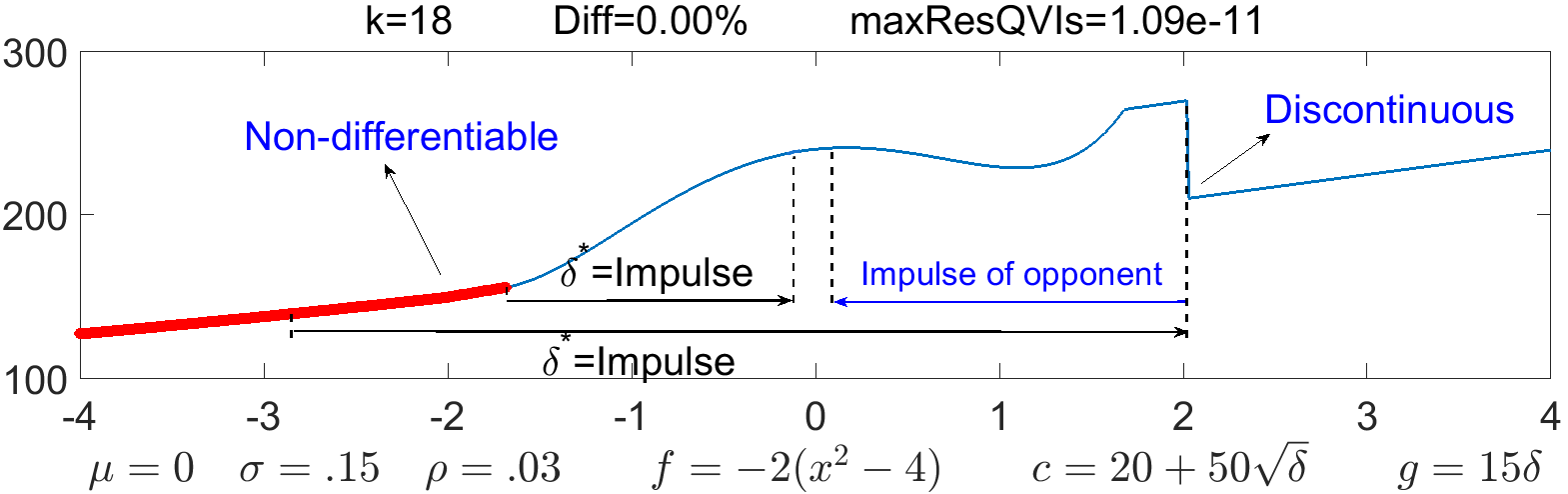}
		\caption[Discontinuous impulses and irregular discontinuous payoffs (second)]{}
\end{figure}

Under the previous reasoning, one could intuitively guess that setting $g=0$ in this game should remove the main incentive the player has to force her opponent to act. This is in fact the case, as shown below. As a result, $\delta^*$ becomes continuous and the equilibrium payoff falls back into the domain of the Verification Theorem. 
\begin{figure}[H]
\hspace*{.2cm}
	\includegraphics[scale=.36]{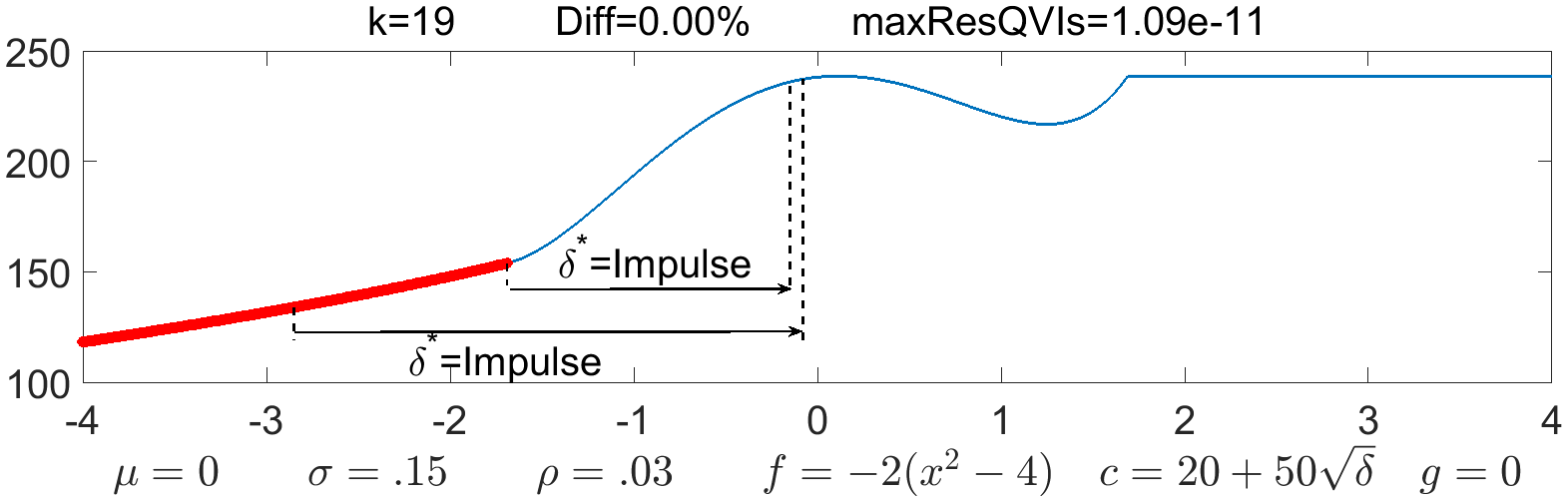}
		\caption[Nash equilibrium and payoff: concave costs, quadratic running payoff and zero gain]{}
\end{figure}
 
We remark that, should the previous solutions correspond indeed to NEs, then the alternative semi-analytical approach of \cite{FK} could not have produced them either, as that method can only yield continuous equilibrium payoffs. 
	

\section{Concluding remarks}
\label{s:conclusions}
This chapter presents a fixed-point policy-iteration-type algorithm to solve systems of quasi-variational inequalities resulting from a Verification Theorem for symmetric nonzero-sum stochastic impulse games. In this context, the method substantially improves the only algorithm available in the literature (Chapter \ref{c:2}) while providing the convergence analysis that we were missing. Graph-theoretic assumptions relating to weakly chained diagonally dominant matrices, which naturally parallel the admissibility of the players strategies, allow to prove properties of contractiveness, boundedness of iterates and convergence to solutions. A result of theoretical interest giving sufficient conditions for convergence is also proved. Additionally, a provably convergent impulse control solver is provided for added efficiency over classical policy iteration. 

Equilibrium payoffs and Nash equilibria of games too challenging for the available analytical approaches are computed with high precision on a discrete setting. Numerical validation with analytical solutions is performed when possible, with reassuring results, but it is noted that grid refinements may be needed at times to overcome stagnation. Thus, formalising the approximating properties of the discrete solutions as well as deriving stronger convergence results for the algorithm may need a viscosity solutions framework currently missing in the theory. This is further substantiated by the irregularity of the solutions, particularly those found which escape the available theoretical results. This motivates further research while providing a tool that can effectively be used to gain insight into these very challenging problems.

\appendix

\chapter[Appendix: Matrix and graph-theoretic definitions and results]{Matrix and graph-theoretic definitions and results}
\chaptermark{Appendix}
\label{appendix:matrices}
\setcounter{section}{1}
\setcounter{subsection}{1}

For the reader's convenience, this appendix summarizes some important algebraic and graph-theoretic definitions and results used primarily throughout Chapter \ref{c:3}. More details can be found in the references given below. Henceforth, $A\in\mathbb R^{N\times N}$ is a real matrix, $Id\in\mathbb R^{N\times N}$ is the identity, $\rho(\cdot)$ denotes the spectral radius and $\mathbb R^{N},\mathbb R^{N\times N}$ are equipped with the elementwise order. We talk about rows and `states' interchangeably.
\begin{definitions}
\label{matrix_definitions}
\begin{enumerate}[label=(D\arabic*)]
\item $A$ is a \textit{Z-matrix} if it has nonpositive off-diagonal elements.
\item $A$ is an \textit{$L$-matrix} (resp. \textit{$L_0$-matrix}) if it is a Z-matrix with positive (resp. nonnegative) diagonal elements.
\item $A$ is an \textit{$M$-matrix} if $A=sId - B$ for some matrix $B\geq 0$ and scalar $s\geq \rho(B)$.
\item $A$ is \textit{monotone} if it is nonsingular and $A^{-1}\geq 0$. Equivalently, $A$ is monotone if $Ax\geq 0$ implies $x\geq 0$ for any $x\in\mathbb R^N$. 
\item The $i$-th row of $A$ is \textit{weakly diagonally dominant (WDD)} (resp. \textit{strictly diagonally dominant or SDD}) if $|A_{ii}|\geq\sum_{j\neq i}|A_{ij}|$ (resp. $>$). 
\item $A$ is \textit{WDD} (resp. \textit{SDD}) if every row of $A$ is WDD (resp. SDD). 
\item The \textit{directed graph of }$A$ is the pair graph$A\defeq(V,E)$, where $V\defeq\{1,\dots,N\}$ is the set of \textit{vertexes} and $E\subseteq V\times V$ is the set of \textit{edges}, such that $(i,j)\in E$ iff $A_{ij}\neq 0$.
\item A \textit{walk} $p$ in graph$A=(V,E)$ from vertex $i$ to vertex $j$ is a nonempty finite sequence $(i,i_1),(i_1,i_2),\dots,(i_{k-1},j)\subseteq E$, which we denote by $i\to i_1\to\dots\to i_{k-1}\to j$. $|p|\defeq k$ is called the \textit{length} of the walk $p$.
\item $A$ is \textit{weakly chained diagonally dominant (WCDD)} if it is WDD and for each WDD row of $A$ there is a walk in graph$A$ to an SDD row (identifying vertexes and rows).
\item $A$ is (right) \textit{substochastic or sub-Markov} (resp. \textit{stochastic or Markov}) if $A\geq 0$ and each row sums at most one (resp. exactly one). Equivalently, $A$ is substochastic if $A\geq 0$ and $\|A\|_\infty\leq 1$. (Recall that $\|A\|_\infty$ is the maximum row-sum of absolute values.)
\item If $A$ is a WDD (resp. substochastic) matrix, its set of `non-trouble states' (or rows) is $J[A]\defeq \{i:\mbox{the }i\mbox{-th row of }A\mbox{ is SDD}\}$ (resp. $\hat J[A]\defeq \{i:\sum_j A_{ij}<1\}$). For each $i$, we write $P_i[A]\defeq\{\mbox{walks in graph}A\mbox{ from }i\mbox{ to some }j\in J[A]\}$ (resp. analogously for $\widehat P_i[A]$). The \textit{index of connectivity} (resp. \textit{index of contraction}) of $A$ \cite{A} is 
$$
\mbox{con}A\defeq\left(\sup_{i\notin J[A]}\left\{ \inf_{p\in P_i[A]}|p|\right\}\right)^+\mbox{ (resp. analogously for $\widehat{\mbox{con}}A$).}\footnote{$(\cdot)^+$ denotes positive part, $\inf\emptyset=+\infty$ and $\sup\emptyset=-\infty$. The index is the least length that needs to be walked on graph$A$ to reach the non-trouble states when starting from an arbitrary trouble one.}
$$
\end{enumerate}
\end{definitions} 

It is clear that SDD$\implies$WCDD$\implies$WDD, and by definition, L-matrix$\implies\mbox{L}_0$-matrix$\implies$Z-matrix. Also by definition, if $A$ is WDD then: $A$ is WCDD $\iff$ con$A<+\infty$.

\begin{proposition}(e.g., \cite[Lem.3.2]{AF})
\label{WCDD_is_nonsing}
Any WCDD matrix is nonsingular.
\end{proposition}

\begin{proposition}(e.g., \cite[Prop.2.15 and 2.17]{A}) \label{M_monotone_equiv}
\label{M-matrix_characterization}

Nonsingular M-matrix $\iff$ monotone L-matrix $\iff$ monotone Z-matrix.
\end{proposition}

\begin{theorem}(e.g., \cite[Thm.2.24]{A})
\label{characterization_theorem}
WCDD $\mbox{L}_0$-matrix $\iff$ WDD nonsingular M-matrix.\footnote{\cite[Thm.2.24]{A} is formulated in terms of L-matrices instead. However, it is trivial to see that: WCDD $\mbox{L}_0$-matrix $\iff$ WCDD L-matrix.}
\end{theorem}

\begin{proposition} (see proof of \cite[Lem.2.22]{A})
\label{substochastic_WDD_link}
$A$ is substochastic if and only if $Id-A$ is a WDD $\mbox{L}_0$-matrix and $A$ has non-negative diagonal elements. In such case, $\hat J[A]=J[Id-A]$, they have the same directed graphs (except possibly for self-loops $i\to i$) and $\widehat{\mbox{con}}A=\mbox{con}[Id-A]$.
\end{proposition}

For the following theorem, recall the characterization of the spectral radius $\rho(A)=\inf\{\|A\|:\|\cdot\|\mbox{ is a matrix norm}\}$ and Gelfand's formula $\rho(A)=\lim_{n\to+\infty}\|A^n\|^{1/n}$, for any matrix norm $\|\cdot\|$. Note also that if $A$ is substochastic, then $A^n$ is also substochastic for any $n\in\mathbb N_0$, $\|A^n\|_{\infty}\leq 1$ and $\rho(A)\leq 1$. 
\begin{theorem}(\cite[Thm.2.5 and Cor.2.6]{A})
\label{contraction}
Suppose $A$ is substochastic. Then 
$$\widehat{\mbox{con}}A=\inf\{n\in\mathbb N_0:\|A^{n+1}\|_\infty<1\}.$$ In particular, $\widehat{\mbox{con}}A<+\infty$ if and only if $\rho(A)<1$.
\end{theorem}

The indexes of contraction and connectivity can be generalized in a natural way to sequences $(A_k)\subseteq\mathbb R^{N\times N}$ by considering walks $i_1\to i_2\to\dots$ such that $i_k\to i_{k+1}$ is an edge in graph$A_k$ (see \cite[App.B]{A} for more details). Theorem \ref{contraction} extends in the following way:

\begin{theorem}(\cite[Thm.B.2]{A})
\label{sequence_contraction}
Suppose $(A_k)$ are substochastic matrices and consider the sequence of products $(B_k)$, where $B_k\defeq A_1\dots A_k$. Then,
$$\widehat{\mbox{con}}\big[(A_k)_k\big]=\inf\{k\in\mathbb N_0:\|B_{k+1}\|_\infty<1\}.$$ 
\end{theorem}

\begin{proof}[\textbf{Proof of Lemma \ref{coefficients_properties}}]
For briefness, we omit the dependence on $\varphi,\overline\varphi$ from the notation.

$\mathbb A^{-1}\mathbb B\geq 0$ since $\mathbb B\geq 0$ and $\mathbb A$ is monotone. To see that its rows sum up to one, let $\boldsymbol{1}\in\mathbb R^\grid$ be the vector of ones. It is easy to check that under \ref{A1}--\ref{A2}, $\mathbb A\mathbb A^{-1}\mathbb B\boldsymbol{1}=\mathbb B\boldsymbol{1}\leq \mathbb A\boldsymbol{1}$, which implies $\mathbb A^{-1}\mathbb B\boldsymbol{1}\leq \boldsymbol{1}$. The WDD $\mbox{L}_0$-property of $\mathbb A-\mathbb B$ is due to \ref{A1}. 

Note that by Proposition \ref{substochastic_WDD_link}, we can consider $\mathbb A^{-1}\mathbb B$ and $Id-\mathbb A^{-1}\mathbb B$ interchangeably. Let us show that $J\big[(\mathbb A-\mathbb B)\big(\varphi,\overline\varphi\big)\big]\subseteq\hat J\big[\mathbb A^{-1}\mathbb B\big(\varphi,\overline\varphi\big)\big]$. 
By monotonicity and $\mbox{L}_0$-property, $\mathbb A^{-1}$ must have strictly positive diagonal elements. Indeed, if there was some index $i$ such that $\mathbb A^{-1}_{ii}=0$, then $1=[\mathbb A\mathbb A^{-1}]_{ii}=\sum_{j\neq i}\mathbb A_{ij}\mathbb A^{-1}_{ji}\leq 0$. Consider now some $x_i\in \hat J\big[\mathbb A^{-1}\mathbb B\big]^c$, i.e., $\sum_j [\mathbb A^{-1}\mathbb B]_{ij}=1$. We want to see that $\sum_j[\mathbb A-\mathbb B]_{ij}=0$.  
We have $0 = \sum_j\big[Id-\mathbb A^{-1}\mathbb B\big]_{ij}=\sum_j\big[\mathbb A^{-1}(\mathbb A-\mathbb B)\big]_{ij} = \sum_k \mathbb A^{-1}_{ik}\sum_{j}\big[\mathbb A-\mathbb B\big]_{kj}$. Since $\mathbb A^{-1}_{ik}\mbox{ and }\sum_{j}\big[\mathbb A-\mathbb B\big]_{kj}$ are non-negative for all $k$, one of the two must be zero for each $k$. But $\mathbb A^{-1}_{ii}>0$, giving what we wanted. 

The final conclusion follows from the previous properties and the fact that $\mathbb A(Id-\mathbb A^{-1}\mathbb B)=\mathbb A-\mathbb B$.
\end{proof}

\clearpage
\phantomsection
\renewcommand*{\bibname}{References}
\addcontentsline{toc}{chapter}{\textbf{References}}
\bibliographystyle{amsalpha}
\bibliography{references}
\end{document}